\documentclass[12pt,reqno]{amsart}
\usepackage[margin=1in]{geometry}
\usepackage{amsmath,amssymb,amsthm,graphicx,amsxtra, setspace}
\usepackage[utf8]{inputenc}
\usepackage{mathrsfs}
\usepackage{hyperref}
\usepackage{upgreek}
\usepackage{mathtools}
\allowdisplaybreaks

\usepackage[pagewise]{lineno}

\newtheorem{theorem}{Theorem}[section]
\newtheorem{lemma}[theorem]{Lemma}
\newtheorem{proposition}[theorem]{Proposition}
\newtheorem{assumption}[theorem]{Assumption}

\newtheorem{definition}[theorem]{Definition}

\newtheorem{remark}[theorem]{Remark}
\newtheorem{hypothesis}[theorem]{Hypothesis}

\let\originalleft\left
\let\originalright\right
\renewcommand{\left}{\mathopen{}\mathclose\bgroup\originalleft}
\renewcommand{\right}{\aftergroup\egroup\originalright}

\newcommand{\Tr}{\mathop{\mathrm{Tr}}}
\renewcommand{\d}{\/\mathrm{d}\/}

\def\w{\textbf{W}^{\varepsilon}_{{\theta}^{\varepsilon}}}

\def\e{\varepsilon}

\def\S{\mathrm{S}}

\def\L{\mathbb{L}}
\def\A{\mathrm{A}}
\def\U{\mathbf{U}}

\def\F{\mathrm{F}}
\def\C{\mathrm{C}}
\def\f{\mathbf{f}}

\def\B{\mathrm{B}}
\def\D{\mathrm{D}}
\def\y{\mathbf{y}}

\def\Y{\mathbf{Y}}
\def\Z{\mathbf{Z}}
\def\E{\mathbb{E}}
\def\X{\mathbf{X}}
\def\x{\mathbf{x}}

\def\z{\mathbf{z}}
\def\v{\mathbf{v}}
\def\V{\mathbb{v}}
\def\w{\mathbf{w}}
\def\W{\mathrm{W}}
\def\G{\mathrm{G}}
\def\Q{\mathrm{Q}}

\def\N{\mathbb{N}}

\def\V{\mathbb{V}}
\def\wi{\widetilde}
\def\Q{\mathrm{Q}}
\def\u{\mathrm{U}}
\def\P{\mathrm{P}}
\def\u{\mathbf{u}}
\def\H{\mathbb{H}}

\newcommand{\R}{\mathbb{R}}

\renewcommand{\d}{\/\mathrm{d}\/}

\newcommand{\Addresses}{{
		\footnote{
			
			\noindent \textsuperscript{1}Department of Mathematics, Indian Institute of Technology Roorkee-IIT Roorkee,
			Haridwar Highway, Roorkee, Uttarakhand 247667, INDIA.\par\nopagebreak
			\noindent  \textit{e-mail:} \texttt{manilfma@iitr.ac.in, maniltmohan@gmail.com.}
			
			\noindent \textsuperscript{*}Corresponding author.

			\textit{Key words:} convective Brinkman-Forchheimer equations; averaging principle; invariant measure; strong
			convergence.
			
			Mathematics Subject Classification (2010): Primary 60H15; Secondary 35R60, 35Q30, 70K70.

}}}

\begin{document}
	
	
	\title[Averaging Principle for the Stochastic CBF equations]{Averaging principle for the stochastic convective Brinkman-Forchheimer equations		\Addresses}
	\author[M. T. Mohan ]{Manil T. Mohan\textsuperscript{1*}}

	\maketitle
	
	\begin{abstract}
The convective Brinkman-Forchheimer equations  describe the motion of incompressible fluid flows in a saturated porous medium.  This work examines the multiscale stochastic convective Brinkman-Forchheimer (SCBF)  equations perturbed by multiplicative Gaussian noise in two and three dimensional bounded domains. We establish a strong averaging principle for the stochastic 2D SCBF equations, which contains a fast time scale component governed by a stochastic reaction-diffusion equation with damping driven by multiplicative Gaussian noise. We exploit the Khasminkii's time discretization approach in the proofs. 
	\end{abstract}

	\section{Introduction}\label{sec1}\setcounter{equation}{0}
The  convective Brinkman-Forchheimer (CBF) equations  describe the motion of  incompressible viscous fluid  through a rigid, homogeneous, isotropic, porous medium.	Let $\mathcal{O}\subset\R^n$ ($n=2,3$) be a bounded domain with a smooth boundary $\partial\mathcal{O}$. Let  $\X_t(x) \in \R^n$ denotes the velocity field at time $t\in[0,T]$ and position $x\in\mathcal{O}$, $p_t(x)\in\R$ represents the pressure field, $\f_t(x)\in\R^n$ stands for an external forcing.   The	CBF  equations  are given by 
\begin{equation}\label{1}
\left\{
\begin{aligned}
\frac{\partial \X_t}{\partial t}-\mu \Delta\X_t+(\X_t\cdot\nabla)\X_t+\alpha\X_t+\beta|\X_t|^{r-1}\X_t+\nabla p_t&=\mathbf{f}_t, \ \text{ in } \ \mathcal{O}\times(0,T), \\ \nabla\cdot\X_t&=0, \ \text{ in } \ \mathcal{O}\times(0,T), \\
\X_t&=\mathbf{0},\ \text{ on } \ \partial\mathcal{O}\times(0,T), \\
\X_0&=\x, \ \text{ in } \ \mathcal{O},
\end{aligned}
\right.
\end{equation}
the constants $\mu>0$ represents the  Brinkman coefficient (effective viscosity),  $\alpha>0$ stands for the Darcy (permeability of porous medium) coefficient and $\beta>0$ denotes the Forchheimer (proportional to the porosity of the material) coefficient.  In order to obtain the uniqueness of the pressure $p$, one can impose the condition $ \int_{\mathcal{O}}p_t(x)\d x=0, $ for $t\in (0,T)$ also. The absorption exponent $r\in[1,\infty)$ and the case $r=3$ is known as the critical exponent.  It should be noted that for $\alpha=\beta=0$, we get the classical 3D Navier-Stokes equations (see \cite{GGP,OAL,JCR3,Te,Te1}, etc). The works \cite{SNA,KWH,KT2,MTM7}, etc discuss the global solvability results (existence and uniqueness of weak as well as strong solutions) of the deterministic CBF equations in bounded domains. As in the case of classical 3D Navier-Stokes equations, the existence of a unique global strong solution for the CBF  equations \eqref{1} for $n=3$ and $r\in[1,3)$ is an open problem.

The existence and uniqueness of strong solutions to the stochastic 3D tamed Navier-Stokes equations on bounded domains with Dirichlet boundary conditions is established in \cite{MRTZ}.  The existence of martingale solutions for the stochastic 3D Navier-Stokes equations with nonlinear damping is established in \cite{LHGH1}.  The existence of a pathwise unique strong solution  satisfying the energy equality (It\^o's formula) to the stochastic convective Brinkman-Forchheimer (SCBF) equations perturbed by multiplicative Gaussian noise is obtained in \cite{MTM8}. The author exploited the monotonicity and hemicontinuity properties of the linear and nonlinear operators as well as a stochastic generalization of  the Minty-Browder technique in the proofs. In order to obtain the It\^o formula (energy equality), the author made use of the fact that there are functions that can approximate functions defined on smooth bounded domains by elements of eigenspaces of linear operators (e.g., the Laplacian or the Stokes operator) in such a way that the approximations are bounded and converge in both Sobolev and Lebesgue spaces simultaneously (such a construction is available in \cite{CLF}). By using the exponential stability of strong solutions, the existence of a unique ergodic and strongly mixing invariant measure for the SCBF  equations subject to multiplicative Gaussian noise is also established in  \cite{MTM8}. The works \cite{HBAM,ZBGD,WLMR,WL,MRXZ1,MRTZ1}, etc discuss various results on the stochastic tamed 3D Navier-Stokes equations and related models on periodic domains as well as on whole space. 

For the past few decades, averaging principle for multiscale systems got its attention and it has wide range  applications in science and engineering (cf. \cite{RBJE,WEBE,EHVK,EAMEL,MMMC,FWTT}, etc and references therein).  An averaging principle for the deterministic systems was first investigated in \cite{NNYA} and for the stochastic differential  system was first studied by Khasminskii  in  \cite{RZK}. Further generalizations of averaging principle for the stochastic differential equations have been carried out in \cite{DGIG,DLi,WLXN,AYV,JXJL,YXBP}, etc. The author in \cite{SC} established an averaging principle for a general class of stochastic reaction-diffusion systems by using  the classical Khasminskii approach.  The validity of an averaging principle for a class of systems of slow-fast reaction-diffusion equations  with the reaction terms in both equations having polynomial growth, perturbed by multiplicative noise is investigated in \cite{SC1}. An averaging principle for a slow-fast system (the almost periodic case) of stochastic reaction-diffusion equations, whose coefficients of the fast equation depend on time is examined in \cite{SC3}. The author in \cite{CEB} showed  an  averaging  principle  for  a  system  of  stochastic  evolution  equations  of  parabolic  type  with slow and fast time scales and derived explicit bounds for the approximation error with respect to the small parameter defining the fast time scale. An averaging principle for stochastic hyperbolic-parabolic equations with slow and fast  time-scales is proved in \cite{HFLW}. The authors in \cite{ZDXS} established an averaging principle for one dimensional stochastic Burgers equation with slow and fast time-scales. An averaging principle  for the slow component as two dimensional stochastic Navier-Stokes equations  and the fast component as stochastic reaction-diffusion equations by using the classical Khasminskii approach based on time discretization is established  in \cite{SLXS}. In \cite{WLMRX}, the authors obtained a strong averaging principle for slow-fast stochastic partial differential equations with locally monotone coefficients, which includes the  systems like stochastic porous medium equation, the stochastic $p$-Laplace equation, the stochastic Burgers type equation and the stochastic 2D Navier-Stokes equation, etc. For more details on averaging principle for different kinds of stochastic systems like reaction-diffusion system, thermoelastic wave propagation system, FitzHugh-Nagumo system, higher order nonlinear Schr\"odinger system, Kuramoto-Sivashinsky system, etc the interested readers are referred to see \cite{JBGY,CEB1,SC2,YCYS,HFJD,HFJL,HFLW2,HFLW1,PGa,PGa1,PGa2,PGa3,WWAJ,WWAJ1,JXu,JXUM}, etc and references therein. 

 In this work, we establish an averaging principle for the following  stochastic coupled  convective Brinkman-Forchheimer equations for $t\in(0,T)$: 
\begin{equation}\label{1p1}
\left\{
\begin{aligned}
\d \X^{\e}_t&=[\mu\Delta\X^{\e}_t-(\X^{\e}_t\cdot\nabla)\X^{\e}_t-\beta|\X^{\e}_t|^{r-1}\X^{\e}_t+f(\X^{\e}_t,\Y^{\e}_t)-\nabla p^{\e}_t]\d t\\&\quad+\sigma_1(\X^{\e}_t)\d\W_t^{\Q_1},\\
\d \Y^{\e}_t&=\frac{1}{\e}[\mu\Delta \Y^{\e}_t-\beta|\Y^{\e}_t|^{r-1}\Y^{\e}_t+g(\X^{\e}_t,\Y^{\e}_t)]\d t+\frac{1}{\sqrt{\e}}\sigma_2(\X^{\e}_t,\Y^{\e}_t)\d\W_t^{\Q_2},\\
\nabla\cdot\X^{\e}_t&=0,\ \nabla\cdot\Y^{\e}_t=0,\\
\X^{\e}_t\big|_{\partial\mathcal{O}}&=\Y^{\e}_t\big|_{\partial\mathcal{O}}=\mathbf{0},\\
\X^{\e}_0&=\x,\ \Y^{\e}_0=\y,
\end{aligned}
\right. 
\end{equation}
where $\e>0$ is a small parameter describing the ratio of the time scales of the slow component $\X^{\e}_t$  and the fast component $\Y^{\e}_t$, $f,g,\sigma_1,\sigma_2$ are appropriate functions, and $\W_t^{\Q_1}$ and $\W_t^{\Q_2}$ are Hilbert space valued Wiener processes on a complete probability space $(\Omega,\mathscr{F},\mathbb{P})$ with filtration $\{\mathscr{F}_t\}_{t\geq 0}$. As $\alpha$ is not playing a major role in our analysis, we fixed $\alpha=0$ in \eqref{1p1}. The system \eqref{1p1} can be considered as stochastic  convective Brinkman-Forchheimer equations, whose drift coefficient is coupled with a stochastic perturbation $\Y^{\e}_t$, which can be considered as the dramatically varying temperature in the system. We establish a strong averaging principle for the system \eqref{1p1} by exploiting the Khasminkii's time discretization approach in the proofs.  The averaging principle also describes the behavior of the solutions of \eqref{1p1}, when $\e\to 0$ and the physical effects at large timescales influence the dynamics of \eqref{1p1}.  Then, we have the following strong averaging principle (for the functional set up and definition of linear and nonlinear operators, see section \ref{sec2}). 

\begin{theorem}\label{maint}
	Let the Assumption \ref{ass3.6} (A1)-(A3) holds true. Then, for any initial values $\x,\y\in\H$, $p\geq 1$ and $T>0$, we have 
	\begin{align}\label{3.126}
	\lim_{\e\to 0} \E\left(\sup_{t\in[0,T]}\|\X^{\e}_t-\bar\X_t\|_{\H}^{2p}\right)=0, 
	\end{align}
	where $(\X_t^{\e},\Y_t^{\e})$ is the unique strong solution of the coupled SCBF system: 
	\begin{equation}\label{3126}
	\left\{
	\begin{aligned}
	\d \X^{\e}_t&=-[\mu\A \X^{\e}_t+\B(\X^{\e}_t)+\beta\mathcal{C}(\X^{\e}_t)-f(\X^{\e}_t,\Y^{\e}_t)]\d t+\sigma_1(\X^{\e}_t)\d\W_t^{\Q_1},\\ 
	\d \Y^{\e}_t&=-\frac{1}{\e}[\mu\A \Y^{\e}_t+\beta\mathcal{C}(\Y_t^{\e})-g(\X^{\e}_t,\Y^{\e}_t)]\d t+\frac{1}{\sqrt{\e}}\sigma_2(\X^{\e}_t,\Y^{\e}_t)\d\W_t^{\Q_2},\\
	\X^{\e}_0&=\x,
	\end{aligned}
	\right. 
	\end{equation}
	and 	$\bar\X_t$ is the unique strong solution of the corresponding averaged system: 
	\begin{equation}\label{3127}
	\left\{
	\begin{aligned}
	\d\bar{\X}_t&=-[\mu\A\bar{\X}_t+\B(\bar{\X}_t)+\beta\mathcal{C}(\bar{\X}_t)]\d t+\bar{f}(\bar{\X}_t)\d t+\sigma_1(\bar{\X}_t)\d\W^{\Q_1}_t, \\ \bar{\X}_0&=\x, 
	\end{aligned}\right. 
	\end{equation}
	with the average 
	$
	\bar{f}(\x)=\int_{\H}f(\x,\y)\nu^{\x}(\d\y), \ \x\in\H, 
	$ and $\nu^{\x}$ is the unique invariant distribution of the transition  semigroup for the frozen system: 
	\begin{equation}\label{3128} 
	\left\{
	\begin{aligned}
	\d\Y_t&=-[\mu\A\Y_t+\beta\mathcal{C}(\Y_t)-g(\x,\Y_t)]\d t+\sigma_2(\x,\Y_t)\d\bar{\W}_t^{\Q_2},\\
	\Y_0&=\y,
	\end{aligned}\right.
	\end{equation}
	where $\bar{\W}_t^{\Q_2}$ is an $\H$-valued $\Q_2$-Wiener process, which is independent of $\W_{t}^{\Q_1}$ and $\W_{t}^{\Q_2}$. 
\end{theorem}
In order to obtain the proof of the above Theorem, we use Khasminskii's time discretization approach (cf. \cite{RZK}). As in the case of 2D Navier-Stokes equations (cf. \cite{SLXS}), for the case $n=2$ and $r\in[1,3]$, we adopted stopping time arguments also to get the proof. But for $n=2,3$ and $r\in(3,\infty)$, ($2\beta\mu\geq 1$ for $n=r=3$), without using stopping time arguments also, we are able to get the proof by making use of the global monotonicity property of the linear and nonlinear operators. Unlike the works available in the literature, we are considering the fast component (second equation in \eqref{1p1}) as the stochastic reaction-diffusion equation with damping (or with the Forchheimer  term). 

The rest of the paper is organized as follows. In the next section, we provide some functional spaces as well as the hypothesis satisfied by the functions $f,g,\sigma_1,\sigma_2$ needed to obtain the global solvability of the system \eqref{1p1}. We discuss about the existence and uniqueness of a pathwise strong solution to the system \eqref{1p1} (Theorem \ref{exis}) in section \ref{sec4}. In the final section, we examine the strong averaging principle for the coupled SCBF equations \eqref{1p1}. We first establish the existence of a unique invariant measure for the frozen system \eqref{3128} (Proposition \ref{prop3.12}). Then, by exploiting the  classical Khasminskii approach based on time discretization, we prove strong convergence of the slow component \eqref{3126} to the solution of the corresponding averaged SCBF system \eqref{3127} (Theorem \ref{maint}).

\section{Mathematical Formulation}\label{sec2}\setcounter{equation}{0}
This section is devoted for the necessary function spaces and the hypothesis satisfied by the functions $f,g$ and the noise coefficients $\sigma_1,\sigma_2$ needed to obtain the global solvability results for the coupled SCBF equations \eqref{1p1}.

\subsection{Function spaces}\label{sub2.1} Let $\C_0^{\infty}(\mathcal{O};\R^n)$ denotes the space of all infinitely differentiable functions  ($\R^n$-valued) with compact support in $\mathcal{O}\subset\R^n$.  We define 
\begin{align*} 
\mathcal{V}&:=\{\X\in\C_0^{\infty}(\mathcal{O},\R^n):\nabla\cdot\X=0\},\\
\mathbb{H}&:=\text{the closure of }\ \mathcal{V} \ \text{ in the Lebesgue space } \L^2(\mathcal{O})=\mathrm{L}^2(\mathcal{O};\R^n),\\
\mathbb{V}&:=\text{the closure of }\ \mathcal{V} \ \text{ in the Sobolev space } \H_0^1(\mathcal{O})=\mathrm{H}_0^1(\mathcal{O};\R^n),\\
\widetilde{\L}^{p}&:=\text{the closure of }\ \mathcal{V} \ \text{ in the Lebesgue space } \L^p(\mathcal{O})=\mathrm{L}^p(\mathcal{O};\R^n),
\end{align*}
for $p\in(2,\infty)$. Then under some smoothness assumptions on the boundary, we characterize the spaces $\H$, $\V$ and $\widetilde{\L}^p$ as 
$
\H=\{\X\in\L^2(\mathcal{O}):\nabla\cdot\X=0,\X\cdot\mathbf{n}\big|_{\partial\mathcal{O}}=0\}$,  with norm  $\|\X\|_{\H}^2:=\int_{\mathcal{O}}|\X(x)|^2\d x,
$
where $\mathbf{n}$ is the outward normal to $\partial\mathcal{O}$,
$
\V=\{\X\in\H_0^1(\mathcal{O}):\nabla\cdot\X=0\},$  with norm $ \|\X\|_{\V}^2:=\int_{\mathcal{O}}|\nabla\X(x)|^2\d x,
$ and $\widetilde{\L}^p=\{\X\in\L^p(\mathcal{O}):\nabla\cdot\X=0, \X\cdot\mathbf{n}\big|_{\partial\mathcal{O}}=0\},$ with norm $\|\X\|_{\widetilde{\L}^p}^p=\int_{\mathcal{O}}|\X(x)|^p\d x$, respectively.
Let $(\cdot,\cdot)$ denotes the inner product in the Hilbert space $\H$ and $\langle \cdot,\cdot\rangle $ denotes the induced duality between the spaces $\V$  and its dual $\V'$ as well as $\widetilde{\L}^p$ and its dual $\widetilde{\L}^{p'}$, where $\frac{1}{p}+\frac{1}{p'}=1$. Note that $\H$ can be identified with its dual $\H'$. We endow the space $\V\cap\widetilde{\L}^{p}$ with the norm $\|\X\|_{\V}+\|\X\|_{\widetilde{\L}^{p}},$ for $\X\in\V\cap\widetilde{\L}^p$ and its dual $\V'+\widetilde{\L}^{p'}$ with the norm $$\inf\left\{\max\left(\|\Y_1\|_{\V'},\|\Y_1\|_{\widetilde{\L}^{p'}}\right):\Y=\Y_1+\Y_2, \ \Y_1\in\V', \ \Y_2\in\widetilde{\L}^{p'}\right\}.$$ Furthermore, we have the continuous embedding $\V\cap\widetilde{\L}^p\hookrightarrow\H\hookrightarrow\V'+\widetilde{\L}^{p'}$. 
\subsection{Linear operator}\label{sub2.2}
Let us define
\begin{equation*}
\left\{
\begin{aligned}
\A\X:&=-\mathrm{P}_{\H}\Delta\X,\;\X\in\D(\A),\\ \D(\A):&=\V\cap\H^{2}(\mathcal{O}).
\end{aligned}
\right.
\end{equation*}
It can be easily seen that the operator $\A$ is a non-negative self-adjoint operator in $\H$ with $\V=\D(\A^{1/2})$ and \begin{align}\label{2.7a}\langle \A\X,\X\rangle =\|\X\|_{\V}^2,\ \textrm{ for all }\ \X\in\V, \ \text{ so that }\ \|\A\X\|_{\V'}\leq \|\X\|_{\V}.\end{align}
For a bounded domain $\mathcal{O}$, the operator $\A$ is invertible and its inverse $\A^{-1}$ is bounded, self-adjoint and compact in $\H$. Thus, using the spectral theorem, the spectrum of $\A$ consists of an infinite sequence $0< \lambda_1\leq \lambda_2\leq\ldots\leq \lambda_k\leq \ldots,$ with $\lambda_k\to\infty$ as $k\to\infty$ of eigenvalues. 
Moreover, there exists an orthogonal basis $\{e_k\}_{k=1}^{\infty} $ of $\H$ consisting of eigenvectors of $\A$ such that $\A e_k =\lambda_ke_k$,  for all $ k\in\mathbb{N}$.  We know that $\X$ can be expressed as $\X=\sum_{k=1}^{\infty}\langle\X,e_k\rangle e_k$ and $\A\X=\sum_{k=1}^{\infty}\lambda_k\langle\X,e_k\rangle e_k$, for all $\X\in\D(\A)$. Thus, it is immediate that 
\begin{align}\label{poin}
\|\nabla\X\|_{\mathbb{H}}^2=\langle \A\X,\X\rangle =\sum_{k=1}^{\infty}\lambda_k|\langle \X,e_k\rangle|^2\geq \lambda_1\sum_{k=1}^{\infty}|\langle\X,e_k\rangle|^2=\lambda_1\|\X\|_{\mathbb{H}}^2,
\end{align}
which is the \emph{Poincar\'e inequality}. 

\subsection{Bilinear operator}
Let us define the \emph{trilinear form} $b(\cdot,\cdot,\cdot):\V\times\V\times\V\to\R$ by $$b(\X,\Y,\Z)=\int_{\mathcal{O}}(\X(x)\cdot\nabla)\Y(x)\cdot\Z(x)\d x=\sum_{i,j=1}^n\int_{\mathcal{O}}\X_i(x)\frac{\partial \Y_j(x)}{\partial x_i}\Z_j(x)\d x.$$ If $\X, \Y$ are such that the linear map $b(\X, \Y, \cdot) $ is continuous on $\V$, the corresponding element of $\V'$ is denoted by $\B(\X, \Y)$. We also denote (with an abuse of notation) $\B(\X) = \B(\X, \X)=\mathrm{P}_{\H}(\X\cdot\nabla)\X$.
An integration by parts yields  
\begin{equation}\label{b0}
\left\{
\begin{aligned}
b(\X,\Y,\Y) &= 0,\text{ for all }\X,\Y \in\V,\\
b(\X,\Y,\Z) &=  -b(\X,\Z,\Y),\text{ for all }\X,\Y,\Z\in \V.
\end{aligned}
\right.\end{equation}
In the trilinear form, an application of H\"older's inequality yields
\begin{align*}
|b(\X,\Y,\Z)|=|b(\X,\Z,\Y)|\leq \|\X\|_{\widetilde{\L}^{r+1}}\|\Y\|_{\widetilde{\L}^{\frac{2(r+1)}{r-1}}}\|\Z\|_{\V},
\end{align*}
for all $\X\in\V\cap\widetilde{\L}^{r+1}$, $\Y\in\V\cap\widetilde{\L}^{\frac{2(r+1)}{r-1}}$ and $\Z\in\V$, so that we get 
\begin{align}\label{2p9}
\|\B(\X,\Y)\|_{\V'}\leq \|\X\|_{\widetilde{\L}^{r+1}}\|\Y\|_{\widetilde{\L}^{\frac{2(r+1)}{r-1}}}.
\end{align}
Hence, the trilinear map $b : \V\times\V\times\V\to \R$ has a unique extension to a bounded trilinear map from $(\V\cap\widetilde{\L}^{r+1})\times(\V\cap\widetilde{\L}^{\frac{2(r+1)}{r-1}})\times\V$ to $\R$. It can also be seen that $\B$ maps $ \V\cap\widetilde{\L}^{r+1}$  into $\V'+\widetilde{\L}^{\frac{r+1}{r}}$ and using interpolation inequality, we get 
\begin{align}\label{212}
\left|\langle \B(\X,\X),\Y\rangle \right|=\left|b(\X,\Y,\X)\right|\leq \|\X\|_{\widetilde{\L}^{r+1}}\|\X\|_{\widetilde{\L}^{\frac{2(r+1)}{r-1}}}\|\Y\|_{\V}\leq\|\X\|_{\widetilde{\L}^{r+1}}^{\frac{r+1}{r-1}}\|\X\|_{\H}^{\frac{r-3}{r-1}}\|\Y\|_{\V},
\end{align}
for all $\Y\in\V\cap\widetilde{\L}^{r+1}$. Thus, we have 
\begin{align}\label{2.9a}
\|\B(\X)\|_{\V'+\widetilde{\L}^{\frac{r+1}{r}}}\leq\|\X\|_{\widetilde{\L}^{r+1}}^{\frac{r+1}{r-1}}\|\X\|_{\H}^{\frac{r-3}{r-1}},
\end{align}
for $r\geq 3$. 

For $n=2$ and $r\in[1,3]$, using H\"older's and Ladyzhenskaya's inequalities, we obtain 
\begin{align*}
|\langle\B(\X,\Y),\Z\rangle|=|\langle\B(\X,\Z),\Y\rangle|\leq\|\X\|_{\wi\L^4}\|\Y\|_{\wi\L^4}\|\Z\|_{\V},
\end{align*}
for all $\X,\Y\in\wi\L^4$ and $\Z\in\V$, so that we get $\|\B(\X,\Y)\|_{\V'}\leq\|\X\|_{\wi\L^4}\|\Y\|_{\wi\L^4}$. Furthermore, we have $$\|\B(\X,\X)\|_{\V'}\leq\|\X\|_{\wi\L^4}^2\leq\sqrt{2}\|\X\|_{\H}\|\X\|_{\V}\leq\sqrt{\frac{2}{\lambda_1}}\|\X\|_{\V}^2,$$ for all $\X\in\V$. 
\subsection{Nonlinear operator}
Let us now consider the operator $\mathcal{C}(\X):=\P_{\H}(|\X|^{r-1}\X)$. It is immediate that $\langle\mathcal{C}(\X),\X\rangle =\|\X\|_{\widetilde{\L}^{r+1}}^{r+1}$. 
For any $r\in[1,\infty)$, we have 
\begin{align}\label{2pp11}
&\langle \mathrm{P}_{\H}(\X|\X|^{r-1})-\mathrm{P}_{\H}(\Y|\Y|^{r-1}),\X-\Y\rangle\nonumber\\&=\int_{\mathcal{O}}\left(\X(x)|\X(x)|^{r-1}-\Y(x)|\Y(x)|^{r-1}\right)\cdot(\X(x)-\Y(x))\d x\nonumber\\&=\int_{\mathcal{O}}\left(|\X(x)|^{r+1}-|\X(x)|^{r-1}\X(x)\cdot\Y(x)-|\Y(x)|^{r-1}\X(x)\cdot\Y(x)+|\Y(x)|^{r+1}\right)\d x\nonumber\\&\geq\int_{\mathcal{O}}\left(|\X(x)|^{r+1}-|\X(x)|^{r}|\Y(x)|-|\Y(x)|^{r}|\X(x)|+|\Y(x)|^{r+1}\right)\d x\nonumber\\&=\int_{\mathcal{O}}\left(|\X(x)|^r-|\Y(x)|^r\right)(|\X(x)|-|\Y(x)|)\d x\geq 0. 
\end{align}
Furthermore, we find 
\begin{align}\label{224}
&\langle\mathrm{P}_{\H}(\X|\X|^{r-1})-\mathrm{P}_{\H}(\Y|\Y|^{r-1}),\X-\Y\rangle\nonumber\\&=\langle|\X|^{r-1},|\X-\Y|^2\rangle+\langle|\Y|^{r-1},|\X-\Y|^2\rangle+\langle\Y|\X|^{r-1}-\X|\Y|^{r-1},\X-\Y\rangle\nonumber\\&=\||\X|^{\frac{r-1}{2}}(\X-\Y)\|_{\H}^2+\||\Y|^{\frac{r-1}{2}}(\X-\Y)\|_{\H}^2\nonumber\\&\quad+\langle\X\cdot\Y,|\X|^{r-1}+|\Y|^{r-1}\rangle-\langle|\X|^2,|\Y|^{r-1}\rangle-\langle|\Y|^2,|\X|^{r-1}\rangle.
\end{align}
But, we know that 
\begin{align*}
&\langle\X\cdot\Y,|\X|^{r-1}+|\Y|^{r-1}\rangle-\langle|\X|^2,|\Y|^{r-1}\rangle-\langle|\Y|^2,|\X|^{r-1}\rangle\nonumber\\&=-\frac{1}{2}\||\X|^{\frac{r-1}{2}}(\X-\Y)\|_{\H}^2-\frac{1}{2}\||\Y|^{\frac{r-1}{2}}(\X-\Y)\|_{\H}^2+\frac{1}{2}\langle\left(|\X|^{r-1}-|\Y|^{r-1}\right),\left(|\X|^2-|\Y|^2\right)\rangle \nonumber\\&\geq -\frac{1}{2}\||\X|^{\frac{r-1}{2}}(\X-\Y)\|_{\H}^2-\frac{1}{2}\||\Y|^{\frac{r-1}{2}}(\X-\Y)\|_{\H}^2.
\end{align*}
From \eqref{224}, we finally have 
\begin{align}\label{2.23}
&\langle\mathrm{P}_{\H}(\X|\X|^{r-1})-\mathrm{P}_{\H}(\Y|\Y|^{r-1}),\X-\Y\rangle\geq \frac{1}{2}\||\X|^{\frac{r-1}{2}}(\X-\Y)\|_{\H}^2+\frac{1}{2}\||\Y|^{\frac{r-1}{2}}(\X-\Y)\|_{\H}^2\geq 0,
\end{align}
for $r\geq 1$.  	It is important to note that 
\begin{align}\label{a215}
\|\X-\Y\|_{\wi\L^{r+1}}^{r+1}&=\int_{\mathcal{O}}|\X(x)-\Y(x)|^{r-1}|\X(x)-\Y(x)|^{2}\d x\nonumber\\&\leq 2^{r-2}\int_{\mathcal{O}}(|\X(x)|^{r-1}+|\Y(x)|^{r-1})|\X(x)-\Y(x)|^{2}\d x\nonumber\\&\leq 2^{r-2}\||\X|^{\frac{r-1}{2}}(\X-\Y)\|_{\L^2}^2+2^{r-2}\||\Y|^{\frac{r-1}{2}}(\X-\Y)\|_{\L^2}^2. 
\end{align}
Combining \eqref{2.23} and \eqref{a215}, we obtain 
\begin{align}\label{214}
\langle\mathcal{C}(\X)-\mathcal{C}(\Y),\X-\Y\rangle\geq\frac{1}{2^{r-1}}\|\X-\Y\|_{\wi\L^{r+1}}^{r+1},
\end{align}
for $r\geq 1$. 
\subsection{Monotonicity}
In this subsection, we discuss about the monotonicity as well as the hemicontinuity properties of the linear and nonlinear operators.
\begin{definition}[\cite{VB}]
	Let $\mathbb{X}$ be a Banach space and let $\mathbb{X}^{'}$ be its topological dual.
	An operator $\G:\mathrm{D}\rightarrow
	\mathbb{X}^{'},$ $\mathrm{D}=\mathrm{D}(\G)\subset \mathbb{X}$ is said to be
	\emph{monotone} if
	$$\langle\G(x)-\G(y),x-y\rangle\geq
	0,\ \text{ for all } \ x,y\in \mathrm{D}.$$ 
	The operator $\G(\cdot)$ is said to be \emph{hemicontinuous}, if for all $x, y\in\mathbb{X}$ and $w\in\mathbb{X}',$ $$\lim_{\lambda\to 0}\langle\G(x+\lambda y),w\rangle=\langle\G(x),w\rangle.$$
	The operator $\G(\cdot)$ is called \emph{demicontinuous}, if for all $x\in\mathrm{D}$ and $y\in\mathbb{X}$, the functional $x \mapsto\langle \G(x), y\rangle$  is continuous, or in other words, $x_k\to x$ in $\mathbb{X}$ implies $\G(x_k)\xrightarrow{w}\G(x)$ in $\mathbb{X}'$. Clearly demicontinuity implies hemicontinuity. 
\end{definition}
\begin{lemma}[Theorem 2.2, \cite{MTM7}]\label{thm2.2}
	Let $\X,\Y\in\V\cap\widetilde{\L}^{r+1}$, for $r>3$. Then,	for the operator $\G(\X)=\mu \A\X+\B(\X)+\beta\mathcal{C}(\X)$, we  have 
	\begin{align}\label{fe}
	\langle(\G(\X)-\G(\Y),\X-\Y\rangle+\eta\|\X-\Y\|_{\H}^2\geq 0,
	\end{align}
	where $\eta=\frac{r-3}{2\mu(r-1)}\left(\frac{2}{\beta\mu (r-1)}\right)^{\frac{2}{r-3}}.$ That is, the operator $\G+\eta\mathrm{I}$ is a monotone operator from $\V\cap\widetilde{\L}^{r+1}$ to $\V'+\widetilde{\L}^{\frac{r+1}{r}}$. 
\end{lemma}

\begin{lemma}[Theorem 2.3, \cite{MTM7}]\label{thm2.3}
	For the critical case $r=3$ with $2\beta\mu \geq 1$, the operator $\G(\cdot):\V\cap\widetilde{\L}^{r+1}\to \V'+\widetilde{\L}^{\frac{r+1}{r}}$ is globally monotone, that is, for all $\X,\Y\in\V$, we have 
	\begin{align}\label{218}\langle\G(\X)-\G(\Y),\X-\Y\rangle\geq 0.\end{align}
\end{lemma}
\begin{lemma}[Remark 2.4, \cite{MTM7}]
	Let $n=2$, $r\in[1,3]$ and $\X,\Y\in\V$. Then,	for the operator $\G(\X)=\mu \A\X+\B(\X)+\beta\mathcal{C}(\X)$, we  have 
	\begin{align}\label{fe2}
	\langle(\G(\X)-\G(\Y),\X-\Y\rangle+ \frac{27}{32\mu ^3}N^4\|\X-\Y\|_{\H}^2\geq 0,
	\end{align}
	for all $\Y\in{\mathbb{B}}_N$, where ${\mathbb{B}}_N$ is an $\widetilde{\L}^4$-ball of radius $N$, that is,
	$
	{\mathbb{B}}_N:=\big\{\z\in\widetilde{\L}^4:\|\z\|_{\widetilde{\L}^4}\leq N\big\}.
	$
\end{lemma}

\begin{lemma}[Lemma 2.5, \cite{MTM7}]\label{lem2.8}
	The operator $\G:\V\cap\widetilde{\L}^{r+1}\to \V'+\widetilde{\L}^{\frac{r+1}{r}}$ is demicontinuous. 
\end{lemma}

\section{Stochastic  Coupled   Convective Brinkman-Forchheimer Equations}\label{sec4}\setcounter{equation}{0}
Let $(\Omega,\mathscr{F},\mathbb{P})$ be a complete probability space equipped with an increasing family of sub-sigma fields $\{\mathscr{F}_t\}_{0\leq t\leq T}$ of $\mathscr{F}$ satisfying:
\begin{enumerate}
	\item [(i)] $\mathscr{F}_0$ contains all elements $F\in\mathscr{F}$ with $\mathbb{P}(F)=0$,
	\item [(ii)] $\mathscr{F}_t=\mathscr{F}_{t+}=\bigcap\limits_{s>t}\mathscr{F}_s,$ for $0\leq t\leq T$.
\end{enumerate}  In this section, we consider the stochastic coupled   convective Brinkman-Forchheimer  equations perturbed by multiplicative Gaussian noise. On  taking orthogonal projection $\mathrm{P}_{\H}$ onto the first two equations in \eqref{1p1}, we obtain  
\begin{equation}\label{3.6}
\left\{
\begin{aligned}
\d \X^{\e}_t&=-[\mu\A \X^{\e}_t+\B(\X^{\e}_t)+\beta\mathcal{C}(\X^{\e}_t)-f(\X^{\e}_t,\Y^{\e}_t)]\d t+\sigma_1(\X^{\e}_t)\d\W_t^{\Q_1},\\
\d \Y^{\e}_t&=-\frac{1}{\e}[\mu\A \Y^{\e}_t+\beta\mathcal{C}(\Y_t^{\e})-g(\X^{\e}_t,\Y^{\e}_t)]\d t+\frac{1}{\sqrt{\e}}\sigma_2(\X^{\e}_t,\Y^{\e}_t)\d\W_t^{\Q_2},\\
\X^{\e}_0&=\x,\ \Y^{\e}_0=\y,
\end{aligned}
\right. 
\end{equation}
where $\W_t^{\Q_i}$ $(i=1,2)$ are $\H$-valued $\Q_i$-Wiener process and $\Q_i$'s, for $i=1,2$  are positive symmetric trace class operators in $\H$. For $i=1,2$, we define $\H^i_0:=\Q^{1/2}_i\H$. Then $\H^i_0$ is a Hilbert space with the inner product $(\x,\y)_{\H^i_0}=(\Q^{-1/2}_i\x,\Q^{-1/2}_i\y),\ \text{ for }\ \x,\y\in\H_0^i$ and norm $\|\x\|_{\H^i_0}=\|\Q^{-1/2}_i\x\|_{\H}$, where $\Q^{-1/2}_i$ is the pseudo-inverse of $\Q^{1/2}_i$. Let $\mathcal{L}(\H)$ denotes the space of all bounded linear operators on $\H$ and $\mathcal{L}_{\Q_i}(\H)$ denotes the space of all Hilbert-Schmidt operators from $\H^i_0$ to $\H$.  Since $\Q_i$ is a trace class operator, the embedding of $\H_0^i$ in $\H$ is Hilbert-Schmidt and the space $\mathcal{L}_{\Q_i}(\H)$ is a Hilbert space equipped with the norm $ \left\|\Psi\right\|^2_{\mathcal{L}_{\Q_i}}=\Tr\left(\Psi{\Q_i}\Psi^*\right)=\sum_{k=1}^{\infty}\|{\Q_i}^{1/2}\Psi^*e_k\|_{\H}^2$ and inner product $(\Psi,\Phi)_{\mathcal{L}_{\Q_i}}=\Tr(\Psi{\Q_i}\Phi^*)=\sum_{k=1}^{\infty}({\Q_i}^{1/2}\Phi^*e_k,{\Q_i}^{1/2}\Psi^*e_k)$. For more details, the interested readers are referred to see \cite{DaZ}. We need the following Assumptions on $f,g,\sigma_1$ and $\sigma_2$ to obtain our main results. 
\begin{assumption}\label{ass3.6}
	The functions $f,g:\H\times\H\to\H$, $\sigma_1:\H\to\mathcal{L}_{\Q_1}(\H)$ and $\sigma_2:\H\times\H\to\mathcal{L}_{\Q_2}(\H)$ satisfy the following Assumptions:
	\begin{itemize}
		\item [(A1)] The functions $f,g,\sigma_1,\sigma_2$ are Lipschitz continuous, that is, there exist positive constants $C,L_g$ and $L_{\sigma_2}$ such that for any $\x_1,\x_2,\y_1,\y_2\in\H$, we have 
		\begin{align*}
		\|f(\x_1,\y_1)-f(\x_2,\y_2)\|_{\H}&\leq C(\|\x_1-\x_2\|_{\H}+\|\y_1-\y_2\|_{\H}),\\
		\|g(\x_1,\y_1)-g(\x_2,\y_2)\|_{\H}&\leq C\|\x_1-\x_2\|_{\H}+L_g\|\y_1-\y_2\|_{\H},\\
		\|\sigma_1(\x_1)-\sigma_1(\x_2)\|_{\mathcal{L}_{\Q_1}}&\leq C\|\x_1-\x_2\|_{\H},\\
		\|\sigma_2(\x_1,\y_1)-\sigma_2(\x_2,\y_2)\|_{\mathcal{L}_{\Q_2}}&\leq C\|\x_1-\x_2\|_{\H}+L_{\sigma_2}\|\y_1-\y_2\|_{\H}. 
		\end{align*}
		\item [(A2)] There exists a constant $\zeta\in(0,1)$ such that 
		\begin{align*}
		\|\sigma_2(\x,\y)\|_{\mathcal{L}_{\Q_2}}\leq C(1+\|\x\|_{\H}+\|\y\|_{\H}^{\zeta}), \ \text{ for all }\ \x,\y\in\H. 
		\end{align*}
		\item [(A3)] The Brinkman coefficient $\mu>0$, the smallest eigenvalue $\lambda_1$  of the Stokes operator and the Lipschitz constants $L_g$ and $L_{\sigma_2}$ satisfy $$\mu\lambda_1-2L_g-2L_{\sigma_2}^2>0.$$ 
	\end{itemize}
\end{assumption}

Let us now provide the definition of a unique global strong solution in the probabilistic sense to the system (\ref{3.6}).
\begin{definition}[Global strong solution]
	Let $(\x,\y)\in\H\times\H$ be given. An $\H\times\H$-valued $(\mathscr{F}_t)_{t\geq 0}$-adapted stochastic process $(\X_{t}^{\e},\Y_{t}^{\e})$ is called a \emph{strong solution} to the system (\ref{3.6}) if the following conditions are satisfied: 
	\begin{enumerate}
		\item [(i)] the process $(\X^{\e},\Y^{\e})\in\mathrm{L}^2(\Omega;\mathrm{L}^{\infty}(0,T;\H)\cap\mathrm{L}^2(0,T;\V))\cap\mathrm{L}^{r+1}(\Omega;\mathrm{L}^{r+1}(0,T;\widetilde{\L}^{r+1}))\times\mathrm{L}^2(\Omega;\mathrm{L}^{\infty}(0,T;\H)\cap\mathrm{L}^2(0,T;\V))\cap\mathrm{L}^{r+1}(\Omega;\mathrm{L}^{r+1}(0,T;\widetilde{\L}^{r+1}))$ and $(\X^{\e}_t,\Y^{\e}_t)$ has a $(\V\cap\widetilde{\L}^{r+1})\times(\V\cap\widetilde{\L}^{r+1})$-valued  modification, which is progressively measurable with continuous paths in $\H\times\H$ and $(\X^{\e},\Y^{\e})\in\C([0,T];\H)\cap\mathrm{L}^2(0,T;\V)\cap\mathrm{L}^{r+1}(0,T;\widetilde{\L}^{r+1})\times\C([0,T];\H)\cap\mathrm{L}^2(0,T;\V)\cap\mathrm{L}^{r+1}(0,T;\widetilde{\L}^{r+1})$, $\mathbb{P}$-a.s.,
		\item [(ii)] the following equality holds for every $t\in [0, T ]$, as an element of $(\V'+\wi\L^{\frac{r+1}{r}})\times(\V'+\wi\L^{\frac{r+1}{r}}),$ $\mathbb{P}$-a.s.
	\begin{equation}\label{4.4}
	\left\{
	\begin{aligned}
	\X^{\e}_t&=\x-\int_0^t[\mu\A \X^{\e}_s+\B(\X^{\e}_s)+\beta\mathcal{C}(\X^{\e}_s)-f(\X^{\e}_s,\Y^{\e}_s)]\d s+\int_0^t\sigma_1(\X^{\e}_s)\d\W_s^{\Q_1},\\
	 \Y^{\e}_t&=\y-\frac{1}{\e}\int_0^t[\mu\A \Y^{\e}_s+\beta\mathcal{C}(\Y_s^{\e})-g(\X^{\e}_s,\Y^{\e}_s)]\d s+\frac{1}{\sqrt{\e}}\int_0^t\sigma_2(\X^{\e}_s,\Y^{\e}_s)\d\W_s^{\Q_2},
	\end{aligned}
	\right. 
	\end{equation}
			\item [(iii)] the following It\^o formula (energy equality) holds true: 
				\begin{align}
			&	\|	\X^{\e}_t\|_{\H}^2+\|\Y^{\e}_t\|_{\H}^2+2\mu \int_0^t\left(\|	\X^{\e}_s\|_{\V}^2+\frac{1}{\e}\|\Y^{\e}_s\|_{\V}^2\right)\d s+2\beta\int_0^t\left(\|\X^{\e}_s\|_{\widetilde{\L}^{r+1}}^{r+1}+\frac{1}{\e}\|\Y^{\e}_s\|_{\widetilde{\L}^{r+1}}^{r+1}\right)\d s\nonumber\\&=\|\x\|_{\H}^2+\|\y\|_{\H}^2+2\int_0^t(f(\X^{\e}_s,\Y^{\e}_s),\X^{\e}_s)\d s+\frac{2}{\e}\int_0^t(g(\X^{\e}_s,\Y^{\e}_s),\X^{\e}_s)\d s\nonumber\\&\quad+\int_0^t\left(\|\sigma_1(\X^{\e}_s)\|_{\mathcal{L}_{\Q_1}}^2+\frac{1}{\e}\|\sigma_2(\X^{\e}_s,\Y^{\e}_s)\|_{\mathcal{L}_{\Q_2}}^2\right)\d s+2\int_0^t(\sigma_1(\X^{\e}_s)\d\W_s^{\Q_1},\X^{\e}_s)\nonumber\\&\quad+\frac{2}{\sqrt{\e}}\int_0^t(\sigma_2(\X^{\e}_s,\Y^{\e}_s)\d\W_s^{\Q_2},\Y^{\e}_s),
			\end{align}
			for all $t\in(0,T)$, $\mathbb{P}$-a.s.
	\end{enumerate}
\end{definition}
An alternative version of condition (\ref{4.4}) is to require that for any  $(\mathrm{U},\mathrm{V})\in(\V\cap\widetilde{\L}^{r+1})\times(\V\cap\widetilde{\L}^{r+1})$, $\mathbb{P}$-a.s.:
	\begin{equation}\label{4.5}
\left\{
\begin{aligned}
(\X^{\e}_t,\mathrm{U})&=(\x,\mathrm{U})-\int_0^t\langle\mu\A \X^{\e}_s+\B(\X^{\e}_s)+\beta\mathcal{C}(\X^{\e}_s)-f(\X^{\e}_s,\Y^{\e}_s),\mathrm{U}\rangle\d s\nonumber\\&\quad+\int_0^t(\sigma_1(\X^{\e}_s)\d\W_s^{\Q_1},\mathrm{U}),\\
(\Y^{\e}_t,\mathrm{V})&=(\y,\mathrm{V})-\frac{1}{\e}\int_0^t\langle\mu\A \Y^{\e}_s+\beta\mathcal{C}(\Y_s^{\e})-g(\X^{\e}_s,\Y^{\e}_s),\mathrm{V}\rangle\d s\nonumber\\&\quad+\frac{1}{\sqrt{\e}}\int_0^t(\sigma_2(\X^{\e}_s,\Y^{\e}_s)\d\W_s^{\Q_2},\mathrm{V}).
\end{aligned}
\right. 
\end{equation}
\begin{definition}
	A strong solution $(\X^{\e}_t,\Y^{\e}_t)$ to (\ref{3.6}) is called a
	\emph{pathwise  unique strong solution} if
	$(\wi\X^{\e}_t,\wi\Y^{\e}_t)$ is an another strong
	solution, then $$\mathbb{P}\Big\{\omega\in\Omega:  (\X^{\e}_t,\Y^{\e}_t)=(\wi\X^{\e}_t,\wi\Y^{\e}_t),\  \text{ for all }\ t\in[0,T]\Big\}=1.$$ 
\end{definition}
The strong solutions to the systems  \eqref{3127} and \eqref{3128} can be defined  in a similar way.

\subsection{Global strong solution} In this section, we discuss the existence and uniqueness of strong solution to the system \eqref{3.6}. 
For convenience, we make use of  the following simplified notations. 	Let us define $\mathscr{H}:=\H\times\H$. For any $\mathrm{U}=(\x_1,\x_2),\mathrm{V}=(\y_1,\y_2)\in\mathscr{H},$ let us denote the inner product and norm on this Hilbert space by 
\begin{align}
(\mathrm{U},\mathrm{V})=(\x_1,\y_1)+(\x_2,\y_2), \ \|\mathrm{U}\|_{\mathscr{H}}=\sqrt{(\mathrm{U},\mathrm{U})}=\sqrt{\|\x_1\|_{\H}^2+\|\x_2\|_{\H}^2}. 
\end{align}
In a similar way, we define $\mathscr{V}:=\V\times\V$. The  inner product and norm on this Hilbert space is defined by 
\begin{align}
(\mathrm{U},\mathrm{V})_{\mathscr{V}}=(\nabla\x_1,\nabla\y_1)+(\nabla\x_2,\nabla\y_2), \ \|\mathrm{U}\|_{\mathscr{V}}=\sqrt{(\mathrm{U},\mathrm{U})_{\mathscr{V}}}=\sqrt{\|\nabla\x_1\|_{\H}^2+\|\nabla\x_2\|_{\H}^2}, 
\end{align}
for all $\mathrm{U},\mathrm{V}\in\mathscr{V}$. We denote $\mathscr{V}'$ as the dual of $\mathscr{V}$. We define the space $\widetilde{\mathfrak{L}}^{r+1}:=\widetilde{\L}^{r+1}\times\widetilde{\L}^{r+1}$ with the norm given by 
$$\|\mathrm{U}\|_{\widetilde{\mathfrak{L}}^{r+1}}=\left\{\|\x_1\|_{\wi\L^{r+1}}^{r+1}+\|\x_2\|_{\wi\L^{r+1}}^{r+1}\right\}^{r+1},$$ 	for all $\mathrm{U}\in\widetilde{\mathfrak{L}}^{r+1}$. We represent the duality pairing between $\mathscr{V}$ and its dual $\mathscr{V}'$,  $\mathfrak{L}^{r+1}$ and its dual $\mathfrak{L}^{\frac{r+1}{r}}$, and $\mathscr{V}\cap\mathfrak{L}^{r+1}$ and its dual $\mathscr{V}'+\mathfrak{L}^{\frac{r+1}{r}}$ as $\langle\cdot,\cdot\rangle$. Note that we have the Gelfand triple $\mathscr{V}\cap\mathfrak{L}^{r+1} \subset\mathscr{H}\subset\mathscr{V}'+\mathfrak{L}^{\frac{r+1}{r}}$. 
Let us denote $\W_t=(\W_t^{\Q_1},\W_t^{\Q_2})$ for  an  $\mathscr{H}$-valued $\Q$-Wiener process in $\mathscr{V}$, where $\Q=(\Q_1,\Q_2)$ is a positive symmetric trace class operator on $\mathscr{H}$. Let us now rewrite the system \eqref{3.6} for $\Z^{\e}_{t}=(\X^{\e}_{t},\Y^{\e}_{t})$ as 
\begin{equation}\label{3p6}
\left\{
\begin{aligned}
\d\Z^{\e}_t&=-\left[\mu\wi\A\Z^{\e}_t+\wi\F(\Z^{\e}_t)\right]\d t+\wi\sigma(\Z^{\e}_t)\d\W_t, \\
\Z^{\e}_0&=(\x,\y)\in\mathscr{H}, 
\end{aligned}
\right. 
\end{equation}
where 
\begin{align*}
\wi\A\Z^{\e}&=\left(\A\X^{\e},\frac{1}{\e}\A\Y^{\e}\right), \\
\wi\F(\Z^{\e})&=\left(\B(\X^{\e})+\beta\mathcal{C}(\X^{\e})-f(\X^{\e},\Y^{\e}),\frac{\beta}{\e}\mathcal{C}(\Y^{\e})-\frac{1}{\e}g(\X^{\e},\Y^{\e})\right), \\
\wi\sigma(\Z^{\e})&=\left(\sigma_1(\X^{\e}),\frac{1}{\sqrt{\e}}\sigma_2(\X^{\e},\Y^{\e})\right). 
\end{align*}
Note that the mappings $\widetilde{\A}:\mathscr{V}\to\mathscr{V}'$ and $\widetilde{\F}:\mathscr{V}\cap\mathfrak{L}^{r+1}\to\mathscr{V}'+\mathfrak{L}^{\frac{r+1}{r}}$ are well defined.   It can be easily seen that the operator $\wi\sigma:\mathscr{H}\to\mathcal{L}_{\Q}(\mathscr{H})$, where $\mathcal{L}_{\Q}(\mathscr{H})$ is the space of all Hilbert-Schmidt operators from $\Q^{1/2}\mathscr{H}$ to $\mathscr{H}$. Furthermore, the norm on $\mathcal{L}_{\Q}(\mathscr{H})$ is defined by 
\begin{align}
\|\wi\sigma(\z)\|_{\mathcal{L}_{\Q}}=\sqrt{\|\sigma_1(\x)\|_{\mathcal{L}_{\Q_1}}^2+\|\sigma_2(\x,\y)\|_{\mathcal{L}_{\Q_2}}^2}, \ \text{ for }\  \z=(\x,\y)\in\mathscr{H}. 
\end{align}

Let us first prove the existence and uniqueness of strong solution to the system \eqref{3p6}. 

\begin{theorem}\label{exis}
	Let $(\x,\y)\in \mathscr{H}$ be given.  Then for $n=2$, $r\in[1,\infty)$ and $n=3$, $r\in [3,\infty)$ ($2\beta\mu\geq 1$, for $r=3$), there exists a \emph{pathwise unique  strong solution}	$\Z^{\e}$ to the system (\ref{3p6}) such that \begin{align*}\Z^{\e}&\in\mathrm{L}^2(\Omega;\mathrm{L}^{\infty}(0,T;\mathscr{H})\cap\mathrm{L}^2(0,T;\mathscr{V}))\cap\mathrm{L}^{r+1}(\Omega;\mathrm{L}^{r+1}(0,T;\widetilde{\mathfrak{L}}^{r+1})),\end{align*} and a continuous modification with trajectories in $\mathscr{H}$ and $\Z^{\e}\in\C([0,T];\mathscr{H})\cap\mathrm{L}^2(0,T;\mathscr{V})\cap\mathrm{L}^{r+1}(0,T;\widetilde{\mathfrak{L}}^{r+1})$, $\mathbb{P}$-a.s.
\end{theorem}

\begin{proof}
In order to complete the existence proof, we need to show the  monotonicity and hemicontinuity properties of the operator $\mu\widetilde{\A}+\widetilde{\F}$. 

\vskip 0.1 cm
\textbf{Step 1:} \emph{$\mu\widetilde{\A}+\widetilde{\F}:\mathscr{V}\cap\mathfrak{L}^{r+1}\to\mathscr{V}'+\mathfrak{L}^{\frac{r+1}{r}}$}. For $n=2,3$, $r\in[3,\infty)$ and $\mathrm{U}=(\x,\y)\in\mathscr{V}\cap\mathfrak{L}^{r+1}$,  we first note that 
\begin{align}\label{3.11}
\|\mu\widetilde{\A}\mathrm{U}+\widetilde{\F}(\mathrm{U})\|_{\mathscr{V}'+\mathfrak{L}^{\frac{r+1}{r}}}&\leq\mu\|\A\x\|_{\V'}+\frac{\mu}{\e}\|\A\y\|_{\V'}+\|\B(\x)\|_{\V'+\wi\L^{\frac{r+1}{r}}}+\beta\|\mathcal{C}(\x)\|_{\wi\L^{\frac{r+1}{r}}}\nonumber\\&\quad+\frac{\beta}{\e}\|\mathcal{C}(\y)\|_{\wi\L^{\frac{r+1}{r}}}+\|f(\x,\y)\|_{\mathscr{V}'}+\frac{1}{\e}\|g(\x,\y)\|_{\mathscr{V}'}\nonumber\\&\leq \mu\|\x\|_{\V}+\frac{\mu}{\e}\|\y\|_{\V}+\|\x\|_{\wi\L^{r+1}}^{\frac{r+1}{r-1}}\|\x\|_{\H}^{\frac{r-3}{r-1}}+\beta\|\x\|_{\wi\L^{r+1}}^r+\frac{\beta}{\e}\|\y\|_{\wi\L^{r+1}}^r\nonumber\\&\quad+\frac{C}{\lambda_1}(1+\|\x\|_{\H}+\|\y\|_{\H})+\frac{1}{\lambda_1\e}(C(1+\|\x\|_{\H})+L_g\|\y\|_{\H})\nonumber\\&\leq\mu\max\left\{1,\frac{1}{\e}\right\}\|\mathrm{U}\|_{\mathscr{V}}+\|\mathrm{U}\|_{\mathfrak{L}^{r+1}}^{\frac{r+1}{r-1}}\|\mathrm{U}\|_{\H}^{\frac{r-3}{r-1}}+\beta\max\left\{1,\frac{1}{\e}\right\}\|\mathrm{U}\|_{\mathfrak{L}^{r+1}}^r\nonumber\\&\quad+\frac{C}{\lambda_1}\left(1+\frac{1}{\e}\right)\left(1+\|\mathrm{U}\|_{\H}\right),
\end{align}
and hence $\mu\widetilde{\A}+\widetilde{\F}:\mathscr{V}\cap\mathfrak{L}^{r+1}\to\mathscr{V}'+\mathfrak{L}^{\frac{r+1}{r}}$. For $n=2$ and $r\in[1,3]$, one can estimate $\|\B(\u)\|_{\V'}\leq\|\u\|_{\wi\L^4}^2\leq\sqrt{2}\|\u\|_{\H}\|\u\|_{\V}$, by using H\"older's and Ladyzhenskaya's inequalities, and an estimate similar to \eqref{3.11} follows.

\vskip 0.1 cm
\textbf{Step 2:} \emph{Monotonicity property of the operator $\mu\widetilde{\A}+\widetilde{\F}$}.  First, we consider the case $n=2$ and $r\in[1,3]$. For $\mathrm{U}=(\x_1,\y_1)$ and $\mathrm{V}=(\x_2,\y_2)$, we have 
\begin{align}\label{3p12}
&\langle\mu\wi\A(\mathrm{U}-\mathrm{V})+\widetilde{\F}(\mathrm{U})-\widetilde{\F}(\mathrm{V}),\mathrm{U}-\mathrm{V}\rangle \nonumber\\&= \langle\mu\A(\x_1-\x_2),\x_1-\x_2\rangle +\frac{\mu}{\e}\langle\A(\y_1-\y_2),\y_1-\y_2\rangle +\langle\B(\x_1)-\B(\x_2),\x_1-\x_2\rangle \nonumber\\&\quad+\beta\langle\mathcal{C}(\x_1)-\mathcal{C}(\x_2),\x_1-\x_2\rangle+\frac{\beta}{\e}\langle\mathcal{C}(\y_1)-\mathcal{C}(\y_2),\y_1-\y_2\rangle\nonumber\\&\quad-(f(\x_1,\y_1)-f(\x_2,\y_2),\x_1-\x_2)-\frac{1}{\e}(g(\x_1,\y_1)-g(\x_2,\y_2),\y_1-\y_2)\nonumber\\&=\mu\|\x_1-\x_2\|_{\V}^2+\frac{\mu}{\e}\|\y_1-\y_2\|_{\V}^2+\langle\B(\x_1-\x_2,\x_2),\x_1-\x_2\rangle \nonumber\\&\quad+\beta\langle\mathcal{C}(\x_1)-\mathcal{C}(\x_2),\x_1-\x_2\rangle+\frac{\beta}{\e}\langle\mathcal{C}(\y_1)-\mathcal{C}(\y_2),\y_1-\y_2\rangle\nonumber\\&\quad-(f(\x_1,\y_1)-f(\x_2,\y_2),\x_1-\x_2)-\frac{1}{\e}(g(\x_1,\y_1)-g(\x_2,\y_2),\y_1-\y_2)\nonumber\\&\geq \mu\|\x_1-\x_2\|_{\V}^2+\frac{\mu}{\e}\|\y_1-\y_2\|_{\V}^2-\|\x_2\|_{\wi\L^4}\|\x_1-\x_2\|_{\V}\|\x_1-\x_2\|_{\wi\L^4}\nonumber\\&\quad+\frac{\beta}{2^{r-1}}\|\x_1-\x_2\|_{\wi\L^{r+1}}^{r+1}+\frac{\beta}{2^{r-1}\e}\|\y_1-\y_2\|_{\wi\L^{r+1}}^{r+1}-\|f(\x_1,\y_1)-f(\x_2,\y_2)\|_{\H}\|\x_1-\x_2\|_{\H}\nonumber\\&\quad-\frac{1}{\e}\|g(\x_1,\y_1)-g(\x_2,\y_2)\|_{\H}\|\y_1-\y_2\|_{\H}\nonumber\\&\geq\frac{\mu}{2}\|\x_1-\x_2\|_{\V}^2+\frac{\mu}{\e}\|\y_1-\y_2\|_{\V}^2+\frac{\beta}{2^{r-1}}\|\x_1-\x_2\|_{\wi\L^{r+1}}^{r+1}+\frac{\beta}{2^{r-1}\e}\|\y_1-\y_2\|_{\wi\L^{r+1}}^{r+1}\nonumber\\&\quad-\frac{27}{32\mu^3}\|\x_2\|_{\wi\L^4}^4\|\x_1-\x_2\|_{\H}^2- C\left(1+\frac{1}{\e}\right)\left(\|\x_1-\x_2\|_{\H}^2+\|\y_1-\y_2\|_{\H}^2\right)-\frac{L_g}{\e}\|\y_1-\y_2\|_{\H}^2, 
\end{align}
where we used the Assumption \ref{ass3.6} (A1), \eqref{2.23}, \eqref{a215}, H\"older's,  Ladyzhenskaya's and Young's inequalities.  From the above relation, it is clear that 
\begin{align}\label{3p13}
\langle\mu\wi\A(\mathrm{U}-\mathrm{V})+\widetilde{\F}(\mathrm{U})-\widetilde{\F}(\mathrm{V}),\mathrm{U}-\mathrm{V}\rangle+\left[\frac{27}{32\mu^3}\|\mathrm{V}\|_{\widetilde{\mathfrak{L}}^4}^4+C\left(1+\frac{1}{\e}\right)+\frac{L_g}{\e}\right]\|\mathrm{U}-\mathrm{V}\|_{\mathscr{H}}^2\geq 0,
\end{align}
for all $\mathrm{U},\mathrm{V}\in\mathscr{V}\cap\mathfrak{L}^{r+1}=\mathscr{V}$, for $n=2$ and $r\in[1,3]$. Now, for an $\widetilde{\mathfrak{L}}^4$-ball in $\mathscr{V}$, that is, for $\mathrm{V}\in\mathcal{B}_R$, where $\mathcal{B}_R:=\left\{\mathrm{Z}\in\mathscr{V}:\|\mathrm{Z}\|_{\widetilde{\mathfrak{L}}^4}\leq R \right\}$, we get from \eqref{3p13} that 
\begin{align}\label{3p14}
\langle\mu\wi\A(\mathrm{U}-\mathrm{V})+\widetilde{\F}(\mathrm{U})-\widetilde{\F}(\mathrm{V}),\mathrm{U}-\mathrm{V}\rangle+\left[\frac{27R^4}{32\mu^3}+C\left(1+\frac{1}{\e}\right)+\frac{L_g}{\e}\right]\|\mathrm{U}-\mathrm{V}\|_{\mathscr{H}}^2\geq 0, 
\end{align}
which implies that the operator $\mu\widetilde{\A}+\widetilde{\F}:\mathscr{V}\to\mathscr{V}'$ is locally monotone. Furthermore, we get 
\begin{align}
&\langle\mu\wi\A(\mathrm{U}-\mathrm{V})+\widetilde{\F}(\mathrm{U})-\widetilde{\F}(\mathrm{V}),\mathrm{U}-\mathrm{V}\rangle+\left[\frac{27R^4}{32\mu^3}+C\left(1+\frac{1}{\e}\right)+\frac{L_g}{\e}+L_{\sigma_2}^2\right]\|\mathrm{U}-\mathrm{V}\|_{\mathscr{H}}^2\nonumber\\&\geq 	\|\sigma_1(\x_1)-\sigma_1(\x_2)\|_{\mathcal{L}_{\Q_1}}^2+	\|\sigma_2(\x_1,\y_1)-\sigma_2(\x_2,\y_2)\|_{\mathcal{L}_{\Q_2}}^2\geq 0. 
\end{align}

Let us now consider the case $n=2,3$ and $r\in(3,\infty)$. 	From \eqref{2.23}, we easily have 
\begin{align}\label{2a27}
\beta	\langle\mathcal{C}(\x_1)-\mathcal{C}(\x_2),\x_1-\x_2\rangle \geq \frac{\beta}{2}\||\x_2|^{\frac{r-1}{2}}(\x_1-\x_2)\|_{\H}^2. 
\end{align}	Using H\"older's and Young's inequalities, we estimate $|\langle\B(\x_1-\x_2,\x_1-\x_2),\x_2\rangle|$ as  
\begin{align}\label{2a28}
|\langle\B(\x_1-\x_2,\x_1-\x_2),\x_2\rangle|&\leq\|\x_1-\x_2\|_{\V}\|\x_2(\x_1-\x_2)\|_{\H}\nonumber\\&\leq\frac{\mu }{2}\|\x_1-\x_2\|_{\V}^2+\frac{1}{2\mu }\|\x_2(\x_1-\x_2)\|_{\H}^2.
\end{align}
We take the term $\|\x_2(\x_1-\x_2)\|_{\H}^2$ from \eqref{2a28} and use H\"older's and Young's inequalities to estimate it as (see \cite{KWH} also)
\begin{align}\label{2a29}
&\int_{\mathcal{O}}|\x_2(x)|^2|\x_1(x)-\x_2(x)|^2\d x\nonumber\\&=\int_{\mathcal{O}}|\x_2(x)|^2|\x_1(x)-\x_2(x)|^{\frac{4}{r-1}}|\x_1(x)-\x_2(x)|^{\frac{2(r-3)}{r-1}}\d x\nonumber\\&\leq\left(\int_{\mathcal{O}}|\x_2(x)|^{r-1}|\x_1(x)-\x_2(x)|^2\d x\right)^{\frac{2}{r-1}}\left(\int_{\mathcal{O}}|\x_1(x)-\x_2(x)|^2\d x\right)^{\frac{r-3}{r-1}}\nonumber\\&\leq{\beta\mu }\left(\int_{\mathcal{O}}|\x_2(x)|^{r-1}|\x_1(x)-\x_2(x)|^2\d x\right)+\frac{r-3}{r-1}\left(\frac{2}{\beta\mu (r-1)}\right)^{\frac{2}{r-3}}\left(\int_{\mathcal{O}}|\x_1(x)-\x_2(x)|^2\d x\right),
\end{align}
for $r>3$. Using \eqref{2a29} in \eqref{2a28}, we find 
\begin{align}\label{2a30}
&|\langle\B(\x_1-\x_2,\x_1-\x_2),\x_2\rangle|\nonumber\\&\leq\frac{\mu }{2}\|\x_1-\x_2\|_{\V}^2+\frac{\beta}{2}\||\x_2|^{\frac{r-1}{2}}(\x_1-\x_2)\|_{\H}^2+\frac{r-3}{2\mu(r-1)}\left(\frac{2}{\beta\mu (r-1)}\right)^{\frac{2}{r-3}}\|\x_1-\x_2\|_{\H}^2.
\end{align}
Combining \eqref{3p12}, \eqref{2a27} and \eqref{2a30}, we get 
\begin{align}
&\langle\mu\wi\A(\mathrm{U}-\mathrm{V})+\widetilde{\F}(\mathrm{U})-\widetilde{\F}(\mathrm{V}),\mathrm{U}-\mathrm{V}\rangle\nonumber\\&\quad+\left[\frac{r-3}{2\mu(r-1)}\left(\frac{2}{\beta\mu (r-1)}\right)^{\frac{2}{(r-3)}}+C\left(1+\frac{1}{\e}\right)+\frac{L_g}{\e}\right]\|\mathrm{U}-\mathrm{V}\|_{\mathscr{H}}^2\geq 0,
\end{align}
for $r>3$ and the monotonicity of the operator  $\mu\widetilde{\A}+\widetilde{\F}:\mathscr{V}\cap\mathfrak{L}^{r+1}\to\mathscr{V}'+\mathfrak{L}^{\frac{r+1}{r}}$ follows. 

Now, for $n=3$ and $r=3$,	from \eqref{2.23}, we have 
\begin{align}\label{2a31}
\beta\langle\mathcal{C}(\x_1)-\mathcal{C}(\x_2),\x_1-\x_2\rangle\geq\frac{\beta}{2}\|\x_2(\x_1-\x_2)\|_{\H}^2. 
\end{align}
We estimate the term $|\langle\B(\x_1-\x_2,\x_1-\x_2),\x_2\rangle|$ using H\"older's and Young's inequalities as 
\begin{align}\label{2a32}
|\langle\B(\x_1-\x_2,\x_1-\x_2),\x_2\rangle|&\leq\|\x_2(\x_1-\x_2)\|_{\H}\|\x_1-\x_2\|_{\V}\nonumber\\& \leq\mu \|\x_1-\x_2\|_{\V}^2+\frac{1}{4\mu }\|\x_2(\x_1-\x_2)\|_{\H}^2.
\end{align}
Combining \eqref{3p12}, \eqref{2a31} and \eqref{2a32}, we obtain 
\begin{align}
&\langle\mu\wi\A(\mathrm{U}-\mathrm{V})+\widetilde{\F}(\mathrm{U})-\widetilde{\F}(\mathrm{V}),\mathrm{U}-\mathrm{V}\rangle+\left[C\left(1+\frac{1}{\e}\right)+\frac{L_g}{\e}\right]\|\mathrm{U}-\mathrm{V}\|_{\mathscr{H}}^2\nonumber\\&\geq\frac{1}{2}\left(\beta-\frac{1}{2\mu }\right)\|\x_2(\x_1-\x_2)\|_{\H}^2\geq 0,
\end{align}
provided $2\beta\mu \geq 1$ and the monotonicity of the operator $\mu\widetilde{\A}+\widetilde{\F}:\mathscr{V}\cap\mathfrak{L}^{4}\to\mathscr{V}'+\mathfrak{L}^{\frac{4}{3}}$ follows.

\vskip 0.1 cm
\textbf{Step 3:} \emph{Hemicontinuity property of the operator $\mu\widetilde{\A}+\widetilde{\F}$}. Let us now show that the operator $\mu\widetilde{\A}+\widetilde{\F}:\mathscr{V}\cap\mathfrak{L}^{r+1}\to\mathscr{V}'+\mathfrak{L}^{\frac{r+1}{r}}$ is hemicontinuous. We first consider the case $n=2,3$ and $r\in[3,\infty)$.	Let us take a sequence $\mathrm{U}^n\to \mathrm{U}$ in $\mathscr{V}\cap\mathfrak{L}^{r+1}$, that is, $\|\mathrm{U}^n-\mathrm{U}\|_{\widetilde{\mathfrak{L}}^{r+1}}+\|\mathrm{U}^n-\mathrm{U}\|_{\mathscr{V}}\to 0$, as $n\to\infty$. Therefore, we know that $\|\x^n-\x\|_{\wi\L^{r+1}}+\|\x^n-\x\|_{\V}+\|\y^n-\y\|_{\wi\L^{r+1}}+\|\y^n-\y\|_{\V}\to 0 $ as $n\to\infty$, for $\mathrm{U}^n=(\x^n,\y^n),\mathrm{U}=(\x,\y)\in\mathscr{V}\cap\mathfrak{L}^{r+1}$. For any $\mathrm{V}=(\wi\x,\wi\y)\in\mathscr{V}\cap\mathfrak{L}^{r+1}$, we consider 
\begin{align}\label{2p14}
&\langle\mu\wi\A(\mathrm{U}^n-\mathrm{U})+\widetilde{\F}(\mathrm{U}^n)-\widetilde{\F}(\mathrm{U}),\mathrm{V}\rangle\nonumber\\&= \langle\mu\A(\x^n-\x),\wi\x\rangle +\frac{\mu}{\e}\langle\A(\y^n-\y),\wi\y\rangle +\langle\B(\x^n)-\B(\x),\wi\x\rangle +\beta\langle\mathcal{C}(\x^n)-\mathcal{C}(\x),\wi\x\rangle\nonumber\\&\quad+\frac{\beta}{\e}\langle\mathcal{C}(\y^n)-\mathcal{C}(\y),\wi\y\rangle-(f(\x^n,\y^n)-f(\x,\y),\wi\x)-\frac{1}{\e}(g(\x^n,\y^n)-g(\x,\y),\wi\y). 
\end{align} 
Let us take $\langle\mu\A(\x^n-\x),\wi\x\rangle $ from \eqref{2p14} and estimate it as 
\begin{align}
|\langle\mu\A(\x^n-\x),\wi\x\rangle| =|(\nabla(\x^n-\x),\nabla\wi\x)|\leq\|\x^n-\x\|_{\V}\|\wi\x\|_{\V}\to 0, \ \text{ as } \ n\to\infty, 
\end{align}
since $\x^n\to \x$ in $\V$. Similarly, we get $	|\langle\mu\A(\y^n-\y),\widetilde\y\rangle| \to 0$ as $n\to\infty$. We estimate the term $\langle\B(\x^n)-\B(\x),\wi\x\rangle$ from \eqref{2p14} using H\"older's inequality as 
\begin{align}\label{325}
|\langle\B(\x^n)-\B(\x),\wi\x\rangle|&=|\langle\B(\x^n,\x^n-\x),\wi\x\rangle+\langle\B(\x^n-\x,\x),\wi\x\rangle|\nonumber\\&
\leq|\langle\B(\x^n,\wi\x),\x^n-\x\rangle|+|\langle\B(\x^n-\x,\wi\x),\x\rangle|\nonumber\\&\leq\left(\|\x^n\|_{\widetilde{\L}^{\frac{2(r+1)}{r-1}}}+\|\x\|_{\widetilde{\L}^{\frac{2(r+1)}{r-1}}}\right)\|\x^n-\x\|_{\widetilde{\L}^{r+1}}\|\wi\x\|_{\V}\nonumber\\&\leq \left(\|\x^n\|_{\H}^{\frac{r-3}{r-1}}\|\x^n\|_{\widetilde{\L}^{r+1}}^{\frac{2}{r-1}}+\|\x\|_{\H}^{\frac{r-3}{r-1}}\|\x\|_{\widetilde{\L}^{r+1}}^{\frac{2}{r-1}}\right)\|\x^n-\x\|_{\widetilde{\L}^{r+1}}\|\wi\x\|_{\V}\nonumber\\& \to 0, \ \text{ as } \ n\to\infty, 
\end{align}
since $\x^n\to\x$ in $\widetilde\L^{r+1}$ and $\x^n,\x\in\V\cap\widetilde\L^{r+1}$. We estimate the term $\langle \mathcal{C}(\x^n)-\mathcal{C}(\x),\wi\x\rangle$ from \eqref{214} using Taylor's formula and H\"older's inequality as (in fact, it is true for all $r\geq 1$)
\begin{align}
|\langle \mathcal{C}(\x^n)-\mathcal{C}(\x),\wi\x\rangle|&\leq \sup_{0<\theta<1}r\|(\x^n-\x)|\theta\x^n+(1-\theta)\x|^{r-1}\|_{\widetilde{\L}^{\frac{r+1}{r}}}\|\wi\x\|_{\widetilde{\L}^{r+1}}\nonumber\\&\leq r\|\x^n-\x\|_{\widetilde{\L}^{r+1}}\left(\|\x^n\|_{\widetilde{\L}^{r+1}}+\|\x\|_{\widetilde{\L}^{r+1}}\right)^{r-1}\|\wi\x\|_{\widetilde{\L}^{r+1}}\to 0, 
\end{align}
as $n\to\infty$,	since $\x^n\to\x$ in $\widetilde{\L}^{r+1}$ and $\x^n,\x\in\V\cap\widetilde{\L}^{r+1}$. Similarly, we obtain 
\begin{align}
|\langle \mathcal{C}(\y^n)-\mathcal{C}(\y),\wi\y\rangle|\to 0, 
\end{align}
as $n\to\infty$. 
Using the Assumption \ref{ass3.6} (A1), we get 
\begin{align}
|(f(\x^n,\y^n)-f(\x,\y),\wi\x)|&\leq\|f(\x^n,\y^n)-f(\x,\y)\|_{\H}\|\wi\x\|_{\H}\nonumber\\&\leq C(\|\x^n-\x\|_{\H}+\|\y^n-\y\|_{\H})\|\wi\x\|_{\H}\to 0, \ \text{ as } \ n\to\infty,
\end{align}
since $\x^n\to\x$ in $\H$ and $\y^n\to\y$ in $\H$. Once again making use of the Assumption \ref{ass3.6} (A.1), we find 
\begin{align}
|(g(\x^n,\y^n)-g(\x,\y),\wi\y)|\leq C\|\x^n-\x\|_{\H}+L_g\|\y^n-\y\|_{\H}\to 0, \ \text{ as } \ n\to\infty,
\end{align}
since $\x^n\to\x$ in $\H$ and $\y^n\to\y$ in $\H$.  From the above convergences, it is immediate that $\langle\mu\wi\A(\mathrm{U}^n-\mathrm{U})+\widetilde{\F}(\mathrm{U}^n)-\widetilde{\F}(\mathrm{U}),\mathrm{V}\rangle\to 0$, for all $\mathrm{V}\in\mathscr{V}\cap\mathfrak{L}^{r+1}$. 
Hence the operator $\mu\widetilde{\A}+\widetilde{\F}:\mathscr{V}\cap\mathfrak{L}^{r+1}\to\mathscr{V}'+\mathfrak{L}^{\frac{r+1}{r}}$ is demicontinuous, which implies that the operator $\mu\widetilde{\A}+\widetilde{\F}$ is hemicontinuous also. 

For $n=2$ and $r\in[1,3]$, we need to show the convergence \eqref{325} only. It can easily be seen that 
\begin{align}
|\langle\B(\x^n)-\B(\x),\wi\x\rangle|
&\leq|\langle\B(\x^n,\wi\x),\x^n-\x\rangle|+|\langle\B(\x^n-\x,\wi\x),\x\rangle|\nonumber\\&\leq\left(\|\x^n\|_{\wi\L^4}+\|\x\|_{\wi\L^4}\right)\|\wi\x\|_{\V}\|\x^n-\x\|_{\wi\L^4} \to 0, \ \text{ as } \ n\to\infty, 
\end{align}
since $\x^n\to\x$ in $\widetilde\L^{4}$ and $\x^n,\x\in\V$.

Now, proceeding similarly as in the proof of Theorem 3.7, \cite{MTM8}, we obtain the existence and uniqueness of strong solution to the system \eqref{3.6}  by using the above properties and a stochastic generalization of the Minty-Browder technique. For the case $n=2$ and $r\in[1,3],$ one can use the similar techniques in the works \cite{ICAM,MTM6,SSSP}, etc, where a localized version of the stochastic generalization of Minty-Browder technique is used in the proofs for the 2D stochastic Navier-Stokes equations and related models perturbed by Gaussian noise. 
\end{proof}

Let us now prove some uniform bounds with respect to $\e\in(0,1)$ for $p^{\mathrm{th}}$-moments of the strong solution $(\X^{\e}_t,\Y^{\e}_t)$ to the system \eqref{3.6}. 
\begin{lemma}\label{lem3.7}
	For any $\x,\y\in\H$, $T>0$, $p\geq 1$ and $\e\in(0,1)$, there exists a constant $C_{p,\mu,\lambda_1,L_g,T}>0$ such that the strong solution $(\X^{\e}_t,\Y^{\e}_t)$ to the system \eqref{3.6} satisfies: 
	\begin{align}\label{3.7}
&	\E\left[\sup_{t\in[0,T]}\|\X^{\e}_t\|_{\H}^{2p}+p\mu\int_0^T\|\X^{\e}_t\|_{\H}^{2p-2}\|\X^{\e}_t\|_{\V}^2\d t+2\beta p\int_0^T\|\X^{\e}_t\|_{\H}^{2p-2}\|\X^{\e}_t\|_{\wi\L^{r+1}}^{r+1}\d t\right]\nonumber\\&\leq C_{p,\mu,\lambda_1,L_g,T}\left(1+\|\x\|_{\H}^{2p}+\|\y\|_{\H}^{2p}\right), 
	\end{align}
	and 
	\begin{align}\label{38}
\sup_{t\in[0,T]}	\E\left[\|\Y^{\e}_t\|_{\H}^{2p}\right]\leq C_{p,\mu,\lambda_1,L_g,T}\left(1+\|\x\|_{\H}^{2p}+\|\y\|_{\H}^{2p}\right). 
	\end{align}
\end{lemma}
\begin{proof}
	\textbf{Step 1:} \emph{$p^{\mathrm{th}}$-moment estimates for $\Y^{\e}_t$}. 
	An application of the infinite dimensional It\^o formula  to the process $\|\Y^{\e}_{t}\|_{\H}^2$  yields (cf. \cite{MTM8})
	\begin{align}\label{3.9}
	\|\Y^{\e}_t\|_{\H}^2&=\|\y\|_{\H}^2-\frac{2\mu}{\e}\int_0^t\|\Y^{\e}_s\|_{\V}^2\d s-\frac{2\beta}{\e}\int_0^t\|\Y^{\e}_s\|_{\wi\L^{r+1}}^{r+1}\d s+\frac{2}{\e}\int_0^t(g(\X^{\e}_s,\Y^{\e}_s),\Y^{\e}_s)\d s\nonumber\\&\quad+\frac{1}{\e}\int_0^t\|\sigma_2(\X^{\e}_s,\Y^{\e}_s)\|_{\mathcal{L}_{\Q_2}}^2\d s+\frac{2}{\sqrt{\e}}\int_0^t(\sigma_2(\X^{\e}_s,\Y^{\e}_s)\d\W_s^{\Q_2},\Y^{\e}_s), 
	\end{align}
	for all $t\in[0,T]$, $\mathbb{P}$-a.s. Now, applying It\^o's formula to the process $(\|\Y^{\e}_{t}\|_{\H}^2)^p$, we find 
	\begin{align}\label{3.10}
		\|\Y^{\e}_t\|_{\H}^{2p}&=\|\y\|_{\H}^{2p}-\frac{2p\mu}{\e}\int_0^t	\|\Y^{\e}_s\|_{\H}^{2p-2}\|\Y^{\e}_s\|_{\V}^2\d s-\frac{2p\beta}{\e}\int_0^t	\|\Y^{\e}_s\|_{\H}^{2p-2}\|\Y^{\e}_s\|_{\wi\L^{r+1}}^{r+1}\d s\nonumber\\&\quad+\frac{2p}{\e}\int_0^t\|\Y^{\e}_s\|_{\H}^{2p-2}(g(\X^{\e}_s,\Y^{\e}_s),\Y^{\e}_s)\d s+\frac{p}{\e}\int_0^t\|\Y^{\e}_s\|_{\H}^{2p-2}\|\sigma_2(\X^{\e}_s,\Y^{\e}_s)\|_{\mathcal{L}_{\Q_2}}^2\d s\nonumber\\&\quad+\frac{2p(p-1)}{\e}\int_0^t\|\Y^{\e}_s\|_{\H}^{2p-4}\Tr((\Y^{\e}_s\otimes\Y^{\e}_s)\sigma_2(\X^{\e}_s,\Y^{\e}_s)\Q_2\sigma_2^*(\X^{\e}_s,\Y^{\e}_s))\d s\nonumber\\&\quad+\frac{2p}{\sqrt{\e}}\int_0^t	\|\Y^{\e}_s\|_{\H}^{2p-2}(\sigma_2(\X^{\e}_s,\Y^{\e}_s)\d\W_s^{\Q_2},\Y^{\e}_s). 
	\end{align}
	Taking expectation in \eqref{3.10} and using the fact that the final term appearing in \eqref{3.10} is a martingale, we obtain 
	\begin{align}
	\E\left[\|\Y^{\e}_t\|_{\H}^{2p}\right]&=\|\y\|_{\H}^{2p}-\frac{2p\mu}{\e}\E\left[\int_0^t	\|\Y^{\e}_s\|_{\H}^{2p-2}\|\Y^{\e}_s\|_{\V}^2\d s\right]-\frac{2p\beta}{\e}\E\left[\int_0^t	\|\Y^{\e}_s\|_{\H}^{2p-2}\|\Y^{\e}_s\|_{\wi\L^{r+1}}^{r+1}\d s\right]\nonumber\\&\quad+\frac{2p}{\e}\E\left[\int_0^t\|\Y^{\e}_s\|_{\H}^{2p-2}(g(\X^{\e}_s,\Y^{\e}_s),\Y^{\e}_s)\d s\right]+\frac{p}{\e}\E\left[\int_0^t\|\Y^{\e}_s\|_{\H}^{2p-2}\|\sigma_2(\X^{\e}_s,\Y^{\e}_s)\|_{\mathcal{L}_{\Q_2}}^2\d s\right]\nonumber\\&\quad+\frac{2p(p-1)}{\e}\E\left[\int_0^t\|\Y^{\e}_s\|_{\H}^{2p-4}\Tr((\Y^{\e}_s\otimes\Y^{\e}_s)\sigma_2(\X^{\e}_s,\Y^{\e}_s)\Q_2\sigma_2^*(\X^{\e}_s,\Y^{\e}_s))\d s\right], 
	\end{align}
	for all $t\in[0,T]$. Thus, it is immediate that 
	\begin{align}\label{312}
	\frac{\d}{\d t}	\E\left[\|\Y^{\e}_t\|_{\H}^{2p}\right]&=-\frac{2p\mu}{\e}\E\left[	\|\Y^{\e}_t\|_{\H}^{2p-2}\|\Y^{\e}_t\|_{\V}^2\right]-\frac{2p\beta}{\e}\E\left[	\|\Y^{\e}_t\|_{\H}^{2p-2}\|\Y^{\e}_t\|_{\wi\L^{r+1}}^{r+1}\right]\nonumber\\&\quad+\frac{2p}{\e}\E\left[\|\Y^{\e}_t\|_{\H}^{2p-2}(g(\X^{\e}_t,\Y^{\e}_t),\Y^{\e}_t)\right]+\frac{p}{\e}\E\left[\|\Y^{\e}_t\|_{\H}^{2p-2}\|\sigma_2(\X^{\e}_t,\Y^{\e}_t)\|_{\mathcal{L}_{\Q_2}}^2\right]\nonumber\\&\quad+\frac{2p(p-1)}{\e}\E\left[\|\Y^{\e}_t\|_{\H}^{2p-4}\Tr((\Y^{\e}_t\otimes\Y^{\e}_t)\sigma_2(\X^{\e}_t,\Y^{\e}_t)\Q_2\sigma_2^*(\X^{\e}_t,\Y^{\e}_t))\right], 
	\end{align}
for a.e. $t\in[0,T]$. 	Using the Assumption \ref{ass3.6} (A1), the Cauchy-Schwarz inequality and Young's inequality, we get 
	\begin{align}\label{313}
\frac{2p}{\e}	\|\Y^{\e}_t\|_{\H}^{2p-2}(g(\X^{\e}_t,\Y^{\e}_t),\Y^{\e}_t)&\leq \frac{2p}{\e}\|g(\X^{\e}_t,\Y^{\e}_t)\|_{\H}\|\Y^{\e}_t\|_{\H}^{2p-1}\nonumber\\&\leq \frac{2p}{\e} \left(\|g(\mathbf{0},\mathbf{0})\|_{\H}+C\|\X^{\e}_t\|_{\H}+L_g\|\Y^{\e}_t\|_{\H}\right)\|\Y^{\e}_t\|_{\H}^{2p-1}\nonumber\\&\leq\left(\frac{p\mu\lambda_1}{2\e}+\frac{2pL_g}{\e}\right)\|\Y^{\e}_t\|_{\H}^{2p}+\frac{C_p}{\mu\lambda_1\e}\|\X^{\e}_t\|_{\H}^{2p}+\frac{C_p}{\mu\lambda_1\e}.
	\end{align}
	Using the Assumption \ref{ass3.6} (A2) and Young's inequality, we obtain 
	\begin{align}\label{314}
	\frac{p}{\e}\|\Y^{\e}_t\|_{\H}^{2p-2}\|\sigma_2(\X^{\e}_t,\Y^{\e}_t)\|_{\mathcal{L}_{\Q_2}}^2&\leq 	\frac{Cp}{\e}\|\Y^{\e}_t\|_{\H}^{2p-2}(1+\|\X^{\e}_t\|_{\H}^2+\|\Y^{\e}_t\|_{\H}^{2\zeta})\nonumber\\&\leq \frac{p\mu\lambda_1}{4\e}\|\Y^{\e}_t\|_{\H}^{2p}+\frac{C_p}{\mu\lambda_1\e}\|\X^{\e}_t\|_{\H}^{2p}+\frac{C_p}{\mu\lambda_1\e}.
	\end{align}
Note that $\Tr((\Y^{\e}_t\otimes\Y^{\e}_t)\sigma_2(\X^{\e}_t,\Y^{\e}_t)\Q_2\sigma_2^*(\X^{\e}_t,\Y^{\e}_t))\leq\|\Y^{\e}_t\|_{\H}^2\|\sigma_2(\X^{\e}_t,\Y^{\e}_t)\|_{\mathcal{L}_{\Q_2}}^2$. Using this fact, the Assumption \ref{ass3.6} (A2), and Young's inequality, we deduce that 
	\begin{align}\label{315}
&\frac{2p(p-1)}{\e}\|\Y^{\e}_t\|_{\H}^{2p-4}\Tr((\Y^{\e}_t\otimes\Y^{\e}_t)\sigma_2(\X^{\e}_t,\Y^{\e}_t)\Q_2\sigma_2^*(\X^{\e}_t,\Y^{\e}_t))\nonumber\\&\leq \frac{p\mu\lambda_1}{4\e}\|\Y^{\e}_t\|_{\H}^{2p}+\frac{C_p}{\mu\lambda_1\e}\|\X^{\e}_t\|_{\H}^{2p}+\frac{C_p}{\mu\lambda_1\e}.
	\end{align}
	Combining \eqref{313}-\eqref{315} and using it in \eqref{312}, we find 
	\begin{align}\label{3.16}
		\frac{\d}{\d t}	\E\left[\|\Y^{\e}_t\|_{\H}^{2p}\right]&=-\frac{p}{\e}\left(\mu{\lambda_1}-2L_g\right)\E\left[	\|\Y^{\e}_t\|_{\H}^{2p}\right]+\frac{C_p}{\mu\lambda_1\e}\E\left[\|\X^{\e}_t\|_{\H}^{2p}\right]+\frac{C_p}{\mu\lambda_1\e}, 
	\end{align}
	for a.e. $t\in[0,T]$. By the Assumption \ref{ass3.6} (A3), we know that $\mu\lambda_1>2L_g$ and an application of variation of constants formula gives 
	\begin{align}\label{3.17}
	\E\left[\|\Y^{\e}_t\|_{\H}^{2p}\right]&\leq\|\y\|_{\H}^{2p}e^{-\frac{p\gamma t}{\e}}+\frac{C_p}{\mu\lambda_1\e}\int_0^te^{-\frac{p\gamma(t-s)}{\e}}\left(1+\E\left[\|\X^{\e}_s\|_{\H}^{2p}\right]\right)\d s,
	\end{align}
	for all $t\in[0,T]$, where $\gamma=\left(\mu{\lambda_1}-2L_g\right)$. 
	
	\vskip 0.1 cm
		\textbf{Step 2:} \emph{$p^{\mathrm{th}}$-moment estimates for $\X^{\e}_t$}. 
Let us now obtain the $p^{\mathrm{th}}$-moment estimates for the process $\X^{\e}_{t}$ also. 	Applying the infinite dimensional It\^o formula to the process $\|\X^{\e}_{t}\|_{\H}^2$ (see \cite{MTM8}), we find 
	\begin{align}
	\|\X^{\e}_t\|_{\H}^2&=\|\x\|_{\H}^2-2\mu\int_0^t\|\X^{\e}_s\|_{\V}^2\d s-2\beta\int_0^t\|\X^{\e}_s\|_{\wi\L^{r+1}}^{r+1}\d s+2\int_0^t(f(\X^{\e}_s,\Y^{\e}_s),\X^{\e}_s)\d s\nonumber\\&\quad+\int_0^t\|\sigma_1(\X^{\e}_s)\|_{\mathcal{L}_{\Q}}^2\d s+2\int_0^t(\sigma_1(\X^{\e}_s)\d\W^{\Q_1}_s,\X^{\e}_s),	\end{align}
	for all $t\in[0,T]$, $\mathbb{P}$-a.s. Now, we apply It\^o's formula to the process $(\|\X^{\e}_\cdot\|_{\H}^2)^p$ to get 
	\begin{align}\label{3.19}
	&	\|\X^{\e}_t\|_{\H}^{2p}+2\mu p\int_0^t\|\X^{\e}_s\|_{\H}^{2p-2}\|\X^{\e}_s\|_{\V}^{2}\d s+2\beta p\int_0^t\|\X^{\e}_s\|_{\H}^{2p-2}\|\X^{\e}_s\|_{\wi\L^{r+1}}^{r+1}\d s\nonumber\\&=\|\x\|_{\H}^{2p}+2p\int_0^t\|\X^{\e}_s\|_{\H}^{2p-2}(f(\X^{\e}_s,\Y^{\e}_s),\X^{\e}_s)\d s +p\int_0^t\|\X^{\e}_s\|_{\H}^{2p-2}\|\sigma_1(\X^{\e}_s)\|_{\mathcal{L}_{\Q_1}}^2\d s\nonumber\\&\quad+2p(p-1)\int_0^t\|\X^{\e}_s\|_{\H}^{2p-4}\Tr((\X^{\e}_s\otimes\X^{\e}_s)\sigma_1(\X^{\e}_s)\Q_1\sigma_1^*(\X^{\e}_s))\d s\nonumber\\&\quad+2p\int_0^t\|\X^{\e}_s\|_{\H}^{2p-2}(\sigma_1(\X^{\e}_s)\d\W^{\Q_1}_s,\X^{\e}_s). 
	\end{align}
	Taking supremum over $[0,T]$ and then taking expectation in \eqref{3.19}, we obtain 
	\begin{align}\label{320}
	&\E\left[\sup_{t\in[0,T]}\|\X^{\e}_t\|_{\H}^{2p}+2\mu p\int_0^T\|\X^{\e}_t\|_{\H}^{2p-2}\|\X^{\e}_t\|_{\V}^{2}\d t+2\beta p\int_0^T\|\X^{\e}_t\|_{\H}^{2p-2}\|\X^{\e}_t\|_{\wi\L^{r+1}}^{r+1}\d t\right]\nonumber\\&\leq \|\x\|_{\H}^{2p}+2p\E\left[\int_0^T\|\X^{\e}_t\|_{\H}^{2p-2}|(f(\X^{\e}_t,\Y^{\e}_t),\X^{\e}_t)|\d t\right] +p\E\left[\int_0^T\|\X^{\e}_t\|_{\H}^{2p-2}\|\sigma_1(\X^{\e}_t)\|_{\mathcal{L}_{\Q_1}}^2\d t\right]\nonumber\\&\quad+2p(p-1)\E\left[\int_0^T\|\X^{\e}_t\|_{\H}^{2p-4}\Tr((\X^{\e}_t\otimes\X^{\e}_t)\sigma_1(\X^{\e}_t)\Q_1\sigma_1^*(\X^{\e}_t))\d t\right]\nonumber\\&\quad+2p\E\left[\sup_{t\in[0,T]}\left|\int_0^t\|\X^{\e}_s\|_{\H}^{2p-2}(\sigma_1(\X^{\e}_s)\d\W^{\Q_1}_s,\X^{\e}_s)\right|\right]\nonumber\\&\leq \|\x\|_{\H}^{2p}+C_pT+C_p\E\left[\int_0^T\|\X^{\e}_t\|_{\H}^{2p}\d t\right]+C_p\E\left[\int_0^T\|\Y^{\e}_t\|_{\H}^{2p}\d t\right]\nonumber\\&\quad+2p\E\left[\sup_{t\in[0,T]}\left|\int_0^t\|\X^{\e}_s\|_{\H}^{2p-2}(\sigma_1(\X^{\e}_s)\d\W^{\Q_1}_s,\X^{\e}_s)\right|\right], 
	\end{align}
where we used  calculations similar to \eqref{313}-\eqref{315}. Using Burkholder-Davis-Gundy inequality (see Theorem 1, \cite{BD} for the Burkholder-Davis-Gundy inequality for the case $p=1$ and Theorem 1.1, \cite{DLB} for the best constant, \cite{CMMR} for BDG inequality in infinite dimensions), we estimate the final term from the right hand side of the inequality \eqref{320} as 
\begin{align}\label{3.21}
&2p\E\left[\sup_{t\in[0,T]}\left|\int_0^t\|\X^{\e}_s\|_{\H}^{2p-2}(\sigma_1(\X^{\e}_s)\d\W^{\Q_1}_s,\X^{\e}_s)\right|\right]\nonumber\\&\leq C_p\E\left[\int_0^T\|\X^{\e}_t\|_{\H}^{4p-2}\|\sigma_1(\X^{\e}_t)\|_{\mathcal{L}_{\Q_1}}^2\d t\right]^{1/2}\nonumber\\&\leq C_p\E\left[\sup_{t\in[0,T]}\|\X^{\e}_t\|_{\H}^{2p}\left(\int_0^T\|\sigma_1(\X^{\e}_t)\|_{\mathcal{L}_{\Q_1}}^2\d t\right)^{1/2}\right]\nonumber\\&\leq\frac{1}{2}\E\left[\sup_{t\in[0,T]}\|\X^{\e}_t\|_{\H}^{2p}\right]+C_p\E\left[\left(\int_0^T\|\sigma_1(\X^{\e}_t)\|_{\mathcal{L}_{\Q_1}}^2\d t\right)^{p}\right]\nonumber\\&\leq\frac{1}{2}\E\left[\sup_{t\in[0,T]}\|\X^{\e}_t\|_{\H}^{2p}\right]+C_p\E\left[\int_0^T\|\X^{\e}_t\|_{\H}^{2p}\d t\right]+C_pT,
\end{align}
where we used the Assumption \ref{ass3.6} (A1). Using \eqref{3.17} and \eqref{3.21} in \eqref{320}, we deduce that 
\begin{align}\label{3.22}
&\E\left[\sup_{t\in[0,T]}\|\X^{\e}_t\|_{\H}^{2p}+4\mu p\int_0^T\|\X^{\e}_t\|_{\H}^{2p-2}\|\X^{\e}_t\|_{\V}^{2}\d t+4\beta p\int_0^T\|\X^{\e}_t\|_{\H}^{2p-2}\|\X^{\e}_t\|_{\wi\L^{r+1}}^{r+1}\d t\right]\nonumber\\&\leq  2\|\x\|_{\H}^{2p}+\frac{C_p\e}{\gamma}\|\y\|_{\H}^{2p}+C_pT+C_p\E\left[\int_0^T\|\X^{\e}_t\|_{\H}^{2p}\d t\right]\nonumber\\&\quad +\frac{C_p}{\mu\lambda_1\e}\int_0^T\int_0^te^{-\frac{p\gamma(t-s)}{\e}}\left(1+\E\left[\|\X^{\e}_s\|_{\H}^{2p}\right]\right)\d s\d t\nonumber\\&\leq 2\|\x\|_{\H}^{2p}+\frac{C_p}{\gamma}\|\y\|_{\H}^{2p}+C_p\left(1+\frac{1}{\mu\lambda_1\gamma}\right)T+C_p\left(1+\frac{1}{\mu\lambda_1\gamma}\right)\int_0^T\E\left(\sup_{s\in[0,t]}\|\X^{\e}_s\|_{\H}^{2p}\right)\d t,
\end{align}
since $\e\in(0,1)$. An application of Gronwall's inequality in \eqref{3.22} implies 
\begin{align}\label{3.23}
\E\left[\sup_{t\in[0,T]}\|\X^{\e}_t\|_{\H}^{2p}\right]\leq C_{p,\mu,\lambda_1,L_g,T}\left(1+\|\x\|_{\H}^{2p}+\|\y\|_{\H}^{2p}\right). 
\end{align}
Substituting \eqref{3.23} in \eqref{3.22} yields the estimate \eqref{3.7}. From \eqref{3.17}, we have 
	\begin{align}\label{3p17}
\E\left[\|\Y^{\e}_t\|_{\H}^{2p}\right]&\leq \|\y\|_{\H}^{2p}+\frac{C_p}{\mu\lambda_1\gamma}\left[1+\sup_{t\in[0,T]}\E\left(\|\X^{\e}_t\|_{\H}^{2p}\right)\right], 
\end{align}
for all $t\in[0,T]$. Using \eqref{3.23} in \eqref{3p17}, we finally obtain \eqref{38}. 
\end{proof}

	\section{Averaging Principle}\label{sec5}\setcounter{equation}{0}
In this section, we investigate the strong averaging principle for the coupled SCBF equations \eqref{3.6}. We exploit the  classical Khasminskii approach based on time discretization in the proof of the strong convergence of slow component   to the solution of the corresponding averaged SCBF equation.  Similar result for the coupled stochastic 2D Navier-Stokes equation is obtained in \cite{SLXS}. 

\begin{lemma}\label{lem3.8}
	For any $T>0$, $\e\in(0,1)$ and $\delta>0$ is small enough, there exists a constant $C_{\mu,\lambda_1,L_g,T}>0$ such that for any $\x,\y\in\H$, we have 
	\begin{align}\label{3.24}
	\E\left[\int_0^T\|\X_t^{\e}-\X^{\e}_{t(\delta)}\|_{\H}^2\d t \right]\leq C_{\mu,\lambda_1,L_g,T}\delta^{1/2}(1+\|\x\|_{\H}^3+\|\y\|_{\H}^3), 
	\end{align}
for $n=2$, $r\in[1,3)$ and 
	\begin{align}\label{324}
\E\left[\int_0^T\|\X_t^{\e}-\X^{\e}_{t(\delta)}\|_{\H}^2\d t \right]\leq C_{\mu,\lambda_1,L_g,T}\delta^{1/2}(1+\|\x\|_{\H}^2+\|\y\|_{\H}^2), 
\end{align}
for $n=2,3$, $r\in[3,\infty)$ ($2\beta\mu\geq 1$, for $n=r=3$), where $t(\delta):=\left[\frac{t}{\delta}\right]\delta$ and $[s]$ stands for the largest integer which is less than or equal $s$. 
\end{lemma}
\begin{proof}
	Using the estimate \eqref{3.7}, one can easily see that 
	\begin{align}\label{3.25}
	&	\E\left[\int_0^T\|\X_t^{\e}-\X^{\e}_{t(\delta)}\|_{\H}^2\d t\right] \nonumber\\&= 	\E\left[\int_0^{\delta}\|\X_t^{\e}-\x\|_{\H}^2\d t\right]+	\E\left[\int_{\delta}^T\|\X_t^{\e}-\X^{\e}_{t(\delta)}\|_{\H}^2\d t\right]\nonumber\\&\leq C_{\mu,\lambda_1,L_g,T}\left(1+\|\x\|_{\H}^{2}+\|\y\|_{\H}^{2}\right)\delta+	2\E\left[\int_{\delta}^T\|\X_t^{\e}-\X^{\e}_{t-\delta}\|_{\H}^2\d t\right]+2	\E\left[\int_{\delta}^T\|\X^{\e}_{t(\delta)}-\X^{\e}_{t-\delta}\|_{\H}^2\d t\right].
	\end{align}
	Let us first estimate the second term from the right hand side of the inequality \eqref{3.25}. Using the infinite dimensional It\^o formula applied to the process $\mathrm{Z}_{r}=\|\X^{\e}_{r}-\X^{\e}_{t-\delta}\|_{\H}^2$ over the interval $[t-\delta, t]$, we find 
	\begin{align}
	\|\X^{\e}_t-\X^{\e}_{t-\delta}\|_{\H}^2&=-2\mu\int_{t-\delta}^t\langle\A\X^{\e}_s,\X^{\e}_s-\X^{\e}_{t-\delta}\rangle\d s-2\int_{t-\delta}^t\langle\B(\X^{\e}_s),\X^{\e}_s-\X^{\e}_{t-\delta}\rangle\d s\nonumber\\&\quad-2\beta\int_{t-\delta}^t\langle\mathcal{C}(\X^{\e}_s),\X^{\e}_s-\X^{\e}_{t-\delta}\rangle\d s+2\int_{t-\delta}^t(f(\X^{\e}_s,\Y^{\e}_s),\X^{\e}_s-\X^{\e}_{t-\delta})\d s\nonumber\\&\quad+\int_{t-\delta}^t\|\sigma_1(\X^{\e}_s)\|_{\mathcal{L}_{\Q_1}}^2\d s+2\int_{t-\delta}^t(\sigma_1(\X^{\e}_s)\d\W_t^{\Q_1},\X^{\e}_s-\X^{\e}_{t-\delta})\d s\nonumber\\&=:\sum_{k=1}^6I_k(t).
	\end{align} 
	Using an integration by parts, H\"older's inequality,  Fubini's Theorem and \eqref{3.7}, we estimate $\E\left(\int_{\delta}^T|I_1(t)|\d t\right)$ as 
	\begin{align}\label{327}
\E\left(	\int_{\delta}^T|I_1(t)|\d t\right)&\leq 2\mu\E\left(\int_{\delta}^T\int_{t-\delta}^t\|\X^{\e}_s\|_{\V}\|\X^{\e}_s-\X^{\e}_{t-\delta}\|_{\V}\d s\d t\right)\nonumber\\&\leq 2\mu\left[\E\left(\int_{\delta}^T\int_{t-\delta}^t\|\X^{\e}_s\|_{\V}^2\d s\d t\right)\right]^{1/2} \left[\E\left(\int_{\delta}^T\int_{t-\delta}^t\|\X^{\e}_s-\X^{\e}_{t-\delta}\|_{\V}^2\d s\d t\right)\right]^{1/2}\nonumber\\&\leq 2\mu\left[\delta\E\left(\int_{0}^T\|\X^{\e}_t\|_{\V}^2\d t\right)\right]^{1/2}\left[2\delta\E\left(\int_{0}^T\|\X^{\e}_t\|_{\V}^2\d t\right)\right]^{1/2}\nonumber\\&\leq C_{\mu,\lambda_1,L_g,T}\delta\left(1+\|\x\|_{\H}^2+\|\y\|_{\H}^2\right). 
	\end{align}
For $n=2$ and $r\in[1,3)$,	using  H\"older's and Ladyzhenskaya's inequalities,  Fubini's Theorem and \eqref{3.7} (taking $p=2$), we estimate $\E\left(\int_{\delta}^T|I_2(t)|\d t\right)$ as 
	\begin{align}
&	\E\left(\int_{\delta}^T|I_2(t)|\d t\right)\nonumber\\&\leq 2\E\left(\int_{\delta}^T\int_{t-\delta}^t\|\X^{\e}_s\|_{\wi\L^4}^2\|\X^{\e}_s-\X^{\e}_{t-\delta}\|_{\V}\d s\d t\right)\nonumber\\&\leq 2\sqrt{2}\left[\E\left(\int_{\delta}^T\int_{t-\delta}^t\|\X^{\e}_s\|_{\H}^2\|\X^{\e}_s\|_{\V}^2\d s\d t\right)\right]^{1/2}\left[\E\left(\int_{\delta}^T\int_{t-\delta}^t\|\X^{\e}_s-\X^{\e}_{t-\delta}\|_{\V}^2\d s\d t\right)\right]^{1/2}\nonumber\\&\leq 2\sqrt{2}\left[\delta\E\left(\int_0^T\|\X^{\e}_t\|_{\H}^2\|\X^{\e}_t\|_{\V}^2\d t\right)\right]^{1/2} \left[2\delta\E\left(\int_{0}^T\|\X^{\e}_t\|_{\V}^2\d t\right)\right]^{1/2}\nonumber\\&\leq  C_{\mu,\lambda_1,L_g,T}\delta\left(1+\|\x\|_{\H}^3+\|\y\|_{\H}^3\right). 
	\end{align}
	For $n=2,3$ and $r\geq 3$ (take $2\beta\mu\geq1$, for $r=3$), we estimate $\E\left(\int_{\delta}^T|I_2(t)|\d t\right)$  using H\"older's inequality, interpolation inequality and \eqref{3.7} (by taking $p=\frac{r-1}{2}$) as 
	\begin{align}
	&	\E\left(\int_{\delta}^T|I_2(t)|\d t\right)\nonumber\\&\leq 2\E\left(\int_{\delta}^T\int_{t-\delta}^t\|\X^{\e}_s\|_{\wi\L^{r+1}}\|\X^{\e}_s\|_{\wi\L^{\frac{2(r+1)}{r-1}}}\|\X^{\e}_s-\X^{\e}_{t-\delta}\|_{\V}\d s\d t\right)\nonumber\\&\leq 2\E\left(\int_{\delta}^T\int_{t-\delta}^t\|\X^{\e}_s\|_{\H}^{\frac{r-3}{r-1}}\|\X^{\e}_s\|_{\wi\L^{r+1}}^{\frac{r+1}{r-1}}\|\X^{\e}_s-\X^{\e}_{t-\delta}\|_{\V}\d s\d t\right) \nonumber\\&\leq 2\left[\E\left(\int_{\delta}^T\int_{t-\delta}^t\|\X^{\e}_s\|_{\H}^{\frac{2(r-3)}{r-1}}\|\X^{\e}_s\|_{\wi\L^{r+1}}^{\frac{2(r+1)}{r-1}}\d s\d t\right)\right]^{1/2}\left[\E\left(\int_{\delta}^T\int_{t-\delta}^t\|\X^{\e}_s-\X^{\e}_{t-\delta}\|_{\V}^2\d s\d t\right)\right]^{1/2}\nonumber\\&\leq 2\left[\delta\E\left(\int_{0}^T\|\X^{\e}_t\|_{\H}^{\frac{2(r-3)}{r-1}}\|\X^{\e}_t\|_{\wi\L^{r+1}}^{\frac{2(r+1)}{r-1}}\d t\right)\right]^{1/2}\left[2\delta\E\left(\int_{0}^T\|\X^{\e}_t\|_{\V}^2\d t\right)\right]^{1/2}\nonumber\\&\leq 2\delta T^{\frac{r-3}{2(r-1)}}\left[\E\left(\int_{0}^T\|\X^{\e}_t\|_{\H}^{r-3}\|\X^{\e}_t\|_{\wi\L^{r+1}}^{r+1}\d t\right)\right]^{\frac{1}{r-1}}\left[2\E\left(\int_{0}^T\|\X^{\e}_t\|_{\V}^2\d t\right)\right]^{1/2}\nonumber\\&\leq C_{\mu,\lambda_1,L_g,T}\delta\left(1+\|\x\|_{\H}^2+\|\y\|_{\H}^2\right). 
	\end{align}
	Remember that $\frac{2r}{r+1}\leq 2$. It should  also be noted that if $x\geq 1$ and $a\leq b$ implies $x^a\leq x^b$ and if $0<x<1$, then $x^a\leq 1+x^b$. 	Once again using H\"older's inequality,  Fubini's Theorem and \eqref{3.7}, we estimate $\E\left(\int_{\delta}^T|I_3(t)|\d t\right)$ as 
		\begin{align}
		\E\left(\int_{\delta}^T|I_3(t)|\d t\right)&\leq 2\beta\E\left(\int_{\delta}^T\int_{t-\delta}^t\|\X^{\e}_s\|_{\wi\L^{r+1}}^r\|\X^{\e}_s-\X^{\e}_{t-\delta}\|_{\wi\L^{r+1}}\d s\d t\right)\nonumber\\&\leq 2\beta\left[\delta\E\left(\int_{0}^T\|\X^{\e}_t\|_{\wi\L^{r+1}}^{r+1}\d t\right)\right]^{\frac{r}{r+1}}\left[2^{r}\delta\E\left(\int_0^T\|\X^{\e}_t\|_{\wi\L^{r+1}}^{r+1}\d t\right)\right]^{\frac{1}{r+1}}\nonumber\\&\leq C_{\mu,\lambda_1,L_g,T}\delta\left(1+\|\x\|_{\H}^{\frac{2r}{r+1}}+\|\y\|_{\H}^{\frac{2r}{r+1}}\right)\nonumber\\&\leq C_{\mu,\lambda_1,L_g,T}\delta(1+\|\x\|_{\H}^2+\|\y\|_{\H}^2). 
	\end{align}
	We estimate $\E\left(\int_{\delta}^T|I_4(t)|\d t\right)$ using the Assumption \ref{ass3.6} (A1),   \eqref{3.7} and \eqref{38} as 
	\begin{align}
&	\E\left(\int_{\delta}^T|I_4(t)|\d t\right)\nonumber\\&\leq 2\E\left(\int_{\delta}^T\int_{t-\delta}^t\|f(\X^{\e}_s,\Y^{\e}_s)\|_{\H}\|\X^{\e}_s-\X^{\e}_{t-\delta}\|_{\H}\d s\d t\right)\nonumber\\&\leq 2\left[\E\left(\int_{\delta}^T\int_{t-\delta}^t\|f(\X^{\e}_s,\Y^{\e}_s)\|_{\H}^2\d s\d t\right)\right]^{1/2}\left[\E\left(\int_{\delta}^T\int_{t-\delta}^t\|\X^{\e}_s-\X^{\e}_{t-\delta}\|_{\H}^2\d s\d t\right)\right]^{1/2}\nonumber\\&\leq C\left[\delta\E\left(\int_{0}^T(1+\|\X^{\e}_t\|_{\H}^2+\|\Y^{\e}_t\|_{\H}^2)\d t\right)\right]^{1/2}\left[2\delta\E\left(\int_{0}^T\|\X^{\e}_t\|_{\H}^2\d t\right)\right]^{1/2}\nonumber\\&\leq C_{\mu,\lambda_1,L_g,T}\delta(1+\|\x\|_{\H}^2+\|\y\|_{\H}^2). 
	\end{align}
Once again using the Assumption \ref{ass3.6} (A1) and   \eqref{3.7}, we estimate $\E\left(\int_{\delta}^T|I_5(t)|\d t\right)$ as 
\begin{align}
\E\left(\int_{\delta}^T|I_5(t)|\d t\right)&\leq C\E\left(\int_{\delta}^T\int_{t-\delta}^t(1+\|\X^{\e}_s\|_{\H}^2)\d s\d t\right)\nonumber\\&\leq CT\delta \E\left[1+\sup_{t\in[0,T]}\|\X^{\e}_t\|_{\H}^2\right]\leq C_{\mu,\lambda_1,L_g,T}\delta(1+\|\x\|_{\H}^2+\|\y\|_{\H}^2). 
\end{align}
Finally, using the Burkholder-Davis-Gundy inequality, the Assumption \ref{ass3.6} (A1), Fubini's theorem and   \eqref{3.7},  we estimate $\E\left(\int_{\delta}^T|I_6(t)|\d t\right)$ as 
\begin{align}\label{333}
\E\left(\int_{\delta}^T|I_6(t)|\d t\right)&\leq C\int_{\delta}^T\E\left[\left(\int_{t-\delta}^t\|\sigma_1(\X^{\e}_s)\|_{\mathcal{L}_{\Q_1}}^2\|\X^{\e}_s-\X^{\e}_{t-\delta}\|_{\H}^2\d s\right)^{1/2}\right]\d t\nonumber\\&\leq CT^{1/2}\left[\E\left(\int_{\delta}^T\int_{t-\delta}^{t}\left(1+\|\X^{\e}_s\|_{\H}^2\right)\|\X^{\e}_s-\X^{\e}_{t-\delta}\|_{\H}^2\d s\d t\right)\right]^{1/2}\nonumber\\&\leq CT\delta^{1/2}\left[\E\left(\sup_{t\in[0,T]}\|\X_t\|_{\H}^2+\sup_{t\in[0,T]}\|\X_t\|_{\H}^4\right)\right]^{1/2}\nonumber\\&\leq C_{\mu,\lambda_1,L_g,T}\delta^{1/2}(1+\|\x\|_{\H}^2+\|\y\|_{\H}^2). 
\end{align}
Combining \eqref{327}-\eqref{333}, we deduce that 
\begin{align}\label{334}
\E\left[\int_{\delta}^T\|\X_t^{\e}-\X^{\e}_{t-\delta}\|_{\H}^2\d t\right]\leq  C_{\mu,\lambda_1,L_g,T}\delta^{1/2}(1+\|\x\|_{\H}^{\ell}+\|\y\|_{\H}^{\ell}), 
\end{align}
where $\ell=3$, for $n=2$,  $r\in[1,3)$, and $\ell=2$,  for $n=2,3$, $r\in[3,\infty)$  ($2\beta\mu\geq 1,$ for $n=3$). A similar argument leads to 
\begin{align}\label{335}
\E\left[\int_{\delta}^T\|\X_{t(\delta)}^{\e}-\X^{\e}_{t-\delta}\|_{\H}^2\d t\right]\leq  C_{\mu,\lambda_1,L_g,T}\delta^{1/2}(1+\|\x\|_{\H}^{\ell}+\|\y\|_{\H}^{\ell}).
\end{align}
Combining \eqref{3.25}, \eqref{334} and \eqref{335}, we obtain the required results \eqref{3.24} and \eqref{324}. 
\end{proof}

\subsection{Estimates of auxiliary process $\widehat{\Y}^{\e}_t$} We use the method proposed by Khasminskii, in \cite{RZK} to obtain the estimates for an auxiliary process. We introduce the auxiliary process $\widehat{\Y}^{\e}_t\in\H$ (see \eqref{3.37} below) and divide the interval $[0,T]$ into subintervals of  size $\delta$, where $\delta$ is a fixed positive number, which depends on $\e$ and it will be chosen later. Let us construct the process $\widehat{\Y}^{\e}_t$ with the initial value $\widehat{\Y}^{\e}_0=\Y^{\e}_0=\y$, and for any $k\in\mathbb{N}$ and $t\in[k\delta,\min\{(k+1)\delta,T\}]$ as 
\begin{align}\label{3.37}
\widehat{\Y}^{\e}_t&=\widehat{\Y}^{\e}_{k\delta}-\frac{\mu}{\e}\int_{k\delta}^t\A\widehat{\Y}^{\e}_s\d s-\frac{\beta}{\e}\int_{k\delta}^t\mathcal{C}(\widehat{\Y}^{\e}_s)\d s+\frac{1}{\e}\int_{k\delta}^tg(\X^{\e}_{k\delta},\widehat{\Y}^{\e}_s)\d s\nonumber\\&\quad+\frac{1}{\sqrt{\e}}\int_{k\delta}^t\sigma_2(\X^{\e}_{k\delta},\widehat{\Y}^{\e}_s)\d\W^{\Q_2}_s, \ \mathbb{P}\text{-a.s.}, 
\end{align}
which is equivalent to 
\begin{equation}\label{338}
\left\{
\begin{aligned}
\d\widehat{\Y}^{\e}_t&=-\frac{1}{\e}\left[\mu\A\widehat{\Y}^{\e}_t+\beta\mathcal{C}(\widehat{\Y}^{\e}_t)-g(\X^{\e}_{t(\delta)},\widehat{\Y}^{\e}_t)\right]\d t+\frac{1}{\sqrt{\e}}\sigma_2(\X^{\e}_{t(\delta)},\widehat{\Y}^{\e}_t)\d\W^{\Q_2}_t,\\
\widehat{\Y}^{\e}_0&=\y. 
\end{aligned}\right. 
\end{equation}
The following energy estimate satisfied by $\widehat{\Y}^{\e}_t$ can be proved in a similar way as in Lemma \ref{lem3.7}. 
\begin{lemma}\label{lem3.9}
	For any $\x,\y\in\H$, $T>0$ and $\e\in(0,1)$, there exists a constant $C_{\mu,\lambda_1,L_g,T}>0$ such that the strong solution $\widehat\Y^{\e}_{t}$ to the system \eqref{338} satisfies: 
	\begin{align}\label{341}
	\sup_{t\in[0,T]}	\E\left[\|\widehat\Y^{\e}_t\|_{\H}^{2}\right]\leq C_{\mu,\lambda_1,L_g,T}\left(1+\|\x\|_{\H}^{2}+\|\y\|_{\H}^{2}\right). 
	\end{align}
\end{lemma}
Our next aim is to establish an estimate on the difference between the processes $\Y^{\e}_t$ and $\widehat\Y^{\e}_t$. 
\begin{lemma}\label{lem3.10}
	For any $\x,\y\in\H$, $T>0$ and $\e\in(0,1)$, there exists a constant $C_{\mu,\lambda_1,L_g,T}>0$ such that 
	\begin{align}\label{3.42}
	\mathbb{E}\left(\int_0^T\|\Y^{\e}_t-\widehat\Y^{\e}_t\|_{\H}^2\d t\right)\leq C_{\mu,\lambda_1,L_g,T}\delta^{1/2}(1+\|\x\|_{\H}^{\ell}+\|\y\|_{\H}^{\ell}), 
	\end{align}
	where $\ell=3$, for $n=2$,  $r\in[1,3)$, and $\ell=2$,  for $n=2,3$, $r\in[3,\infty)$  ($2\beta\mu\geq 1,$ for $n=3$). 
\end{lemma}
\begin{proof}
	Let us define $\mathbf{U}^{\e}_t:=\Y^{\e}_t-\widehat\Y^{\e}_t$. Then $\mathbf{U}^{\e}_{t}$ satisfies the following It\^o stochastic differential: 
	\begin{equation}
	\left\{
	\begin{aligned}
	\d\U^{\e}_t&=-\frac{1}{\e}\left[\mu\A\U^{\e}_t+\beta(\mathcal{C}(\Y^{\e}_t)-\mathcal{C}(\widehat\Y^{\e}_t))-(g(\X^{\e}_{t},{\Y}^{\e}_t)-g(\X^{\e}_{t(\delta)},\widehat{\Y}^{\e}_t))\right]\d t\\&\quad+\frac{1}{\sqrt{\e}}\left[\sigma_2(\X^{\e}_{t},{\Y}^{\e}_t)-\sigma_2(\X^{\e}_{t(\delta)},\widehat{\Y}^{\e}_t)\right]\d\W_t^{\Q_2},\\
	\U^{\e}_0&=\mathbf{0}. 
	\end{aligned}
	\right. 
	\end{equation}
	An application of the infinite dimensional It\^o formula to the process $\|\U^{\e}_{t}\|_{\H}^2$ yields 
	\begin{align}\label{344}
	\|\U^{\e}_t\|_{\H}^2&=-\frac{2\mu}{\e}\int_0^t\|\U^{\e}_s\|_{\V}^2\d s-\frac{2\beta}{\e}\int_0^t\langle\mathcal{C}(\Y^{\e}_s)-\mathcal{C}(\widehat\Y^{\e}_s),\U^{\e}_s\rangle\d s\nonumber\\&\quad+\frac{2}{\e}\int_0^t(g(\X^{\e}_s,{\Y}^{\e}_s)-g(\X^{\e}_{s(\delta)},\widehat{\Y}^{\e}_s),\U^{\e}_s)\d s\nonumber\\&\quad+\frac{1}{\sqrt{\e}}\int_0^t([\sigma_2(\X^{\e}_s,{\Y}^{\e}_s)-\sigma_2(\X^{\e}_{s(\delta)},\widehat{\Y}^{\e}_s)]\d\W_s^{\Q_2},\U^{\e}_s)\nonumber\\&\quad+\frac{1}{\e}\int_0^t\|\sigma_2(\X^{\e}_s,{\Y}^{\e}_s)-\sigma_2(\X^{\e}_{s(\delta)},\widehat{\Y}^{\e}_s)\|_{\mathcal{L}_{\Q_2}}^2\d s, \ \mathbb{P}\text{-a.s.},
	\end{align}
	for all $t\in[0,T]$. Taking expectation in \eqref{344}, we obtain 
	\begin{align}
	\E\left[	\|\U^{\e}_t\|_{\H}^2\right]&=-\frac{2\mu}{\e}\int_0^t\E\left[\|\U^{\e}_s\|_{\V}^2\right]\d s-\frac{2\beta}{\e}\E\left[\int_0^t\langle\mathcal{C}(\Y^{\e}_s)-\mathcal{C}(\widehat\Y^{\e}_s),\U^{\e}_s\rangle\d s\right]\nonumber\\&\quad+\frac{2}{\e}\int_0^t\E\left[(g(\X^{\e}_s,{\Y}^{\e}_s)-g(\X^{\e}_{s(\delta)},\widehat{\Y}^{\e}_s),\U^{\e}_s)\right]\d s\nonumber\\&\quad+\frac{1}{\e}\int_0^t\E\left[\|\sigma_2(\X^{\e}_s,{\Y}^{\e}_s)-\sigma_2(\X^{\e}_{s(\delta)},\widehat{\Y}^{\e}_s)\|_{\mathcal{L}_{\Q_2}}^2\right]\d s. 
	\end{align}
	Thus, it is immediate that 
	\begin{align}\label{3.46}
	\frac{\d}{\d t}	\E\left[	\|\U^{\e}_t\|_{\H}^2\right]&=-\frac{2\mu}{\e}\E\left[\|\U^{\e}_t\|_{\V}^2\right]-\frac{2\beta}{\e}\E\left[\langle\mathcal{C}(\Y^{\e}_t)-\mathcal{C}(\widehat\Y^{\e}_t),\U^{\e}_t\rangle\right]\nonumber\\&\quad+\frac{2}{\e}\E\left[(g(\X^{\e}_t,{\Y}^{\e}_t)-g(\X^{\e}_{t(\delta)},\widehat{\Y}^{\e}_t),\U^{\e}_t)\right]\nonumber\\&\quad+\frac{1}{\e}\E\left[\|\sigma_2(\X^{\e}_t,{\Y}^{\e}_t)-\sigma_2(\X^{\e}_{t(\delta)},\widehat{\Y}^{\e}_t)\|_{\mathcal{L}_{\Q_2}}^2\right], 
	\end{align}
	for a.e. $t\in[0,T]$. Applying \eqref{214}, we find 
	\begin{align}\label{3p47}
&-\frac{2\beta}{\e}\langle\mathcal{C}(\Y^{\e}_t)-\mathcal{C}(\widehat\Y^{\e}_t),\U^{\e}_t\rangle\leq-\frac{\beta}{2^{r-2}\e}\|\U^{\e}_t\|_{\wi\L^{r+1}}^{r+1}.
	\end{align}
	Using the Assumption \ref{ass3.6} (A1), we get 
	\begin{align}\label{3.47}
\frac{2}{\e}(g(\X^{\e}_t,\widehat{\Y}^{\e}_t)-g(\X^{\e}_{t(\delta)},\widehat{\Y}^{\e}_t),\U^{\e}_t)&\leq \frac{2}{\e}\|g(\X^{\e}_t,\widehat{\Y}^{\e}_t)-g(\X^{\e}_{t(\delta)},\widehat{\Y}^{\e}_t)\|_{\H}\|\U^{\e}_t\|_{\H}\nonumber\\&\leq \frac{C}{\e}\|\X^{\e}_t-\X^{\e}_{t(\delta)}\|_{\H}\|\U^{\e}_t\|_{\H}+\frac{2L_g}{\e}\|\U^{\e}_t\|_{\H}^2\nonumber\\&\leq\left(\frac{\mu\lambda_1}{\e}+\frac{2L_g}{\e}\right)\|\U^{\e}_t\|_{\H}^2+\frac{C}{\mu\lambda_1\e}\|\X^{\e}_t-\X^{\e}_{t(\delta)}\|_{\H}^2. 
	\end{align} 
	Similarly, using the Assumption \ref{ass3.6} (A1), we obtain
	\begin{align}\label{3.48}
\frac{1}{\e}	\|\sigma_2(\X^{\e}_t,{\Y}^{\e}_t)-\sigma_2(\X^{\e}_{t(\delta)},\widehat{\Y}^{\e}_t)\|_{\mathcal{L}_{\Q_2}}^2&\leq \frac{C}{\e}\|\X^{\e}_t-\X^{\e}_{t(\delta)}\|_{\H}^2+\frac{L_{\sigma_2}^2}{\e}\|\U^{\e}_t\|_{\H}^2. 
	\end{align}
	Combining \eqref{3p47}-\eqref{3.48} and substituting it in \eqref{3.46}, we deduce that 
	\begin{align}\label{3.49}
		\frac{\d}{\d t}	\E\left[	\|\U^{\e}_t\|_{\H}^2\right]&\leq-\frac{1}{\e}(\mu\lambda_1-2L_g-L_{\sigma_2}^2)\E\left[\|\U^{\e}_t\|_{\H}^2\right]+\frac{C}{\e}\left(1+\frac{1}{\mu\lambda_1}\right)\E\left[\|\X^{\e}_t-\X^{\e}_{t(\delta)}\|_{\H}^2\right]. 
	\end{align}
	Using the Assumption \ref{ass3.6} (A3) and variation of constants  formula, from \eqref{3.49}, we infer that 
	\begin{align}
	\E\left[\|\U^{\e}_t\|_{\H}^2\right]\leq \frac{C}{\e}\left(1+\frac{1}{\mu\lambda_1}\right)\int_0^te^{-\frac{\kappa}{\e}(t-s)}\E\left[\|\X^{\e}_s-\X^{\e}_{s(\delta)}\|_{\H}^2\right]\d s,
	\end{align}
	where $\kappa=\mu\lambda_1-2L_g-L_{\sigma_2}^2>0$. Applying Fubini's theorem, for any $T>0$, we have 
	\begin{align}
	\E\left[\int_0^T\|\U^{\e}_t\|_{\H}^2\d t\right]&\leq \frac{C}{\e}\left(1+\frac{1}{\mu\lambda_1}\right)\int_0^T\int_0^te^{-\frac{\kappa}{\e}(t-s)}\E\left[\|\X^{\e}_s-\X^{\e}_{s(\delta)}\|_{\H}^2\right]\d s\d t\nonumber\\&=\frac{C}{\e}\left(1+\frac{1}{\mu\lambda_1}\right)\E\left[\int_0^T\|\X^{\e}_s-\X^{\e}_{s(\delta)}\|_{\H}^2\left(\int_s^Te^{-\frac{\kappa}{\e}(t-s)}\d t\right)\d s\right]\nonumber\\&\leq\frac{C}{\kappa}\left(1+\frac{1}{\mu\lambda_1}\right)\E\left[\int_0^T\|\X^{\e}_t-\X^{\e}_{t(\delta)}\|_{\H}^2\d t\right]\nonumber\\&\leq C_{\mu,\lambda_1,L_g,T}\delta^{1/2}(1+\|\x\|_{\H}^{\ell}+\|\y\|_{\H}^{\ell}), 
	\end{align}
	using Lemma \ref{lem3.8} (see \eqref{3.24} and \eqref{324}), which completes the proof. 
\end{proof}

\subsection{Frozen equation} The frozen equation associated with the fast motion for fixed slow component  $\x\in\H$ is given by 
\begin{equation}\label{3.74} 
\left\{
\begin{aligned}
\d\Y_t&=-[\mu\A\Y_t+\beta\mathcal{C}(\Y_t)-g(\x,\Y_t)]\d t+\sigma_2(\x,\Y_t)\d\bar{\W}_t^{\Q_2},\\
\Y_0&=\y,
\end{aligned}\right.
\end{equation}
where $\bar{\W}_t^{\Q_2}$ is a $\Q_2$-Wiener process independent of $\W_{t}^{\Q_1}$ and $\W_{t}^{\Q_2}$. From the Assumption \ref{ass3.6}, we know that $g(\x,\cdot)$ and $\sigma_2(\x,\cdot)$ are Lipschitz continuous. Thus, one can show that for any fixed $\x\in\H$ and any initial data $\y\in\H$, there exists a unique strong solution $\Y_{t}^{\x,\y}\in\mathrm{L}^2(\Omega;\mathrm{L}^{\infty}(0,T;\H)\cap\mathrm{L}^2(0,T;\V))\cap\mathrm{L}^{r+1}(\Omega;\mathrm{L}^{r+1}(0,T;\wi\L^{r+1}))$ to the system \eqref{3.74} with a continuous modification with paths in $\C([0,T];\H)\cap\mathrm{L}^2(0,T;\V)\cap\mathrm{L}^{r+1}(0,T;\wi\L^{r+1})$, $\mathbb{P}$-a.s.  A proof of this can be obtained in a same way as in Theorem 3.7, \cite{MTM8} by making use of the monotonicity property of the linear and nonlinear operators (see Lemmas \ref{thm2.2}-\ref{lem2.8}) as well as a stochastic generalization of the Minty-Browder technique (localized version for the case $n=2$ and $r\in[1,3]$). Furthermore, the strong solution satisfies the infinite dimensional It\^o formula (see \eqref{3.76} below). For the stochastic reaction-diffusion equation ($\beta=0$ in \eqref{3.74}), the following result is available in the work \cite{SC}.

\begin{proposition}\label{prop3.12}
	For any given $\x,\y\in\H$, there exists a unique invariant measure for the system \eqref{3.74}. Furthermore, there exists a constant $C_{\mu,\lambda_1,L_g}>0$ and $\zeta>0$ such that for any Lipschitz function $\varphi:\H\to\R$, we have 
	\begin{align}\label{393}
	\left|\mathrm{P}_t^{\x}\varphi(\y)-\int_{\H}\varphi(\z)\nu^{\x}(\d\z)\right|\leq C_{\mu,\lambda_1,L_g}(1+\|\x\|_{\H}+\|\y\|_{\H})e^{-\frac{\zeta t}{2}}\|\varphi\|_{\mathrm{Lip}(\H)},
	\end{align}
	where $\zeta=2\mu\lambda_1-2L_g-L_{\sigma_2}^2>0$ and $ \|\varphi\|_{\mathrm{Lip}(\H)}=\sup\limits_{\x,\y\in\H}\frac{|\varphi(\x)-\varphi(\y)|}{\|\x-\y\|_{\H}}$. 
\end{proposition}
\begin{proof} 
  Let $\mathrm{P}_t^{\x}$ be the transition semigroup associated with the process $\Y_t^{\x,\y}$, that is, for any bounded measurable function $\varphi$ on $\H$, we have 
\begin{align}
\mathrm{P}_t^{\x}\varphi(\y)=\E\left[\varphi(\Y_t^{\x,\y})\right], \ \y\in\H \ \text{ and }\ t>0. 
\end{align}
An application of the infinite dimensional It\^o formula to the process $\|\Y_{t}^{\x,\y}\|_{\H}^2$ yields 
\begin{align}\label{3.76}
&\|\Y_t^{\x,\y}\|_{\H}^2+2\mu\int_0^t\|\Y_s^{\x,\y}\|_{\V}^2\d s+2\beta\int_0^t\|\Y_s^{\x,\y}\|_{\wi\L^{r+1}}^{r+1}\d s\nonumber\\&=\|\y\|_{\H}^2+2\int_0^t(g(\x,\Y_s^{\x,\y}),\Y_s^{\x,\y})\d s+2\int_0^t(\sigma_2(\x,\Y_s)\d\bar{\W}_s^{\Q_2},\Y_s^{\x,\y})+\int_0^t\|\sigma_2(\x,\Y_s^{\x,\y})\|_{\mathcal{L}_{\Q_2}}^2\d s, 
\end{align}
$\mathbb{P}$-a.s., for all $t\in[0,T]$. Taking expectation in \eqref{3.76}, we find 
\begin{align}\label{3.77}
\E\left[\|\Y_t^{\x,\y}\|_{\H}^2\right]&=\|\y\|_{\H}^2-2\mu\E\left[\int_0^t\|\Y_s^{\x,\y}\|_{\V}^2\d s\right]-2\beta\E\left[\int_0^t\|\Y_s^{\x,\y}\|_{\wi\L^{r+1}}^{r+1}\d s\right]\nonumber\\&\quad+2\E\left[\int_0^t(g(\x,\Y_s^{\x,\y}),\Y_s^{\x,\y})\d s\right]+\E\left[\int_0^t\|\sigma_2(\x,\Y_s^{\x,\y})\|_{\mathcal{L}_{\Q_2}}^2\d s\right],
\end{align}
for all $t\in[0,T]$. Thus, we have 
\begin{align}\label{3.78}
\frac{\d}{\d t}\E\left[\|\Y_t^{\x,\y}\|_{\H}^2\right]&= -2\mu\E\left[\|\Y_t^{\x,\y}\|_{\V}^2\right]-2\beta\E\left[\|\Y_t^{\x,\y}\|_{\wi\L^{r+1}}^{r+1}\right]+2\E\left[(g(\x,\Y_t^{\x,\y}),\Y_t^{\x,\y})\right]\nonumber\\&\quad+\E\left[\|\sigma_2(\x,\Y_t^{\x,\y})\|_{\mathcal{L}_{\Q_2}}^2\right] \nonumber\\&\leq -(\mu\lambda_1-2L_g)\E\left[\|\Y_t^{\x,\y}\|_{\H}^2\right]+\frac{C}{\mu\lambda_1}(1+\|\x\|_{\H}^2),
\end{align}
where we used the Assumption \ref{ass3.6} (A1) and (A2). For $\gamma=\mu\lambda_1-2L_g>0$, an application of Gronwall's inequality in \eqref{3.78} yields
\begin{align}\label{3.79}
\E\left[\|\Y_t^{\x,\y}\|_{\H}^2\right]\leq e^{-\gamma t}\|\y\|_{\H}^2+ \frac{C}{\mu\lambda_1\gamma}(1+\|\x\|_{\H}^2)\leq C_{\mu,\lambda_1,L_g}(1+\|\x\|_{\H}^2+e^{-\gamma t}\|\y\|_{\H}^2),
\end{align} 
for all $t\in[0,T]$. Furthermore, from \eqref{3.77}, we also obtain 
\begin{align}\label{380}
&2\mu\E\left[\int_0^t\|\Y_s^{\x,\y}\|_{\V}^2\d s\right]\nonumber\\&\leq\|\y\|_{\H}^2+C_{L_g,L_{\sigma_2}}(1+\|\x\|_{\H}^2)t+(C_{\mu,\lambda_1,L_g}+2L_g+L_{\sigma_2}^2)\E\left[\int_0^t\|\Y_s^{\x,\y}\|_{\H}^2\d s\right]\nonumber\\&\leq \|\y\|_{\H}^2+C_{L_g,L_{\sigma_2}}(1+\|\x\|_{\H}^2)t+(C_{\mu,\lambda_1,L_g}+2L_g+L_{\sigma_2}^2)\int_0^t(1+\|\x\|_{\H}^2+e^{-\gamma s}\|\y\|_{\H}^2)\d s\nonumber\\&\leq C_{\mu,\lambda_1,L_g,L_{\sigma_2}}(\|\y\|_{\H}^2+(1+\|\x\|_{\H}^2)t),
\end{align}
for $t\geq 0$. 

Let us denote by $\mathcal{M}(\H)$ as the space of all probability measures defined on $(\H,\mathscr{B}(\H))$. For any $n\in\N$, we define the Krylov-Bogoliubov measure as 
\begin{align*}
\nu_n^{\x}:=\frac{1}{n}\int_{\mathbf{0}}^n\delta_{\mathbf{0}}\mathrm{P}_t^{\x}\d t, \ n\geq 1, 
\end{align*}
where $\delta_{\mathbf{0}}$ is the Dirac measure at $\mathbf{0}$. Then, it can be easily seen that $\nu_n^{\x}$ is a probability measure such that for any $\varphi\in\mathscr{B}(\H)$, we have $$\int_{\H}\varphi\d\nu_n^{\x}=\frac{1}{n}\int_0^n\mathrm{P}_t^{\x}\varphi(\mathbf{0})\d t.$$ 
We know that the embedding of $\V\subset\H$ is compact and hence for any $R\in(0,\infty)$, the set $K_R:=\{\x\in\H:\|\x\|_{\V}\leq R\}$ is relatively compact in $\H$. Thus,  using Markov's inequality, it is immediate from \eqref{380} that 
\begin{align}
\nu_n^{\x}(K_R^c)&=\frac{1}{n}\int_0^n\mathbb{P}\left\{\|\Y_0^{\x,\mathbf{0}}\|_{\H}>R\right\}\d s\leq \frac{1}{nR^2}\int_0^n\E\left[\|\Y_s^{\x,\mathbf{0}}\|_{\V}^2\right]\d s\nonumber\\&\leq \frac{C_{\mu,\lambda_1,L_g,L_{\sigma_2}}\|\y\|_{\H}^2}{n_0 R^2\mu}+ \frac{C_{\mu,\lambda_1,L_g,L_{\sigma_2}}}{R^2\mu}(1+\|\x\|_{\H}^2), \ \text{ for all }\ n\geq n_0, \nonumber\\&\to 0\ \text{ as } \ R\to\infty, 
\end{align} 
which implies that the sequence $\{\nu_n^{\x}\}_{n\geq 1}$ is tight. Hence, using Prokhorov's theorem, there exists a probability measure $\nu^{\x}$ and a subsequence $\{\nu_{n_k}^{\x}\}_{k\in\mathbb{N}}$ such that $\nu_{n_k}^{\x} \xrightarrow{w}\nu^{\x}$, as $k\to\infty$. One can easily show that $\nu^{\x}$ is an invariant probability measure for $\{\mathrm{P}_t^{\x}\}_{t\geq 0}$. Thus, by the Krylov-Bogoliubov theorem (or by a  result of Chow and Khasminskii, see \cite{CHKH})  $\nu^{\x}$ results to be an invariant measure for the transition semigroup $(\mathrm{P}_t^{\x})_{t\geq 0}$,  defined by 	$\mathrm{P}_t^{\x}\varphi(\x)=\E\left[\varphi(\Y_t^{\x,\y})\right],$ for all $\varphi\in\C_b(\H)$.

Let us now prove the uniqueness of invariant measure. For $\y_1,\y_2\in\H$, we have 
\begin{equation}\label{3.80} 
\left\{
\begin{aligned}
\d\mathbf{V}_t&=-\{\mu\A\mathbf{V}_t+\beta[\mathcal{C}(\Y_{t}^{\x,\y_1})-\mathcal{C}(\Y_{t}^{\x,\y_2})]-[g(\x,\Y_t^{\x,\y_1})-g(\x,\Y_t^{\x,\y_2})]\}\d t\\&\quad+[\sigma_2(\x,\Y_t^{\x,\y_1})-\sigma_2(\x,\Y_t^{\x,\y_2})]\d\bar{\W}_t^{\Q_2},\\
\mathbf{V}_0&=\y_1-\y_2,
\end{aligned}\right.
\end{equation}
where $\mathbf{V}_{t}=\Y_{t}^{\x,\y_1}-\Y_{t}^{\x,\y_2}$. An application of the infinite dimensional  It\^o formula to the process $\|\mathbf{V}_{t}\|_{\H}^2$  yields 
\begin{align}
\E\left[\|\mathbf{V}_{t}\|_{\H}^2\right]&=\|\mathbf{V}_{0}\|_{\H}^2-2\mu\E\left[\int_0^t\|\mathbf{V}_{s}\|_{\V}^2\d s\right]-2\beta\E\left[\int_0^t\langle\mathcal{C}(\Y_{s}^{\x,\y_1})-\mathcal{C}(\Y_{s}^{\x,\y_2}),\mathbf{V}_{s}\rangle\d s\right]\nonumber\\&\quad+2\E\left[\int_0^t(g(\x,\Y_s^{\x,\y_1})-g(\x,\Y_s^{\x,\y_2}),\mathbf{V}_s)\d s\right]\nonumber\\&\quad+\E\left[\int_0^t\|\sigma_2(\x,\Y_s^{\x,\y_1})-\sigma_2(\x,\Y_s^{\x,\y_2})\|_{\mathcal{L}_{\Q_2}}^2\d s\right].
\end{align}
Thus, using the Assumption \ref{ass3.6} (A1)-(A2), and \eqref{214},  it is immediate that 
\begin{align}\label{3.82}
\frac{\d}{\d t}\E\left[\|\mathbf{V}_{t}\|_{\H}^2\right]&=-2\mu\E\left[\|\mathbf{V}_{t}\|_{\V}^2\right]-2\beta\E\left[\langle\mathcal{C}(\Y_{t}^{\x,\y_1})-\mathcal{C}(\Y_{t}^{\x,\y_2}),\mathbf{V}_{t}\rangle\right]\nonumber\\&\quad+2\E\left[(g(\x,\Y_t^{\x,\y_1})-g(\x,\Y_t^{\x,\y_2}),\mathbf{V}_t)\right]\nonumber\\&\quad+\E\left[\|\sigma_2(\x,\Y_t^{\x,\y_1})-\sigma_2(\x,\Y_t^{\x,\y_2})\|_{\mathcal{L}_{\Q_2}}^2\right]\nonumber\\&\leq-\left(2\mu\lambda_1-2L_g-L_{\sigma_2}^2\right)\E\left[\|\mathbf{V}_{t}\|_{\H}^2\right]-\frac{\beta}{2^{r-2}}\E\left[\|\mathbf{V}_{t}\|_{\wi\L^{r+1}}^{r+1}\right]\nonumber\\&\leq-\left(2\mu\lambda_1-2L_g-L_{\sigma_2}^2\right)\E\left[\|\mathbf{V}_{t}\|_{\H}^2\right]. 
\end{align}
An application of Gronwall's inequality in \eqref{3.82} gives the following exponential stability result: 
\begin{align}\label{385}
\sup_{\x\in\H}\E\left[\|\Y_t^{\x,\y_1}-\Y_t^{\x,\y_2}\|_{\H}^2\right]\leq e^{-\zeta t}\|\y_1-\y_2\|_{\H}^2, \ \text{ for all } \ t\geq 0,
\end{align} 
where $\zeta=2\mu\lambda_1-2L_g-L_{\sigma_2}^2>0$. Furthermore, using the invariance of $\nu^{\x}$ and \eqref{3.79}, we obtain 
\begin{align*}
\int_{\H}\|\y\|_{\H}^2\nu^{\x}(\d\y)=\int_{\H}\E\left[\|\Y_t^{\x,\y}\|_{\H}^2\right]\nu^{\x}(\d\y)\leq {C_{\mu,\lambda_1,L_g}}(1+\|\x\|_{\H}^2)+e^{-\gamma t} \int_{\H}\|\y\|_{\H}^2\nu^{\x}(\d\y),
\end{align*}
for any $t>0$. Taking $t$ large enough so that $e^{-\gamma t}<1$ and
\begin{align}\label{386}
\int_{\H}\|\y\|_{\H}^2\nu^{\x}(\d\y)\leq{C_{\mu,\lambda_1,L_g}}(1+\|\x\|_{\H}^2). 
\end{align}
Thus, for any Lipschitz function $\varphi:\H\to\R$, using the invariance of $\nu^{\x}$, \eqref{385} and \eqref{386}, we have 
\begin{align}
	\left|\mathrm{P}_t^{\x}\varphi(\y)-\int_{\H}\varphi(\z)\nu^{\x}(\d\z)\right|&=\left|\int_{\H}\E\left[\varphi(\Y_t^{\x,\y})-\varphi(\Y_t^{\x,\z})\right]\nu^{\x}(\d\z)\right|\nonumber\\&\leq \|\varphi\|_{\mathrm{Lip}}\int_{\H}\E\left[\|\Y_t^{\x,\y}-\Y_t^{\x,\z}\|_{\H}\right]\nu^{\x}(\d\z)\nonumber\\&\leq \|\varphi\|_{\mathrm{Lip}}\int_{\H}\left(\E\left[\|\Y_t^{\x,\y}-\Y_t^{\x,\z}\|_{\H}^2\right]\right)^{1/2}\nu^{\x}(\d\z)\nonumber\\&\leq \|\varphi\|_{\mathrm{Lip}}e^{-\frac{\zeta t}{2}}\int_{\H}\|\y-\z\|_{\H}\nu^{\x}(\d\z)\nonumber\\&\leq C_{\mu,\lambda_1,L_g} \|\varphi\|_{\mathrm{Lip}}(1+\|\x\|_{\H}+\|\y\|_{\H})e^{-\frac{\zeta t}{2}}, 
\end{align}
and by the density of $\text{Lip}(\H)$ in $\C_b (\H)$, we obtain the above result  for every $\varphi\in \C_b (\H)$. In particular, the uniqueness of invariant measure $\nu^{\x}$ follows.
\end{proof} 
The interested readers are referred to see \cite{GDJZ,ADe,FFBM,MHJC}, etc for more details on the invariant measures and ergodicity for the infinite dimensional dynamical systems and stochastic Navier-Stokes equations. We need the following lemma in the sequel. 
\begin{lemma}
	There exists a constant $C>0$ such that for any $\x_1,\x_2,\y\in\H$, we have
	\begin{align}\label{394}
	\sup_{t\geq 0}\E\left[\|\Y_t^{\x_1,\y}-\Y_t^{\x_2,\y}\|_{\H}^2\right]\leq C_{\mu,\lambda_1,L_g,L_{\sigma_2}}\|\x_1-\x_2\|_{\H}^2. 
	\end{align}
\end{lemma}
\begin{proof}
	We know that $\mathbf{W}_{t}:=\Y_{t}^{\x_1,\y}-\Y_{t}^{\x_2,\y}$ satisfies the following system:
	\begin{equation}
	\left\{
\begin{aligned}
\d\mathbf{W}_t&=-[\mu\A\mathbf{W}_t+\beta(\mathcal{C}(\Y_{t}^{\x_1,\y})-\mathcal{C}(\Y_{t}^{\x_2,\y}))]\d t+[g(\x_1,\Y_t^{\x_1,\y})-g(\x_2,\Y_t^{\x_2,\y})]\d t\\&\quad +[\sigma_2(\x_1,\Y_t^{\x_1,\y})-\sigma_2(\x_2,\Y_t^{\x_2,\y})]\d\bar\W_t^{\Q_2},\\
\mathbf{W}_0&=\mathbf{0}. 
\end{aligned}\right. 
	\end{equation}
	Applying the infinite dimensional It\^o formula to the process $\|\mathbf{W}_t\|_{\H}^2$, we find 
	\begin{align}
&	\E\left[\|\mathbf{W}_t\|_{\H}^2\right]+2\mu\E\left[\int_0^t\|\mathbf{W}_s\|_{\V}^2\d s\right]\nonumber\\&=-2\beta\E\left[\int_0^t\langle\mathcal{C}(\Y_{s}^{\x_1,\y})-\mathcal{C}(\Y_{s}^{\x_2,\y}),\mathbf{W}_s\rangle\d s\right]\nonumber\\&\quad+2\E\left[\int_0^t([g(\x_1,\Y_s^{\x_1,\y})-g(\x_2,\Y_s^{\x_2,\y})],\mathbf{W}_s)\d s\right]\nonumber\\&\quad+\E\left[\int_0^t\|\sigma_2(\x_1,\Y_s^{\x_1,\y})-\sigma_2(\x_2,\Y_s^{\x_2,\y})\|_{\mathcal{L}_{\Q_2}}^2\d s\right],
	\end{align}
	so that we have 
	\begin{align}
	\frac{\d}{\d t}	\E\left[\|\mathbf{W}_t\|_{\H}^2\right]&=-2\mu\E\left[\|\mathbf{W}_t\|_{\V}^2\right]-2\beta\E\left[\langle\mathcal{C}(\Y_{t}^{\x_1,\y})-\mathcal{C}(\Y_{t}^{\x_2,\y}),\mathbf{W}_t\rangle\right]\nonumber\\&\quad+2\E\left[([g(\x_1,\Y_t^{\x_1,\y})-g(\x_2,\Y_t^{\x_2,\y})],\mathbf{W}_t)\right]\nonumber\\&\quad+\E\left[\|\sigma_2(\x_1,\Y_t^{\x_1,\y})-\sigma_2(\x_2,\Y_t^{\x_2,\y})\|_{\mathcal{L}_{\Q_2}}^2\right]\nonumber\\&\leq-2\mu\lambda_1\E\left[\|\mathbf{W}_t\|_{\H}^2\right]-\frac{\beta}{2^{r-2}}\E\left[\|\mathbf{W}_t\|_{\wi\L^{r+1}}^{r+1}\right]\nonumber\\&\quad+2\E\left[C\|\x_1-\x_2\|_{\H}\|\mathbf{W}_t\|_{\H}+L_g\|\mathbf{W}_t\|_{\H}^2\right]\nonumber\\&\quad+\E\left[\left(C\|\x_1-\x_2\|_{\H}+L_{\sigma_2}\|\mathbf{W}_t\|_{\H}\right)^2\right]\nonumber\\&\leq -\left(\mu\lambda_1-2L_g-2L_{\sigma_2}^2\right)\E\left[\|\mathbf{W}_t\|_{\H}^2\right]+C\left(1+\frac{1}{\mu\lambda_1}\right)\|\x_1-\x_2\|_{\H}^2,
	\end{align}
	for a.e. $t\in[0,T]$. Using variation of constants formula, we further have 
	\begin{align}
	\E\left[\|\mathbf{W}_t\|_{\H}^2\right]\leq C\left(1+\frac{1}{\mu\lambda_1}\right)\|\x_1-\x_2\|_{\H}^2\int_0^te^{-\xi(t-s)}\d s\leq \frac{C}{\xi}\left(1+\frac{1}{\mu\lambda_1}\right)\|\x_1-\x_2\|_{\H}^2,
	\end{align}
for any $t>0$, where $\xi=\mu\lambda_1-2L_g-2L_{\sigma_2}^2>0$ and the estimate \eqref{394} follows. 
\end{proof}
\subsection{Averaged equation}
Let us now consider the following averaged equation: 
\begin{equation}\label{3.85}
\left\{
\begin{aligned}
\d\bar{\X}_t&=-[\mu\A\bar{\X}_t+\B(\bar{\X}_t)+\beta\mathcal{C}(\bar{\X}_t)]\d t+\bar{f}(\bar{\X}_t)\d t+\sigma_1(\bar{\X}_t)\d\W^{\Q_1}_t, \\ \bar{\X}_0&=\x, 
\end{aligned}\right. 
\end{equation}
with the average 
\begin{align*}
\overline{f}(\x)=\int_{\H}f(\x,\y)\nu^{\x}(\d\y), \ \x\in\H, 
\end{align*}
where $\nu^{\x}$ is the unique invariant measure for the system \eqref{3.74}. Since $f(\cdot,\cdot)$ is Lipschitz, one can show that $\bar{f}(\cdot)$ is Lipschitz in the following way. Using the Assumption \ref{ass3.6} (A1), \eqref{393} and \eqref{394}, we have 
\begin{align*}
&\|\bar{f}(\x_1)-\bar{f}(\x_2)\|_{\H}\nonumber\\&=\left\|\int_{\H}f(\x_1,\z)\mu^{\x_1}(\d\z)-\int_{\H}f(\x_1,\z)\mu^{\x_1}(\d\z)\right\|_{\H}\nonumber\\&\leq\left\|\int_{\H}f(\x_1,\z)\mu^{\x_1}(\d\z)-\E\left[f(\x_1,\Y_t^{\x_1,\y})\right]\right\|_{\H}+\left\|\E\left[f(\x_2,\Y_t^{\x_2,\y})\right]-\int_{\H}f(\x_1,\z)\mu^{\x_1}(\d\z)\right\|_{\H}\nonumber\\&\quad+\|\E\left[f(\x_1,\Y_t^{\x_1,\y})\right]-\E\left[f(\x_2,\Y_t^{\x_2,\y})\right]\|_{\H}\nonumber\\&\leq C_{\mu,\lambda_1,L_g}(1+\|\x_1\|_{\H}+\|\x_2\|_{\H}+\|\y\|_{\H})e^{-\frac{\zeta t}{2}}+C\left(\|\x_1-\x_2\|_{\H}+\E\left[\|\Y_t^{\x_1,\y}-\Y_t^{\x_2,\y}\|_{\H}\right]\right)\nonumber\\&\leq C_{\mu,\lambda_1,L_g}(1+\|\x_1\|_{\H}+\|\x_2\|_{\H}+\|\y\|_{\H})e^{-\frac{\zeta t}{2}}+C_{\mu,\lambda_1,L_g,L_{\sigma_2}}\|\x_1-\x_2\|_{\H}. 
\end{align*}
Taking $t\to\infty$, we get 
\begin{align}\label{3p93}
&\|\bar{f}(\x_1)-\bar{f}(\x_2)\|_{\H}\leq C_{\mu,\lambda_1,L_g,L_{\sigma_2}}\|\x_1-\x_2\|_{\H}. 
\end{align}
Since the system \eqref{3.85} is a Lipschitz perturbation of the SCBF equations, one can show that it has a unique strong solution as in Theorem 3.7, \cite{MTM8}. The following result can be proved in a similar way as in Lemmas \ref{lem3.7} and \ref{lem3.8}.

\begin{lemma}\label{lem3.13}
	For any $T>0$, $p\geq 1$, there exists a constant $C_{p,\mu,\lambda_1,L_g,L_{\sigma_2},T}>0$ such that for any $\x\in\H$, we have 
	\begin{align}\label{3.86}
&	\E\left[\sup_{t\in[0,T]}\|\bar{\X}_t\|_{\H}^{2p}\right]+2p\mu\E\left[\int_0^T\|\bar{\X}_t\|_{\H}^{2p-2}\|\X_t\|_{\V}^2\d t\right]+2p\beta\E\left[\int_0^T\|\bar{\X}_t\|_{\H}^{2p-2}\|\X_t\|_{\wi\L^{r+1}}^{r+1}\d t\right]\nonumber\\&\leq C_{p,\mu,\lambda_1,L_g,L_{\sigma_2},T}\left(1+\|\x\|_{\H}^{2p}\right), 
	\end{align}
	and 
		\begin{align}\label{3.87}
	\E\left[\int_0^T\|\bar\X_t-\bar\X_{t(\delta)}\|_{\H}^2\d t \right]\leq C_{p,\mu,\lambda_1,L_g,L_{\sigma_2},T}\delta^{1/2}(1+\|\x\|_{\H}^{\ell}+\|\y\|_{\H}^{\ell}), 
	\end{align}
where $\ell=3$, for $n=2$,  $r\in[1,3)$, and $\ell=2$,  for $n=2,3$, $r\in[3,\infty)$  ($2\beta\mu\geq 1,$ for $n=r=3$). 
\end{lemma}
Our next aim is to establish an estimate for the process ${\X}^{\e}_{t}-\bar{\X}_{t}$. For $n=2$ and $r\in[1,3]$, we need to construct a stopping time to obtain an estimate for $\|\X^{\e}_t-\bar\X_t\|_{\H}^2$. For fixed $\e\in(0,1)$ and $R>0$, we define 
\begin{align}\label{stop}
\tau^{\e}_R:=\inf_{t\geq 0}\left\{t:\int_0^t\|\X^{\e}_s\|_{\V}^2\d s\geq R\right\}. 
\end{align}
We see in the next Lemma that such a stopping time is not needed for $r\in(3,\infty)$ for $n=2$ and $r\in[3,\infty)$ for $n=3$ ($2\beta\mu\geq 1,$ for $r=n=3$). 

\begin{lemma}\label{lem3.14}
	For any $\x,\y\in\H$, $T>0$ and $\e\in(0,1)$, the following estimate holds: 
	\begin{align}\label{389}
	\E\left[\sup_{t\in[0,T\wedge\tau_R^{\e}]}\|\X_t^{\e}-\bar{\X}_t\|_{\H}^2\right]\leq C_{R,\mu,\lambda_1,L_g,L_{\sigma_2},T}\left(1+\|\x\|_{\H}^3+\|\y\|_{\H}^3\right)\left(\frac{\e}{\delta}+\frac{\e^{1/2}}{\delta^{1/2}}+\delta^{1/4}\right), 
	\end{align}	
	for $n=2$ and $r\in[1,3]$, and 
	\begin{align}\label{390}
	\E\left[\sup_{t\in[0,T]}\|\X_t^{\e}-\bar{\X}_t\|_{\H}^2\right]\leq C_{\mu,\lambda_1,L_g,L_{\sigma_2},T}\left(1+\|\x\|_{\H}^2+\|\y\|_{\H}^2\right)\left(\frac{\e}{\delta}+\frac{\e^{1/2}}{\delta^{1/2}}+\delta^{1/4}\right), 
	\end{align}	
	for $n=2$, $r\in(3,\infty)$ and $n=3$, $r\in[3,\infty)$ ($2\beta\mu\geq 1,$ for $r=3$). 
\end{lemma}
\begin{proof}
	Let us denote $\Z_{t}^{\e}:=\X_{t}^{\e}-\bar{\X}_{t}$. Then $\Z^{\e}_{t}$ satisfies the following  It\^o stochastic differential: 
		\begin{equation}\label{3.91}
	\left\{
	\begin{aligned}
	\d\Z_t^{\e}&=-[\mu\A\Z_t^{\e}+(\B(\X_t^{\e})-\B(\bar\X_t))+\beta(\mathcal{C}(\X_t^{\e})-\mathcal{C}(\bar\X_t))]\d t\\&\quad+[f(\X_t^{\e},\Y_t^{\e})-\bar f(\bar\X_{t})]\d t+[\sigma_1(\X_t^{\e})-\sigma_1(\bar\X_t)]\d\W_t^{\Q_1},\\
	\Z_0^{\e}&=\mathbf{0}. 
	\end{aligned}\right. 
	\end{equation}
	We first consider the case $n=2$ and $r\in[1,3]$. An application of the infinite dimensional It\^o formula to the process $\|\Z^{\e}_{t}\|_{\H}^2$ yields 
	\begin{align}\label{3.92}
&	\|\Z^{\e}_{t}\|_{\H}^2\nonumber\\&=-2 \mu\int_0^t\|\Z^{\e}_s\|_{\V}^2\d s-2\int_0^t\langle (\B(\X_s^{\e})-\B(\bar\X_s)),\Z^{\e}_s\rangle\d s-2\beta\int_0^t\langle\mathcal{C}(\X_s^{\e})-\mathcal{C}(\bar\X_s),\Z^{\e}_s\rangle\d s\nonumber\\&\quad+2\int_0^t(f(\X_s^{\e},\Y_s^{\e})-\bar f(\bar\X_{s}),\Z^{\e}_s)\d s+\int_0^t\|\sigma_1(\X_s^{\e})-\sigma_1(\bar\X_s)\|_{\mathcal{L}_{\Q_1}}^2\d s\nonumber\\&\quad+2\int_0^t(\sigma_1(\X_s^{\e})-\sigma_1(\bar\X_s)\d\W^{\Q_1}_s,\Z^{\e}_s)\nonumber\\&=- 2 \mu\int_0^t\|\Z^{\e}_s\|_{\V}^2\d s-2\int_0^t\langle (\B(\X_s^{\e})-\B(\bar\X_s)),\Z^{\e}_s\rangle\d s-2\beta\int_0^t\langle\mathcal{C}(\X_s^{\e})-\mathcal{C}(\bar\X_s),\Z^{\e}_s\rangle\d s\nonumber\\&\quad+2\int_0^t(\bar f(\X_s^{\e})-\bar f(\bar\X_{s}),\Z^{\e}_s)\d s+2\int_0^t(f(\X_s^{\e},\Y_s^{\e})-\bar f(\X_{s}^{\e})-f(\X_{s(\delta)}^{\e},\widehat\Y_s^{\e})+\bar f(\X_{s(\delta)}^{\e}),\Z^{\e}_s)\d s\nonumber\\&\quad+2\int_0^t(f(\X_{s(\delta)}^{\e},\widehat\Y_s^{\e})-\bar f(\X_{s(\delta)}^{\e}),\Z^{\e}_s-\Z^{\e}_{s(\delta)})\d s+2\int_0^t(f(\X_{s(\delta)}^{\e},\widehat\Y_s^{\e})-\bar f(\X_{s(\delta)}^{\e}),\Z^{\e}_{s(\delta)})\d s\nonumber\\&\quad +\int_0^t\|\sigma_1(\X_s^{\e})-\sigma_1(\bar\X_s)\|_{\mathcal{L}_{\Q_1}}^2\d s+2\int_0^t(\sigma_1(\X_s^{\e})-\sigma_1(\bar\X_s)\d\W^{\Q_1}_s,\Z^{\e}_s),\  \mathbb{P}\text{-a.s.},
	\end{align}
	for all $t\in[0,T]$. Using H\"older's, Ladyzhenskaya's and Young's inequalities, we estimate $2|\langle(\B(\X_s^{\e})-\B(\bar\X_s)),\Z^{\e}_s\rangle|$ as 
	\begin{align}\label{3.93}
	2|	\langle(\B(\X_s^{\e})-\B(\bar\X_s)),\Z^{\e}_s\rangle|&= 2|\langle\B(\Z_s^{\e},\X_s^{\e}),\Z_s^{\e}\rangle |\leq 2\|\X_s^{\e}\|_{\V}\|\Z_s^{\e}\|_{\wi\L^4}^2\leq 2\sqrt{2}\|\X_s^{\e}\|_{\V}\|\Z_s^{\e}\|_{\H}\|\Z_s^{\e}\|_{\V}\nonumber\\&\leq\mu\|\Z_s^{\e}\|_{\V}^2+\frac{2}{\mu}\|\X_{s}^{\e}\|_{\V}^2\|\Z^{\e}_s\|_{\H}^2. 
	\end{align}
	Using \eqref{2.23} and \eqref{a215}, we know that 
	\begin{align}
	-2\beta\langle\mathcal{C}(\X_s^{\e})-\mathcal{C}(\bar\X_s),\Z^{\e}_s\rangle\leq-\frac{\beta}{2^{r-2}}\|\Z^{\e}_s\|_{\wi\L^{r+1}}^{r+1},
	\end{align}
	for $r\in[1,\infty)$. Using \eqref{3p93}, H\"older's and Young's inequalities, we estimate $2(\bar f(\X_s^{\e})-\bar f(\bar\X_{s}),\Z^{\e}_s)$ as 
	\begin{align}
	2(\bar f(\X_s^{\e})-\bar f(\bar\X_{s}),\Z^{\e}_s)&\leq 2\|\bar f(\X_s^{\e})-\bar f(\bar\X_{s})\|_{\H}\|\Z^{\e}_s\|_{\H}\leq C_{\mu,\lambda_1,L_g,L_{\sigma_2}}\|\Z_s^{\e}\|_{\H}^2. 
	\end{align}
	Similarly, we estimate $2(f(\X_s^{\e},\Y_s^{\e})-\bar f(\X_{s}^{\e})-f(\X_{s(\delta)}^{\e},\widehat\Y_s^{\e})+\bar f(\X_{s(\delta)}^{\e}),\Z^{\e}_s)$ as 
	\begin{align}
	&2(f(\X_s^{\e},\Y_s^{\e})-\bar f(\X_{s}^{\e})-f(\X_{s(\delta)}^{\e},\widehat\Y_s^{\e})+\bar f(\X_{s(\delta)}^{\e}),\Z^{\e}_s)\nonumber\\&\leq 2\left(\|f(\X_s^{\e},\Y_s^{\e})-f(\X_{s(\delta)}^{\e},\widehat\Y_s^{\e})\|_{\H}+\|\bar f(\X_{s(\delta)}^{\e})-\bar f(\X_{s}^{\e})\|_{\H}\right)\|\Z^{\e}_s\|_{\H}\nonumber\\&\leq\|\Z^{\e}_s\|_{\H}^2+C_{\mu,\lambda_1,L_g,L_{\sigma_2}}(\|\X_{s}^{\e}-\X_{s(\delta)}^{\e}\|_{\H}^2+\|\Y_s^{\e}-\widehat\Y_s^{\e}\|_{\H}^2).
	\end{align}
	Once again using the Assumption \ref{ass3.6} (A1), we get 
	\begin{align}
	\|\sigma_1(\X_s^{\e})-\sigma_1(\bar\X_s)\|_{\mathcal{L}_{\Q_1}}^2\leq C\|\Z^{\e}_s\|_{\H}^2. 
	\end{align}
	Using the Assumption \ref{ass3.6} (A1), H\"older's and Young's inequalities, we estimate the term  $2\int_0^t(f(\X_{s(\delta)}^{\e},\widehat\Y_s^{\e})-\bar f(\X_{s(\delta)}^{\e}),\Z^{\e}_s-\Z^{\e}_{s(\delta)})\d s$ as 
	\begin{align}\label{3.97}
	&2\int_0^t(f(\X_{s(\delta)}^{\e},\widehat\Y_s^{\e})-\bar f(\X_{s(\delta)}^{\e}),\Z^{\e}_s-\Z^{\e}_{s(\delta)})\d s\nonumber\\&\leq 2\int_0^t \|f(\X_{s(\delta)}^{\e},\widehat\Y_s^{\e})-\bar f(\X_{s(\delta)}^{\e})\|_{\H}\|\Z^{\e}_s-\Z^{\e}_{s(\delta)}\|_{\H}\d s\nonumber\\&\leq 4\left(\int_0^t(\|f(\X_{s(\delta)}^{\e},\widehat\Y_s^{\e})\|_{\H}^2+\|\bar f(\X_{s(\delta)}^{\e})\|_{\H}^2)\d s\right)^{1/2}\left(\int_0^t(\|\X^{\e}_{s}-\X^{\e}_{s(\delta)}\|_{\H}^2+\|\bar\X_s-\bar\X_{s(\delta)}\|_{\H}^2)\d s\right)^{1/2}\nonumber\\&\leq C\left(\int_0^t(1+\|\X_{s(\delta)}^{\e}\|_{\H}^2+\|\widehat\Y_s^{\e}\|_{\H}^2)\d s\right)^{1/2}\left(\int_0^t(\|\X^{\e}_{s}-\X^{\e}_{s(\delta)}\|_{\H}^2+\|\bar\X_s-\bar\X_{s(\delta)}\|_{\H}^2)\d s\right)^{1/2}.
	\end{align}
	Combining \eqref{3.93}-\eqref{3.97}, and then substituting it in \eqref{3.92}, we obtain 
	\begin{align}\label{3.99} 
	&	\|\Z^{\e}_{t}\|_{\H}^2+\mu\int_0^t\|\Z^{\e}_s\|_{\V}^2\d s+\frac{\beta}{2^{r-2}}\int_0^t\|\Z^{\e}_s\|_{\wi\L^{r+1}}^{r+1}\d s\nonumber\\&\leq \frac{2}{\mu}\int_0^t\|\X_{s}^{\e}\|_{\V}^2\|\Z^{\e}_s\|_{\H}^2\d s+C_{\mu,\lambda_1,L_g,L_{\sigma_2}}\int_0^t\|\Z^{\e}_s\|_{\H}^2\d s\nonumber\\&\quad+C_{\mu,\lambda_1,L_g,L_{\sigma_2}}\int_0^t\|\X_{s}^{\e}-\X_{s(\delta)}^{\e}\|_{\H}^2\d s+C_{\mu,\lambda_1,L_g,L_{\sigma_2}}\int_0^t\|\Y_s^{\e}-\widehat\Y_s^{\e}\|_{\H}^2\d s\nonumber\\&\quad+C\left(\int_0^t(1+\|\X_{s(\delta)}^{\e}\|_{\H}^2+\|\widehat\Y_s^{\e}\|_{\H}^2)\d s\right)^{1/2}\left(\int_0^t(\|\X^{\e}_{s}-\X^{\e}_{s(\delta)}\|_{\H}^2+\|\bar\X_s-\bar\X_{s(\delta)}\|_{\H}^2)\d s\right)^{1/2} \nonumber\\&\quad +2\int_0^t(f(\X_{s(\delta)}^{\e},\widehat\Y_s^{\e})-\bar f(\X_{s(\delta)}^{\e}),\Z^{\e}_{s(\delta)})\d s+2\int_0^t(\sigma_1(\X_s^{\e})-\sigma_1(\bar\X_s)\d\W^{\Q_1}_s,\Z^{\e}_s), 
	\end{align}
	for all $t\in[0,T]$. An application of the Gronwall inequality in \eqref{3.99} gives 
	\begin{align}\label{3100} 
	\sup_{t\in[0,T\wedge\tau_R^{\e}]}	\|\Z^{\e}_{t}\|_{\H}^2&\leq C_{R,\mu,\lambda_1,L_g,L_{\sigma_2},T}\bigg\{\int_0^{T\wedge\tau_R^{\e}}\|\X_{s}^{\e}-\X_{s(\delta)}^{\e}\|_{\H}^2\d s+\int_0^{T\wedge\tau_R^{\e}}\|\Y_s^{\e}-\widehat\Y_s^{\e}\|_{\H}^2\d s\nonumber\\&\quad+\left(\int_0^{T\wedge\tau_R^{\e}}(1+\|\X_{s(\delta)}^{\e}\|_{\H}^2+\|\widehat\Y_s^{\e}\|_{\H}^2)\d s\right)^{1/2}\nonumber\\&\qquad\times\left(\int_0^{T\wedge\tau_R^{\e}}(\|\X^{\e}_{s}-\X^{\e}_{s(\delta)}\|_{\H}^2+\|\bar\X_s-\bar\X_{s(\delta)}\|_{\H}^2)\d s\right)^{1/2}\nonumber\\&\quad+	\sup_{t\in[0,T\wedge\tau_R^{\e}]}\left|\int_0^t(f(\X_{s(\delta)}^{\e},\widehat\Y_s^{\e})-\bar f(\X_{s(\delta)}^{\e}),\Z^{\e}_{s(\delta)})\d s\right|\nonumber\\&\quad+	\sup_{t\in[0,T\wedge\tau_R^{\e}]}\left|\int_0^t(\sigma_1(\X_s^{\e})-\sigma_1(\bar\X_s)\d\W^{\Q_1}_s,\Z^{\e}_s)\right|\bigg\}, 
	\end{align}
	where we used the definition of stopping time given in \eqref{stop} also. Taking expectation on both sides of \eqref{3100} and then using Lemmas \ref{lem3.7}-\ref{lem3.13} (see \eqref{3.24}, \eqref{3.42} and \eqref{3.87}), we obtain 
	\begin{align}\label{3101}
	\E\left[	\sup_{t\in[0,T\wedge\tau_R^{\e}]}	\|\Z^{\e}_{t}\|_{\H}^2\right]&\leq C_{R,\mu,\lambda_1,L_g,L_{\sigma_2},T}\left(1+\|\x\|_{\H}^3+\|\y\|_{\H}^3\right)\delta^{1/4}\nonumber\\&\quad+C_{R,\mu,\lambda_1,L_g,L_{\sigma_2},T}\E\left[\sup_{t\in[0,T\wedge\tau_R^{\e}]}\left|\int_0^t(f(\X_{s(\delta)}^{\e},\widehat\Y_s^{\e})-\bar f(\X_{s(\delta)}^{\e}),\Z^{\e}_{s(\delta)})\d s\right|\right]\nonumber\\&\quad+C_{R,\mu,\lambda_1,L_g,L_{\sigma_2},T}\E\left[\sup_{t\in[0,T\wedge\tau_R^{\e}]}\left|\int_0^t(\sigma_1(\X_s^{\e})-\sigma_1(\bar\X_s)\d\W^{\Q_1}_s,\Z^{\e}_s)\right|\right]. 
	\end{align}
	Using the Burkholder-Davis-Gundy inequality and Assumption \ref{ass3.6} (A1), we estimate the final term from the right hand side of the inequality \eqref{3101} as 
	\begin{align}\label{3102}
&C_{R,\mu,\lambda_1,L_g,L_{\sigma_2},T}\E\left[\sup_{t\in[0,T\wedge\tau_R^{\e}]}\left|\int_0^t(\sigma_1(\X_s^{\e})-\sigma_1(\bar\X_s)\d\W^{\Q_1}_s,\Z^{\e}_s)\right|\right]\nonumber\\&\leq C_{R,\mu,\lambda_1,L_g,L_{\sigma_2},T} \E\left[\int_0^{T\wedge\tau_R^{\e}}\|\sigma_1(\X_s^{\e})-\sigma_1(\bar\X_s)\|_{\mathcal{L}_{\Q_1}}^2\|\Z^{\e}_s\|_{\H}^2\d s\right]^{1/2}\nonumber\\&\leq C_{R,\mu,\lambda_1,L_g,L_{\sigma_2},T}\E\left[\sup_{s\in[0,T\wedge\tau_R^{\e}]}\|\Z^{\e}_s\|_{\H}\left(\int_0^{T\wedge\tau_R^{\e}}\|\sigma_1(\X_s^{\e})-\sigma_1(\bar\X_s)\|_{\mathcal{L}_{\Q_1}}^2\d s\right)^{1/2}\right]\nonumber\\&\leq\frac{1}{2}\E\left[\sup_{s\in[0,T\wedge\tau_R^{\e}]}\|\Z^{\e}_s\|_{\H}^2\right]+C_{R,\mu,\lambda_1,L_g,L_{\sigma_2},T}\E\left[\int_0^{T\wedge\tau_R^{\e}}\|\Z^{\e}_s\|_{\H}^2\d s\right]. 
	\end{align}
	Using \eqref{3102} in \eqref{3101}, we get 
	\begin{align}\label{3103} 
	&	\E\left[	\sup_{t\in[0,T\wedge\tau_R^{\e}]}	\|\Z^{\e}_{t}\|_{\H}^2\right]\nonumber\\&\leq C_{R,\mu,\lambda_1,L_g,L_{\sigma_2},T}\left(1+\|\x\|_{\H}^3+\|\y\|_{\H}^3\right)\delta^{1/4}+C_{R,\mu,\lambda_1,L_g,L_{\sigma_2},T}\E\left[\int_0^{T\wedge\tau_R^{\e}}\|\Z^{\e}_s\|_{\H}^2\d s\right]\nonumber\\&\quad+C_{R,\mu,\lambda_1,L_g,L_{\sigma_2},T}\E\left[\sup_{t\in[0,T\wedge\tau_R^{\e}]}\left|\int_0^t(f(\X_{s(\delta)}^{\e},\widehat\Y_s^{\e})-\bar f(\X_{s(\delta)}^{\e}),\Z^{\e}_{s(\delta)})\d s\right|\right].
	\end{align}
	An application of Gronwall's inequality in \eqref{3103} yields 
	\begin{align}\label{3104}
	&	\E\left[	\sup_{t\in[0,T\wedge\tau_R^{\e}]}	\|\Z^{\e}_{t}\|_{\H}^2\right]\nonumber\\&\leq C_{R,\mu,\lambda_1,L_g,L_{\sigma_2},T}\left(1+\|\x\|_{\H}^3+\|\y\|_{\H}^3\right)\delta^{1/4}+C_{R,\mu,\lambda_1,L_g,L_{\sigma_2},T}\E\left[\sup_{t\in[0,T]}|I(t)|\right], 
	\end{align}
	where $I(t)=\int_0^t(f(\X_{s(\delta)}^{\e},\widehat\Y_s^{\e})-\bar f(\X_{s(\delta)}^{\e}),\Z^{\e}_{s(\delta)})\d s$. 
	Let us now estimate the final term from the right hand side of the inequality \eqref{3104}.  We follow similar arguments given in Lemma 3.8, \cite{SLXS} to obtain the required result. For the sake of completeness, we give a proof here. Note that 
	\begin{align}\label{3105}
	|I(t)|&= \left|\sum_{k=0}^{[t/\delta]-1}\int_{k\delta}^{(k+1)\delta}(f(\X_{k\delta}^{\e},\widehat\Y_s^{\e})-\bar f(\X_{k\delta}^{\e}),\X^{\e}_{k\delta}-\bar\X_{s(\delta)})\d s\right.\nonumber\\&\quad+\left.\int_{t(\delta)}^t(f(\X_{s(\delta)}^{\e},\widehat\Y_s^{\e})-\bar f(\X_{s(\delta)}^{\e}),\X^{\e}_{s(\delta)}-\bar\X_{s(\delta)})\d s\right|=:|I_1(t)+I_2(t)|. 
	\end{align}
	Using the Assumption \ref{ass3.6} (A1), we estimate $\E\left[\sup\limits_{t\in[0,T]}|I_2(t)|\right]$ as 
	\begin{align}
&	\E\left[\sup\limits_{t\in[0,T]}|I_2(t)|\right]\nonumber\\&\leq\E\left[\sup\limits_{t\in[0,T]}\int_{t(\delta)}^t\|f(\X_{s(\delta)}^{\e},\widehat\Y_s^{\e})-\bar f(\X_{s(\delta)}^{\e})\|_{\H}\|\X^{\e}_{s(\delta)}-\bar\X_{s(\delta)}\|_{\H}\d s\right]\nonumber\\&\leq C\left[\E\left(\sup_{t\in[0,T]}\|\X^{\e}_t-\bar{\X}_t\|_{\H}^2\right)\right]^{1/2}\left[\E\left(\sup_{t\in[0,T]}\left|\int_{t(\delta)}^t\left(1+\|\X^{\e}_{s(\delta)}\|_{\H}+\|\widehat\Y^{\e}_s\|_{\H}\right)\d s\right|^2\right)\right]^{1/2}\nonumber\\&\leq C\delta^{1/2} \left[\E\left(\sup_{t\in[0,T]}(\|\X^{\e}_t\|_{\H}^2+\|\bar{\X}_t\|_{\H}^2)\right)\right]^{1/2}\left[\E\left(\int_0^T\left(1+\|\X^{\e}_{s(\delta)}\|_{\H}^2+\|\widehat\Y^{\e}_s\|_{\H}^2\right)\d s\right)\right]^{1/2}\nonumber\\&\leq C_{\mu,\lambda_1,L_g,T}(1+\|\x\|_{\H}^2+\|\y\|_{\H}^2)\delta^{1/2}, 
	\end{align} 
where we used \eqref{3.7} and \eqref{3.86}. 	Next, we estimate the term $\E\left[\sup\limits_{t\in[0,T]}|I_1(t)|\right]$ as 
	\begin{align}\label{3107}
&	\E\left[\sup\limits_{t\in[0,T]}|I_1(t)|\right]\nonumber\\&\leq \E\left[\sum_{k=0}^{[T/\delta]-1}\left|\int_{k\delta}^{(k+1)\delta}(f(\X_{k\delta}^{\e},\widehat\Y_s^{\e})-\bar f(\X_{k\delta}^{\e}),\X^{\e}_{k\delta}-\bar\X_{k\delta})\d s\right|\right]\nonumber\\&\leq\left[\frac{T}{\delta}\right]\max_{0\leq k\leq [T/\delta]-1}\E\left[\left|\int_{k\delta}^{(k+1)\delta}(f(\X_{k\delta}^{\e},\widehat\Y_s^{\e})-\bar f(\X_{k\delta}^{\e}),\X^{\e}_{k\delta}-\bar\X_{k\delta})\d s\right|\right]\nonumber\\&\leq\frac{C_T}{\delta}\max_{0\leq k\leq [T/\delta]-1}\left[\E\left(\|\X^{\e}_{k\delta}-\bar\X_{k\delta}\|_{\H}^2\right)\right]^{1/2}\left[\E\left\|\int_{k\delta}^{(k+1)\delta}f(\X_{k\delta}^{\e},\widehat\Y_s^{\e})-\bar f(\X_{k\delta}^{\e})\d s\right\|_{\H}^2\right]^{1/2}\nonumber\\&\leq \frac{C_T\e}{\delta}\max_{0\leq k\leq [T/\delta]-1}\left[\E\left(\|\X^{\e}_{k\delta}-\bar\X_{k\delta}\|_{\H}^2\right)\right]^{1/2}\left[\E\left\|\int_{0}^{\frac{\delta}{\e}}f(\X_{k\delta}^{\e},\widehat\Y_{s\e+k\delta}^{\e})-\bar f(\X_{k\delta}^{\e})\d s\right\|_{\H}^2\right]^{1/2}\nonumber\\&\leq\frac{C_{\mu,\lambda_1,L_g,T}(1+\|\x\|_{\H}+\|\y\|_{\H})\e}{\delta}\max_{0\leq k\leq [T/\delta]-1}\left[\int_0^{\frac{\delta}{\e}}\int_r^{\frac{\delta}{\e}}\Phi_k(s,r)\d s\d r\right]^{1/2},
	\end{align}
	where for any $0\leq r\leq s\leq \frac{\delta}{\e}$, 
	\begin{align}\label{3108}
	\Phi_k(s,r):=\E\left[(f(\X_{k\delta}^{\e},\widehat\Y_{s\e+k\delta}^{\e})-\bar f(\X_{k\delta}^{\e}),f(\X_{k\delta}^{\e},\widehat\Y_{r\e+k\delta}^{\e})-\bar f(\X_{k\delta}^{\e}))\right].
	\end{align}
	Now, for any $\e>0$, $s>0$ and $\mathscr{F}_s$-measurable $\H$-valued random variables $\X$ and $\Y$, let $\{\Y_t^{\e,s,\X,\Y}\}_{t\geq s}$ be the unique strong solution of the following It\^o stochastic differential equation: 
	\begin{equation}
	\left\{
	\begin{aligned}
	\d\wi\Y_t&=-\frac{1}{\e}[\mu\A\wi\Y_t+\beta\mathcal{C}(\wi\Y_t)-g(\X,\wi\Y_t)]\d t+\frac{1}{\sqrt{\e}}\sigma_2(\X,\wi\Y_t)\d\W_t^{\Q_2},\\
	\wi\Y_s&=\Y. 
	\end{aligned}
	\right.
	\end{equation}
	Then, from the construction of the  process $\widehat\Y_{t}^{\e}$ (see \eqref{3.37}), for any $t\in[k\delta,(k+1)\delta]$ with $k\in\mathbb{N}$, we have 
	\begin{align}
\widehat\Y_t^{\e}=\widetilde\Y_t^{\e,k\delta,\X_{k\delta}^{\e},\widehat\Y_{k\delta}^{\e}}, \ \mathbb{P}\text{-a.s.}
	\end{align}
	Using this fact in \eqref{3108}, we infer that 
	\begin{align}
	\Phi_k(s,r)&=\E\left[(f(\X_{k\delta}^{\e},\widetilde\Y_{s\e+k\delta}^{\e,k\delta,\X_{k\delta}^{\e},\widehat\Y_{k\delta}^{\e}})-\bar f(\X_{k\delta}^{\e}),f(\X_{k\delta}^{\e},\widetilde\Y_{r\e+k\delta}^{\e,k\delta,\X_{k\delta}^{\e},\widehat\Y_{k\delta}^{\e}})-\bar f(\X_{k\delta}^{\e}))\right]\nonumber\\&=\int_{\Omega}\E\left[(f(\X_{k\delta}^{\e},\widetilde\Y_{s\e+k\delta}^{\e,k\delta,\X_{k\delta}^{\e},\widehat\Y_{k\delta}^{\e}})-\bar f(\X_{k\delta}^{\e}),f(\X_{k\delta}^{\e},\widetilde\Y_{r\e+k\delta}^{\e,k\delta,\X_{k\delta}^{\e},\widehat\Y_{k\delta}^{\e}})-\bar f(\X_{k\delta}^{\e}))\Big|{\mathscr{F}_{k\delta}}\right]\d\mathbb{P}(\omega)\nonumber\\&=\int_{\Omega}\E\left[(f(\X_{k\delta}^{\e}(\omega),\widetilde\Y_{s\e+k\delta}^{\e,k\delta,\X_{k\delta}^{\e}(\omega),\widehat\Y_{k\delta}^{\e}(\omega)})-\bar f(\X_{k\delta}^{\e}(\omega)),\right.\nonumber\\&\qquad\left.f(\X_{k\delta}^{\e}(\omega),\widetilde\Y_{r\e+k\delta}^{\e,k\delta,\X_{k\delta}^{\e}(\omega),\widehat\Y_{k\delta}^{\e}(\omega)})-\bar f(\X_{k\delta}^{\e}(\omega)))\right]\mathbb{P}(\d\omega),
	\end{align}
	where in the final step we used the fact the processes $\X_{k\delta}^{\e}$ and $\widehat{\Y}_{k\delta}^{\e}$ are $\mathscr{F}_{k\delta}$-measurable, whereas the process $\{\widetilde\Y^{\e,k\delta,\x,\y}_{s\e+k\delta}\}_{s\geq 0}$ is independent of $\mathscr{F}_{k\delta}$, for any fixed $(\x,\y)\in\H\times\H$. Moreover, from the definition of the process  $\widetilde\Y^{\e,k\delta,\x,\y}_{s\e+k\delta}$, we obtain 
	\begin{align}\label{3112}
&	\widetilde\Y^{\e,k\delta,\x,\y}_{s\e+k\delta}\nonumber\\&= \y-\frac{\mu}{\e}\int_{k\delta}^{s\e+k\delta}\A\widetilde{\Y}^{\e,k\delta,\x,\y}_r\d r-\frac{\beta}{\e}\int_{k\delta}^{s\e+k\delta}\mathcal{C}(\widetilde{\Y}^{\e,k\delta,\x,\y}_r)\d r+\frac{1}{\e}\int_{k\delta}^{s\e+k\delta}g(\x,\widetilde{\Y}^{\e,k\delta,\x,\y}_r)\d r\nonumber\\&\quad+\frac{1}{\sqrt{\e}}\int_{k\delta}^{s\e+k\delta}\sigma_2(\x,\widetilde{\Y}^{\e,k\delta,\x,\y}_r)\d\W^{\Q_2}_r\nonumber\\&=\y-\frac{\mu}{\e}\int_{0}^{s\e}\A\widetilde{\Y}^{\e,k\delta,\x,\y}_{r+k\delta}\d r-\frac{\beta}{\e}\int_{0}^{s\e}\mathcal{C}(\widetilde{\Y}^{\e,k\delta,\x,\y}_{r+k\delta})\d r+\frac{1}{\e}\int_{0}^{s\e}g(\x,\widetilde{\Y}^{\e,k\delta,\x,\y}_{r+k\delta})\d r\nonumber\\&\quad+\frac{1}{\sqrt{\e}}\int_{0}^{s\e}\sigma_2(\x,\widetilde{\Y}^{\e,k\delta,\x,\y}_{r+k\delta})\d\W^{\Q_2,k\delta}_r\nonumber\\&=\y-\mu\int_{0}^{s}\A\widetilde{\Y}^{\e,k\delta,\x,\y}_{r\e+k\delta}\d r-\beta\int_{0}^{s}\mathcal{C}(\widetilde{\Y}^{\e,k\delta,\x,\y}_{r\e+k\delta})\d r+\int_{0}^{s}g(\x,\widetilde{\Y}^{\e,k\delta,\x,\y}_{r\e+k\delta})\d r\nonumber\\&\quad+\int_{0}^{s}\sigma_2(\x,\widetilde{\Y}^{\e,k\delta,\x,\y}_{r\e+k\delta})\d\widehat\W^{\Q_2,k\delta}_r,
	\end{align}
	where $\{\W_r^{\Q_2,k\delta}:=\W_{r+k\delta}^{\Q_2}-\W_{k\delta}^{\Q_2}\}_{r\geq 0}$ and $\{\widehat\W^{\Q_2,k\delta}_r:=\frac{1}{\sqrt{\e}}\W_{r\e}^{\Q_2,k\delta}\}_{r\geq 0}$. It should be noted that $\widehat\W^{\Q_2,k\delta}_r$ is still a $\Q_2$-Wiener process in $\H$. Using the uniqueness of the strong solutions of \eqref{3.74} and \eqref{3112}, we infer that 
	\begin{align}
\mathscr{L}\left(\{\widetilde\Y^{\e,k\delta,\x,\y}_{s\e+k\delta}\}_{0\leq s\leq\frac{\delta}{\e}}\right)=\mathscr{L}\left(\{\Y^{\x,\y}_{s}\}_{0\leq s\leq\frac{\delta}{\e}}\right),
	\end{align}
	where $\mathscr{L}(\cdot)$ denotes the law of the distribution. Using Markov's property, Proposition \ref{prop3.12}, the estimates \eqref{3.7} and \eqref{341}, we estimate $\Phi_{k}(\cdot,\cdot)$ as 
		\begin{align}\label{3114}
	\Phi_k(s,r)&=\int_{\Omega}\wi\E\left[(f(\X_{k\delta}^{\e}(\omega),\Y_{s}^{\X_{k\delta}^{\e}(\omega),\widehat\Y_{k\delta}^{\e}(\omega)})-\bar f(\X_{k\delta}^{\e}(\omega)),\right.\nonumber\\&\qquad\left.f(\X_{k\delta}^{\e}(\omega),\Y_{r}^{\X_{k\delta}^{\e}(\omega),\widehat\Y_{k\delta}^{\e}(\omega)})-\bar f(\X_{k\delta}^{\e}(\omega)))\right]\mathbb{P}(\d\omega)\nonumber\\&=\int_{\Omega}\int_{\wi{\Omega}}\Big(\wi\E\Big[f(\X_{k\delta}^{\e}(\omega),\Y_{s-r}^{\X_{k\delta}^{\e}(\omega),\z})-\bar f(\X_{k\delta}^{\e}(\omega))\Big],\nonumber\\&\qquad f(\X_{k\delta}^{\e}(\omega),\z)-\bar f(\X_{k\delta}^{\e}(\omega))\Big)\Big|_{\z=\Y_{r}^{\X_{k\delta}^{\e}(\omega),\widehat\Y_{k\delta}^{\e}(\omega)}(\wi\omega)}\wi{\mathbb{P}}(\d\wi{\omega})\mathbb{P}(\d\omega)\nonumber\\&\leq \int_{\Omega}\int_{\wi{\Omega}}\left\|\wi\E\Big[f(\X_{k\delta}^{\e}(\omega),\Y_{s-r}^{\X_{k\delta}^{\e}(\omega),\z})-\bar f(\X_{k\delta}^{\e}(\omega))\Big]\right\|_{\H}\nonumber\\&\qquad\times \|f(\X_{k\delta}^{\e}(\omega),\z)-\bar f(\X_{k\delta}^{\e}(\omega))\|_{\H}\Big|_{\z=\Y_{r}^{\X_{k\delta}^{\e}(\omega),\widehat\Y_{k\delta}^{\e}(\omega)}(\wi\omega)}\wi{\mathbb{P}}(\d\wi{\omega})\mathbb{P}(\d\omega)\nonumber\\&\leq C_{\mu,\lambda_1,L_g}\int_{\Omega}\int_{\wi{\Omega}}\left[1+\|\X_{k\delta}^{\e}(\omega)\|_{\H}+\|\Y_{r}^{\X_{k\delta}^{\e}(\omega),\widehat\Y_{k\delta}^{\e}(\omega)}(\wi\omega)\|_{\H}\right]e^{-\frac{(s-r)\zeta}{2}}\nonumber\\&\qquad \times\left[1+\|\X_{k\delta}^{\e}(\omega)\|_{\H}+\|\Y_{r}^{\X_{k\delta}^{\e}(\omega),\widehat\Y_{k\delta}^{\e}(\omega)}(\wi\omega)\|_{\H}\right]\wi{\mathbb{P}}(\d\wi{\omega})\mathbb{P}(\d\omega)\nonumber\\&\leq C_{\mu,\lambda_1,L_g}\int_{\Omega}\left(1+\|\X_{k\delta}^{\e}(\omega)\|_{\H}^2+\|\widehat\Y_{k\delta}^{\e}(\omega)\|_{\H}^2\right)\mathbb{P}(\d\omega)e^{-\frac{(s-r)\zeta}{2}}\nonumber\\&\leq C_{\mu,\lambda_1,L_g,T}(1+\|\x\|_{\H}^2+\|\y\|_{\H}^2)e^{-\frac{(s-r)\zeta}{2}}. 
	\end{align}
	Using \eqref{3114} in \eqref{3107}, we obtain 
	\begin{align}\label{3115}
		\E\left[\sup\limits_{t\in[0,T]}|I_1(t)|\right]&\leq\frac{C_{\mu,\lambda_1,L_g,T}(1+\|\x\|_{\H}^2+\|\y\|_{\H}^2)\e}{\delta}\left[\int_0^{\frac{\delta}{\e}}\int_r^{\frac{\delta}{\e}}e^{-\frac{(s-r)\zeta}{2}}\d s\d r\right]^{1/2}\nonumber\\&={C_{\mu,\lambda_1,L_g,T}(1+\|\x\|_{\H}^2+\|\y\|_{\H}^2)}\frac{\e}{\delta}\left(\frac{2}{\zeta}\right)^{1/2}\left[\frac{\delta}{\e}+\frac{2}{\zeta}\left(e^{-\frac{\delta\zeta}{2\e}}-1\right)\right]^{1/2}\nonumber\\&\leq C_{\mu,\lambda_1,L_g,L_{\sigma_2},T}(1+\|\x\|_{\H}^2+\|\y\|_{\H}^2)\left[\left(\frac{\e}{\delta}\right)^{1/2}+\frac{\e}{\delta}\right]. 
	\end{align}
	Combining the estimates \eqref{3105}-\eqref{3115}, we get 
		\begin{align}\label{3116}
	\E\left[\sup\limits_{t\in[0,T\wedge\tau_R^{\e}]}|I(t)|\right]&\leq C_{\mu,\lambda_1,L_g,L_{\sigma_2},T}(1+\|\x\|_{\H}^2+\|\y\|_{\H}^2)\left[\left(\frac{\e}{\delta}\right)^{1/2}+\frac{\e}{\delta}+\delta^{1/2}\right]. 
	\end{align}
	Using \eqref{3116} in \eqref{3104}, we find 
	\begin{align}\label{3z25}
		\E\left[	\sup_{t\in[0,T\wedge\tau_R^{\e}]}	\|\Z^{\e}_{t}\|_{\H}^2\right]&\leq C_{R,\mu,\lambda_1,L_g,L_{\sigma_2},T}(1+\|\x\|_{\H}^3+\|\y\|_{\H}^3)\delta^{1/4}\nonumber\\&\quad+C_{\mu,\lambda_1,L_g,L_{\sigma_2},T}(1+\|\x\|_{\H}^2+\|\y\|_{\H}^2)\left[\left(\frac{\e}{\delta}\right)^{1/2}+\frac{\e}{\delta}+\delta^{1/2}\right]\nonumber\\&\leq C_{R,\mu,\lambda_1,L_g,L_{\sigma_2},T}(1+\|\x\|_{\H}^3+\|\y\|_{\H}^3)\left[\left(\frac{\e}{\delta}\right)^{1/2}+\frac{\e}{\delta}+\delta^{1/4}\right], 
	\end{align}
	and the proof is complete for $n=2$ and $r\in[1,3]$.

		Let us now discuss the case $n=3$ and $r\in(3,\infty)$. We just need to estimate the term $2|\langle(\B(\X_s^{\e})-\B(\bar\X_s)),\Z^{\e}_s\rangle|$ only to get the required result. 	From the estimate \eqref{2.23}, we easily have 
	\begin{align}\label{2z27}
	-2	\beta	\langle\mathcal{C}(\X_s^{\e})-\mathcal{C}(\bar\X_s),\Z_s^{\e}\rangle \leq- \beta\||\X_s^{\e}|^{\frac{r-1}{2}}\Z_s^{\e}\|_{\H}^2. 
	\end{align}
	Using H\"older's and Young's inequalities, we estimate the term  $2|\langle(\B(\X_s^{\e})-\B(\bar\X_s)),\Z^{\e}_s\rangle|=2|\langle\B(\Z_s^{\e},\X_s^{\e}),\Z_s^{\e}\rangle |$ as  
	\begin{align}\label{2z28}
	2|\langle\B(\Z_s^{\e},\X_s^{\e}),\Z_s^{\e}\rangle |&\leq 2\|\Z_s^{\e}\|_{\V}\|\X_s^{\e}\Z_s^{\e}\|_{\H}\leq\mu\|\Z_s^{\e}\|_{\V}^2+\frac{1}{\mu }\|\X_s^{\e}\Z_s^{\e}\|_{\H}^2.
	\end{align}
	We take the term $\|\X_s^{\e}\Z_s^{\e}\|_{\H}^2$ from \eqref{2z28} and perform a similar calculation in \eqref{2a29} to get 
	\begin{align}\label{2z29}
	\|\X_s^{\e}\Z_s^{\e}\|_{\H}^2\leq{\beta\mu }\left(\int_{\mathcal{O}}|\X_s^{\e}(x)|^{r-1}|\Z_s^{\e}(x)|^2\d x\right)+\frac{r-3}{r-1}\left(\frac{2}{\beta\mu (r-1)}\right)^{\frac{2}{r-3}}\left(\int_{\mathcal{O}}|\Z_s^{\e}(x)|^2\d x\right),
	\end{align}
	for $r>3$. Combining \eqref{2z27}, \eqref{2z28}  and \eqref{2z29}, we obtain 
	\begin{align}\label{2z30}
	&-2	\beta	\langle\mathcal{C}(\X_s^{\e})-\mathcal{C}(\bar\X_s),\Z_s^{\e}\rangle-2\langle\B(\Z_s^{\e},\X_s^{\e}),\Z_s^{\e}\rangle\leq \frac{r-3}{\mu(r-1)}\left(\frac{2}{\beta\mu (r-1)}\right)^{\frac{2}{r-3}}\|\Z^{\e}_s\|_{\H}^2.
	\end{align}
	Thus a calculation similar to the estimate \eqref{3.99} yields 
	\begin{align}
&	\|\Z^{\e}_{t}\|_{\H}^2+\mu\int_0^t\|\Z^{\e}_s\|_{\V}^2\d s\nonumber\\&\leq \left[\frac{r-3}{\mu(r-1)}\left(\frac{2}{\beta\mu (r-1)}\right)^{\frac{2}{r-3}}+C_{\mu,\lambda_1,L_g,L_{\sigma_2}}\right]\int_0^t\|\Z^{\e}_s\|_{\H}^2\d s\nonumber\\&\quad+C_{\mu,\lambda_1,L_g,L_{\sigma_2}}\int_0^t\|\X_{s}^{\e}-\X_{s(\delta)}^{\e}\|_{\H}^2\d s+C_{\mu,\lambda_1,L_g,L_{\sigma_2}}\int_0^t\|\Y_s^{\e}-\widehat\Y_s^{\e}\|_{\H}^2\d s\nonumber\\&\quad+C\left(\int_0^t(1+\|\X_{s(\delta)}^{\e}\|_{\H}^2+\|\widehat\Y_s^{\e})\|_{\H}^2)\d s\right)^{1/2}\left(\int_0^t(\|\X^{\e}_{s}-\X^{\e}_{s(\delta)}\|_{\H}^2+\|\bar\X_s-\bar\X_{s(\delta)}\|_{\H}^2)\d s\right)^{1/2} \nonumber\\&\quad +2\int_0^t(f(\X_{s(\delta)}^{\e},\widehat\Y_s^{\e})-\bar f(\X_{s(\delta)}^{\e}),\Z^{\e}_{s(\delta)})\d s+2\int_0^t(\sigma_1(\X_t^{\e})-\sigma_1(\bar\X_t)\d\W^{\Q_1}_s,\Z^{\e}_s), 
	\end{align}
for all $t\in[0,T]$, $\mathbb{P}$-a.s. Applying Gronwall's inequality and then using a calculation similar to \eqref{3z25} yields	 the estimate \eqref{390}.

	For $r=3$, 	from \eqref{2.23}, we have 
	\begin{align}\label{2z31}
	-2	\beta	\langle\mathcal{C}(\X_s^{\e})-\mathcal{C}(\bar\X_s),\Z_s^{\e}\rangle \leq- \beta\|\X_s^{\e}\Z_s^{\e}\|_{\H}^2, 
	\end{align}
	and a calculation similar to \eqref{2z28} gives 
	\begin{align}\label{3z72}
	2|\langle\B(\Z_s^{\e},\X_s^{\e}),\Z_s^{\e}\rangle |&\leq 2\|\Z_s^{\e}\|_{\V}\|\X_s^{\e}\Z_s^{\e}\|_{\H}\leq 2\mu\|\Z_s^{\e}\|_{\V}^2+\frac{1}{2\mu }\|\X_s^{\e}\Z_s^{\e}\|_{\H}^2.
	\end{align}
	Combining \eqref{2z31} and \eqref{3z72}, we obtain 
	\begin{align}\label{3125}
	&-2	\beta	\langle\mathcal{C}(\X_s^{\e})-\mathcal{C}(\bar\X_s),\Z_s^{\e}\rangle-2\langle\B(\Z_s^{\e},\X_s^{\e}),\Z_s^{\e}\rangle\leq2\mu\|\Z_s^{\e}\|_{\V}^2 -\left(\beta-\frac{1}{2\mu }\right)\|\X_s^{\e}\Z_s^{\e}\|_{\H}^2,
	\end{align}
	and the hence estimate \eqref{390} follows for  $2\beta\mu\geq 1$. Note that by making use of the estimates \eqref{2z30} and \eqref{3125}, we don't need any stopping time arguments to get the required estimate \eqref{390}. 
\end{proof}
Let us now establish our main result. 
\begin{proof}[Proof of Theorem \ref{maint}]
	We first prove the Theorem for $n=2$ and $r\in[1,3]$. Let $\tau_R^{\e}$ be the stopping time defined in \eqref{stop}. Then, we have 
	\begin{align}\label{3130}
	 \E\left[\sup_{t\in[0,T]}\|\X^{\e}_t-\bar\X_t\|_{\H}^2\right]= \E\left[\sup_{t\in[0,T]}\|\X^{\e}_t-\bar\X_t\|_{\H}^2\chi_{\{T\leq\tau_R^{\e}\}}\right]+\E\left[\sup_{t\in[0,T]}\|\X^{\e}_t-\bar\X_t\|_{\H}^2\chi_{\{T>\tau_R^{\e}\}}\right],
	\end{align}
where $\chi_{t}$ is the indicator function.	From Lemma \ref{lem3.14} (see \eqref{389}), choosing $\delta=\e^{2/3}$, we obtain 
	\begin{align}\label{3131}
\E\left[\sup_{t\in[0,T]}\|\X_t^{\e}-\bar{\X}_t\|_{\H}^2\chi_{\{T\leq\tau_R^{\e}\}}\right]\leq C_{R,\mu,\lambda_1,L_g,L_{\sigma_2},T}\left(1+\|\x\|_{\H}^3+\|\y\|_{\H}^3\right)\left(\e^{1/3}+\e^{1/6}\right).
	\end{align}	
Using H\"older's and Chebyshev’s inequalities,  \eqref{3.7} and \eqref{3.86}, we estimate the second term from the right hand side of the inequality \eqref{3130} as
\begin{align}\label{3132}
\E\left[\sup_{t\in[0,T]}\|\X^{\e}_t-\bar\X_t\|_{\H}^2\chi_{\{T>\tau_R^{\e}\}}\right]&\leq \left[\E\left(\sup_{t\in[0,T]}\|\X^{\e}_t-\bar\X_t\|_{\H}^4\right)\right]^{1/2}[\mathbb{P}(T>\tau_R^{\e})]^{1/2}\nonumber\\&\leq\frac{C_{\mu,\lambda_1,L_g,T}(1+\|\x\|_{\H}^2+\|\y\|_{\H}^2)}{\sqrt{R}}\left[\E\left(\int_0^T\|\X^{\e}_s\|_{\V}^2\d s\right)\right]^{1/2}\nonumber\\&\leq \frac{C_{\mu,\lambda_1,L_g,T}(1+\|\x\|_{\H}^3+\|\y\|_{\H}^3)}{\sqrt{R}}. 
\end{align} 
Combining \eqref{3131}-\eqref{3132} and substitute it in \eqref{3130}, we find 
\begin{align}\label{3133}
& \E\left[\sup_{t\in[0,T]}\|\X^{\e}_t-\bar\X_t\|_{\H}^2\right]\nonumber\\&\leq C_{R,\mu,\lambda_1,L_g,L_{\sigma_2},T}\left(1+\|\x\|_{\H}^3+\|\y\|_{\H}^3\right)\left(\e^{1/3}+\e^{1/6}\right)+\frac{C_{\mu,\lambda_1,L_g,T}(1+\|\x\|_{\H}^3+\|\y\|_{\H}^3)}{\sqrt{R}}.
\end{align}
Letting $\e\to 0$ and then $R\to\infty$ in \eqref{3133}, we obtain
	\begin{align}\label{3134}
\lim_{\e\to 0} \E\left(\sup_{t\in[0,T]}\|\X^{\e}_t-\bar\X_t\|_{\H}^{2}\right)=0.
\end{align}
Using H\"older's inequality, \eqref{3134}, \eqref{3.7} and \eqref{3.86}, we finally get 
\begin{align}
&\lim_{\e\to 0} \E\left(\sup_{t\in[0,T]}\|\X^{\e}_t-\bar\X_t\|_{\H}^{2p}\right)\nonumber\\&=\lim_{\e\to 0} \E\left(\sup_{t\in[0,T]}\|\X^{\e}_t-\bar\X_t\|_{\H}\|\X^{\e}_t-\bar\X_t\|_{\H}^{2p-1}\right)\nonumber\\&\leq \lim_{\e\to 0}\left\{ \left[\E\left(\sup_{t\in[0,T]}\|\X^{\e}_t-\bar\X_t\|_{\H}^2\right)\right]^{1/2}\left[\E\left(\sup_{t\in[0,T]}\|\X^{\e}_t-\bar\X_t\|_{\H}^{4p-2}\right)\right]^{1/2}\right\}=0, 
\end{align}
which completes the proof. 

For $n=2$, $r\in(3,\infty)$ and $n=3$, $r\in[3,\infty)$ ($2\beta\mu\geq 1$ for $r=3$), we don't need the stopping time arguments and the convergence \eqref{3134} can be easily obtained from the estimate \eqref{390} (see \eqref{3131} also).  
\end{proof}
\noindent \textbf{Conflict of interest.} On behalf of all authors, the corresponding author states that there is no conflict of interest. 

 \medskip\noindent
{\bf Acknowledgments:} M. T. Mohan would  like to thank the Department of Science and Technology (DST), India for Innovation in Science Pursuit for Inspired Research (INSPIRE) Faculty Award (IFA17-MA110).


\begin{thebibliography}{99}
	
\bibitem{SNA}	S.N. Antontsev and H.B. de Oliveira, The Navier–Stokes problem modified by an absorption term, \emph{Applicable Analysis}, {\bf 89}(12),  2010, 1805--1825. 
	
	
	
	
	
	
	
	

\bibitem{JBGY}	 J. Bao, G. Yin and C. Yuan, Two-time-scale stochastic partial differential equations driven by $\alpha$-stable noises: averaging principles, \emph{Bernoulli}, {\bf 23} (1), 645--669 (2017). 
	
	\bibitem{VB} V. Barbu, {\it Analysis and control of nonlinear infinite dimensional	systems}, Academic Press, Boston, 1993.
	
	
	
	
	

	\bibitem{RBJE} R. Bertram and J.E. Rubin, \emph{Multi-timescale systems and fast-slow analysis}, Math. Biosci. 287 (2017) 105--121.
	
	
	\bibitem{HBAM}		H. Bessaih and A. Millet,	On stochastic modified 3D Navier–Stokes equations with anisotropic viscosity, \emph{Journal of Mathematical Analysis and Applications}, 462 (2018), 915--956. 
	
	
	\bibitem{NNYA}	N.N. Bogoliubov, Y.A. Mitropolsky, Asymptotic Methods in the Theory of Non-linear Oscillations, \emph{Gordon	and Breach Science Publishers}, New York (1961).
	
	



\bibitem{CEB} C.E. Br\'ehier, Strong and weak orders in averaging for SPDEs, \emph{Stochastic Process. Appl.}, {\bf  122} (2012), 2553--2593.

\bibitem{CEB1} C.E. Br\'ehier, Orders of convergence in the averaging principle for SPDEs: the case of a stochastically forced slow component, \emph{Stochastic Process. Appl.} {\bf  130}(6) (2020), 3325--3368. 

\bibitem{ZBGD}  Z. Brze\'zniak and Gaurav Dhariwal, Stochastic tamed Navier-Stokes equations on $\mathbb{R}^3$: the existence and the uniqueness of solutions and the existence of an invariant measure,  https://arxiv.org/pdf/1904.13295.pdf. 

	
	

	\bibitem{DLB}	D. L. Burkholder,	The best constant in the Davis inequality for the expectation of the martingale square function, \emph{Transactions of the	American Mathematical Society} {\bf 354} (1), 91--105.
	
	
	
	
	
\bibitem{SC} 	S. Cerrai, A Khasminskii type averaging principle for stochastic reaction-diffusion equations, \emph{Ann. Appl.	Probab.}, {\bf 19} (2009), 899--948.

\bibitem{SC1}  S. Cerrai, Averaging principle for systems of reaction-diffusion equations with polynomial nonlinearities
perturbed by multiplicative noise, \emph{SIAM J. Math. Anal.} {\bf  43} (2011) 2482--2518.

\bibitem{SC2} S. Cerrai, M. Freidlin, Averaging principle for stochastic reaction-diffusion equations, \emph{Probab.Theory Related Fields}, {\bf 144} (2009), 137--177.

\bibitem{SC3} S. Cerrai, A. Lunardi, Averaging principle for nonautonomous slow-fast systems of stochastic reaction-
diffusion equations: the almost periodic case, \emph{SIAM J. Math. Anal.} 49 (2017) 2843-2884.
	

\bibitem{YCYS} Y. Chen,  Y. Shi and X. Sun, Averaging principle for slow-fast stochastic Burgers equation driven by $\alpha$-stable process, \emph{Appl. Math. Lett.}, {\bf 103} (2020), 106199 (9 pages). 
	
	
		\bibitem {CHKH} P.-L. Chow and R. Khasminskii,  Stationary solutions of nonlinear stochastic evolution equations,	\emph{Stochastic Analysis and Applications}, {\bf 15} (1997),	671--699.
	
		
		
	
	
	
	
	\bibitem {ICAM} I. Chueshov and A. Millet,	Stochastic 2D hydrodynamical type systems: Well posedness and Large Deviations, \emph{Applied Mathematics and  Optimization}, {\bf 61} (2010), 379--420.
	
	
	
	
	
	
	
	
	\bibitem{DaZ}
	\newblock G. Da Prato and J. Zabczyk,
	\newblock \emph{Stochastic Equations in Infinite Dimensions},
	\newblock Cambridge University Press, 1992.
	
		\bibitem{GDJZ}
	\newblock G. Da Prato and J. Zabczyk,
	\newblock \emph{Ergodicity for Infinite Dimensional Systems},
	\newblock London Mathematical Society Lecture Notes, {\bf 229}, Cambridge
	University Press, 1996.
	
		\bibitem {BD}	B. Davis, On the integrability of the martingale square function, \emph{Israel Journal of Mathematics} {\bf 8}(2) (1970), 187--190.
		
		
	
	
		\bibitem{ADe}
	\newblock A. Debussche,
	\newblock Ergodicity results for the stochastic Navier-Stokes equations: An introduction,
	\newblock Topics in Mathematical Fluid Mechanics, Volume {\bf 2073} of the series Lecture Notes in Mathematics, Springer,	23--108, 2013.
	
\bibitem{ZDXS}	Z. Dong, X. Sun, H. Xiao, J. Zhai, Averaging principle for one dimensional stochastic Burgers equation,	\emph{J. Differential Equations}, {\bf 265} (2018), 4749-4797.
	
\bibitem{WEBE}	W. E, B. Engquist, Multiscale modeling and computations, \emph{Notice of AMS}, {\bf 50} (2003) 1062-1070.
	
\bibitem{ZDRZ} Z.  Dong and R.  Zhang,	3D tamed Navier-Stokes equations driven by multiplicative L\'evy noise: Existence, uniqueness and large deviations, https://arxiv.org/pdf/1810.08868.pdf
	
	
	
	
	
	
	
	
	
	
	
	
	
	\bibitem{CLF} 	C. L. Fefferman, K. W. Hajduk and J. C. Robinson,	\emph{Simultaneous approximation in Lebesgue and Sobolev norms via eigenspaces}, https://arxiv.org/abs/1904.03337.

	\bibitem{FFBM}  F. Flandoli and B.  Maslowski, Ergodicity of the 2-D Navier-Stokes equation under random perturbations, \emph{Communications in Mathematical Physics}, {\bf 172} (1995), 119--141. 
	
\bibitem{HFJD}	H. Fu and J. Duan,  An averaging principle for two-scale stochastic partial differential equations, \emph{Stoch. Dyn.},	{\bf 11}(2011), 353--367. 


	
		\bibitem{HFJL}	H. Fu and J. Liu, Strong convergence in stochastic averaging principle for two time-scales stochastic partial	differential equations, \emph{J. Math. Anal. Appl.} 384 (2011) 70-86.
	
	\bibitem{HFLW}	 H. Fu, L. Wan and J. Liu, Strong convergence in averaging principle for stochastic hyperbolic-parabolic	equations with two time-scales, \emph{Stochastic Process. Appl.}, {\bf  125} (2015), 3255--3279.
	
		\bibitem{HFLW2}	H. Fu,  L. Wan, Y. Wang and J. Liu, Strong convergence rate in averaging principle for stochastic FitzHugh-Nagumo system with two time-scales, \emph{J. Math. Anal. Appl.}, {\bf 416}(2) (2014), 609--628. 
	

	
	\bibitem{HFLW1}	H. Fu, L. Wan, Y. Wang and J. Liu, Strong convergence rate in averaging principle for stochastic FitzHugh-Nagumo system with two time-scales, \emph{J. Math. Anal. Appl.}, {\bf  416}, (2014) 609--628.
	
		\bibitem{GGP} G. P. Galdi,  An introduction to the Navier–Stokes initial-boundary value problem. pp. 11-70 in \emph{Fundamental directions in mathematical fluid mechanics}, Adv. Math. Fluid Mech. Birkha\"user, Basel 2000.
		
	
\bibitem{PGa}	 P. Gao,  Averaging principle for stochastic Korteweg-de Vries equation, \emph{J. Differential Equations}, {\bf 267}(12) (2019),  6872--6909. 

\bibitem{PGa1} P. Gao, Averaging principle for the higher order nonlinear Schr\"odinger equation with a random fast oscillation, \emph{J. Stat. Phys.}, {\bf 171}(5) (2018), 897--926. 

\bibitem{PGa2} P. Gao, Averaging principle for stochastic Kuramoto-Sivashinsky equation with a fast oscillation, \emph{Discrete Contin. Dyn. Syst. A}, {\bf 38}(11) (2018), 5649--5684.

\bibitem{PGa3}  P. Gao, Averaging principle for multiscale stochastic Klein-Gordon-Heat system, \emph{Journal of Nonlinear Science}, {\bf 29} (2019), 1701--1759. 
	
	\bibitem{DGIG} 	D. Givon, I.G. Kevrekidis, R. Kupferman, Strong convergence of projective integeration schemes for
	singularly perturbed stochastic differential systems, Comm. Math. Sci. 4 (2006) 707-729.
	
		\bibitem{GK1} I. Gy\"ngy and  N. V. Krylov,  On stochastic equations with respect to semimartingales II, It\^o formula in Banach spaces, Stochastics, {\bf 6}(3--4) (1982), 153--173. 
		
	
	
	
	

	\bibitem{MHJC} M. Hairer, J.C. Mattingly, Ergodicity of the 2D Navier-Stokes equations with degenerate stochastic
forcing, \emph{Annals of Mathematics}, {\bf 164} (2006), 993--1032.



	\bibitem{KWH}	K. W. Hajduk and J. C. Robinson, Energy equality for the 3D critical convective Brinkman-Forchheimer equations, \emph{Journal of Differential Equations}, {\bf 263} (2017), 7141--7161.
	

\bibitem{EHVK} E. Harvey, V. Kirk, M. Wechselberger and J. Sneyd, \emph{Multiple timescales, mixed mode oscillations and canards in models of intracellular calcium dynamics}, \emph{J. Nonlinear Sci.} {\bf  21} (2011) 639--683.
	
	
	

	

	
	
	
	
	
	
	
	
	\bibitem{KT2}  V. K. Kalantarov and S. Zelik, Smooth attractors for the Brinkman-Forchheimer equations with fast growing nonlinearities, \emph{Commun. Pure Appl. Anal.}, {\bf 11}	(2012) 2037--2054.
	

\bibitem{RZK}  R.Z. Khasminskii, On the principle of averaging the It\^o's stochastic differential equations, \emph{Kybernetica} {\bf 4} (1968), 260--279.
	
	
	\bibitem{OAL}	O. A. Ladyzhenskaya, \emph{The Mathematical Theory of Viscous Incompressible Flow}, Gordon and Breach, New York, 1969.
	

\bibitem{DLi} D. Liu, Strong convergence of principle of averaging for multiscale stochastic dynamical systems, \emph{Commun. Math. Sci.},  {\bf 8}(2010), 999--1020.

	\bibitem{SLXS} S. Li, X. Sun, Y.  Xie and Y. Zhao, Average principle for two dimensional stochastic Navier-Stokes equations, https://arxiv.org/pdf/1810.02282.pdf. 
	
	
		\bibitem{LHGH1}	H. Liu and H. Gao, Stochastic 3D Navier–Stokes equations with nonlinear damping: martingale solution, strong solution and small time LDP, Chapter 2 in \emph{Interdisciplinary Mathematical SciencesStochastic PDEs and Modelling of Multiscale Complex System}, 9--36, 2019.
		
			
	\bibitem{WLXN} 		W. Liu, X. Sun, Y. Xie, Averaging principle for a class of stochastic differential equations, \emph{https://128.84.21.199/abs/1809.01424v2}. 
	
	\bibitem{WLMRX} W. Liu, M.  R\"ockner, X. Sun, Y. Xie	Strong averaging principle for slow-fast stochastic partial differential equations with locally monotone coefficients, https://arxiv.org/abs/1907.03260. 
	
	
\bibitem{WLMR} W. Liu and M. R\"ockner,	Local and global well-posedness of SPDE with generalized coercivity conditions, \emph{Journal of Differential Equations}, {\bf 254} (2013), 725--755. 

\bibitem{WL}  W. Liu, Well-posedness of stochastic partial differential equations with Lyapunov condition, \emph{Journal of Differential Equations}, {\bf 255} (2013), 572--592. 


	

	
	

	\bibitem{CMMR} C. Marinelli and M. R\"ockner,	On the maximal inequalities of Burkholder, Davis and Gundy, \emph{Expositiones Mathematicae}, 	{\bf 34}(1)(2016), 1--26.
	
	\bibitem{EAMEL}	E.A. Mastny, E.L. Haseltine and J.B. Rawlings, Two classes of quasi-steady-state model reductions for	stochastic kinetics, \emph{J. Chem. Phys}, {\bf 127} (2007) 094106.



	
	
	\bibitem{MMMC} 	M. Mikikian, M. Cavarroc, L. Couedel, Y. Tessier, L. Boufendi, Mixed-mode oscillations in complex-
	plasma instabilities, \emph{Phys. Rev. Lett.} {\bf 100} (2008) 225005.
	
	\bibitem{Me} M. M\'{e}tivier, \emph{Stochastic partial differential equations in infinite dimensional	spaces}, Quaderni, Scuola Normale Superiore, Pisa, 1988.
		
	
	
	
	
	\bibitem{MTM7} M. T. Mohan, On the convective Brinkman-Forchheimer equations, \emph{Submitted},  https://arxiv.org/pdf/2007.09376.pdf.
	
	\bibitem{MTM8} M. T. Mohan, Stochastic convective Brinkman-Forchheimer equations, \emph{Submitted}.
	
	
		\bibitem{MTM6} M.T. Mohan, Well posedness, large deviations and ergodicity of the stochastic 2D Oldroyd model of order one, \emph{Stochastic Processes and their Applications}, {\bf 130}(8) (2020), 4513--4562. 

\bibitem{JCR3}	J. C. Robinson and W. Sadowski,	A local smoothness criterion for solutions of the 3D Navier-Stokes equations, \emph{Rendiconti del Seminario Matematico della Universit\'a di Padova} {\bf 131} (2014), 159--178.
	
		
		

		
		
		
			\bibitem{MRXZ1}	M. R\"ockner and X. Zhang, Stochastic tamed 3D Navier-Stokes equation: existence, uniqueness and ergodicity, \emph{Probability Theory and Related Fields}, {\bf 145} (2009) 211--267.
			
		\bibitem{MRTZ1}	M. R\"ockner, T. Zhang and X. Zhang,	Large deviations for stochastic tamed 3D Navier-Stokes equations, \emph{Applied Mathematics and Optimization}, {\bf 61} (2010), 267--285. 
		
			\bibitem{MRTZ}	 M. R\"ockner and T. Zhang, Stochastic 3D tamed Navier-Stokes equations: Existence, uniqueness and small time large deviations principles, J\emph{ournal of Differential Equations}, {\bf 252} (2012), 716--744.
	
	
	
	\bibitem{SSSP} S. S. Sritharan and P. Sundar,	{Large deviations for the two-dimensional Navier-Stokes	equations with multiplicative noise}, \emph{Stochastic Processes and their		Applications}, {\bf 116} (2006), 1636--1659.
	
	
	\bibitem{Te} R. Temam,  \emph{Navier-Stokes Equations, Theory and Numerical Analysis}, North-Holland, Amsterdam, 1984.
	

	
	
	
	\bibitem{Te1} R. Temam, 	\emph{Navier-Stokes Equations and Nonlinear Functional Analysis}, Second Edition, CBMS-NSF Regional Conference Series in Applied Mathematics, 1995.
	
	

\bibitem{AYV}  A. Y. Veretennikov, On an averaging principle for systems of stochastic differential equations,  \emph{Math. USSR-Sb.} {\bf  69} (1) (1991), 271--284.

\bibitem{WWAJ} W. Wang and A.J. Roberts, Average and deviation for slow-fast stochastic partial differential equations, \emph{J. Differential Equations}, {\bf 253} (2012), 1265--1286.


\bibitem{WWAJ1} W. Wang, A.J. Roberts and J. Duan, Large deviations and approximations for slow-fast stochastic reaction-diffusion equations, \emph{J. Differential Equations}, {\bf 253} (2012), 3501--3522.

	\bibitem{FWTT}  F. Wu, T. Tian, J.B. Rawlings, G. Yin, Approximate method for stochastic chemical kinetics with two-
time scales by chemical Langevin equations, \emph{J. Chem. Phys}, {\bf 144} (2016) 174112.

\bibitem{JXu} J. Xu, $L^p$-strong convergence of the averaging principle for slow-fast SPDEs with jumps, \emph{ J. Math. Anal. Appl.}, {\bf 445}(1) (2017), 342--373. 

\bibitem{JXJL} J. Xu, J. Liu, Y. Miao, Strong averaging principle for two-time-scale SDEs with non-Lipschitz coefficients, \emph{J. Math. Anal. Appl.}, {\bf 468} (2018), 116--140.

\bibitem{YXBP} Y. Xu, B. Pei, J.-L. Wu, Stochastic averaging principle for differential equations with non-Lipschitz coefficients driven by fractional Brownian motion, \emph{Stoch. Dyn.}, {\bf  17} (2017), 1750013.

\bibitem{JXUM} J. Xu, Y. Miao and J. Liu, Strong averaging principle for two-time-scale non-autonomous stochastic FitzHugh–Nagumo system with jumps, \emph{J. Math. Phys.}, {\bf 57}(9) (2016), 092704. 
	

	
	
	
	
	
	
	
	

	
\end{thebibliography}
\end{document}